\documentclass{article}  
\usepackage{mathtools}
\usepackage[%pdftex,
bookmarksnumbered,
bookmarksopen,
colorlinks,
citecolor=blue,
linkcolor=blue,]{hyperref}
\usepackage{amsmath}
\usepackage{amssymb}
\usepackage{mathrsfs}
\usepackage{amsfonts}
\usepackage{amsthm}
\usepackage{algorithm}
\usepackage{booktabs}
\usepackage{listings}
\usepackage{boxedminipage}
\usepackage{algorithmic}
\usepackage{stmaryrd}
\usepackage{cite}
\usepackage{relsize}
\usepackage{lscape}
\usepackage[top=1.3in, bottom=1.5in, left=1.5in, right=1.5in]{geometry}
\usepackage{stmaryrd} 
\usepackage{relsize}%
\usepackage{url}
\usepackage{color,xcolor}
\usepackage{epsfig}
\usepackage{graphicx}
\usepackage{subfigure}

\newtheorem{theorem}{Theorem}[section]
\newtheorem{definition}{Definition}[section]
\newtheorem{proposition}{Proposition}[section]

\newtheorem{lemma}{Lemma}[section]

\newtheorem{remark}{Remark}[section]

\newenvironment{mytabular}{\bgroup\tiny\tabular}{\endtabular\egroup}

	  \newcommand\MC[1]{\ensuremath{\mathcal{#1}  }} 
    \newcommand{\tprod}{\ensuremath{  *     }}

    \newcommand{\bcirc}[1]{\ensuremath{ {\rm bcirc}\left(  #1 \right)  }}

            \newcommand{\fourieri}[1]{\ensuremath{  \left(  F_{l}\otimes I_{#1}    \right)   }}
                        \newcommand{\fourieriH}[1]{\ensuremath{  \left( F_{l}^{H}\otimes I_{#1} \right)       }}
     \newcommand{\trace}[1]{\ensuremath{ {\rm tr}\left(   #1 \right)  }}
          \newcommand{\st}[1]{\ensuremath{ {\rm St}\left(   #1\right)  }}
                    \newcommand{\conj}[1]{\ensuremath{ {\rm conj}\left(   #1\right)  }}

               \newcommand\M[2]{\ensuremath{\mathbb R^{#1\times #2}   }} 
                 \newcommand\C[2]{\ensuremath{\mathbb C^{#1\times #2}   }}

             \newcommand\T[3]{\ensuremath{\mathbb R^{#1\times #2 \times #3}   }} 
               \newcommand\TU[3]{\ensuremath{\mathbb{R}_{upp}^{#1\times #2 \times #3}  }} 
                \newcommand\TUP[3]{\ensuremath{\mathbb{R}_{upp+}^{#1\times #2 \times #3}  }} 
                
                \newcommand\TCU[3]{\ensuremath{\mathbb{C}_{upp}^{#1\times #2 \times #3}  }} 
                \newcommand\TCUP[3]{\ensuremath{\mathbb{C}_{upp+}^{#1\times #2 \times #3}  }}

  \newcommand\TS[3]{\ensuremath{\operatorname{Sym}(\mathbb R^{#1\times #2 \times #3} )  }} 
\newcommand\TSP[3]{\ensuremath{\operatorname{Sym}(\mathbb R_+^{#1\times #2 \times #3} )  }} 
\newcommand\TSPP[3]{\ensuremath{\operatorname{Sym}(\mathbb R_{++}^{#1\times #2 \times #3})   }} 

 \newcommand\TSkew[3]{\ensuremath{\operatorname{Skew}(\mathbb R^{#1\times #2 \times #3} )  }} 
                 
               \newcommand\TC[3]{\ensuremath{\mathbb C^{#1\times #2 \times #3}   }}

    \newcommand{\bigxiaokuohao}[1]{\ensuremath{ \left(  #1 \right) }}      
    \newcommand{\bigjueduizhi}[1]{\ensuremath{ \left|  #1 \right| }}   
    \newcommand{\bigdakuohao}[1]{\ensuremath{ \left\{  #1 \right\} }}         
    \newcommand{\bigzhongkuohao}[1]{\ensuremath{ \left[   #1 \right] }}      
    
     \newcommand{\bigfnorm}[1]{\ensuremath{ \left\|   #1 \right\|_F }}    
                \newcommand{\bignorm}[1]{\ensuremath{ \left\|   #1 \right\|  }}

    \newcommand{\innerprod}[2]{\ensuremath{ \left\langle   #1 , #2\right\rangle }}      
                     \newcommand{\R}{\mathbb{R}}
                      \newcommand{\ex}[1]{\ensuremath{ \operatorname{exp}\bigzhongkuohao{#1}}}       
          \newcommand{\Diag}[1]{\ensuremath{\operatorname{Diag}\left(#1: i \in[l]\right) }}      
           \newcommand{\Vector}[1]{\ensuremath{\operatorname{Vec}\left(#1: i \in[l]\right) }} 
             \newcommand{\Fold}[1]{\ensuremath{\operatorname{fold}\left(#1: i \in[l]\right) }}   

	\definecolor{darkgray}{rgb}{0.66, 0.66, 0.66}

\title{Computation over Tensor Stiefel Manifold: A Preliminary Study}

\author{Xianpeng Mao\thanks{School of Physical Science and Technology, Guangxi University, Nanning 530004, China},~ Ying Wang$^\dagger$, Yuning Yang\thanks{College of Mathematics and Information Science, Guangxi University, Nanning, 530004, China} \thanks{Corresponding author: Yuning Yang, yyang@gxu.edu.cn}                             }
\begin{document} %\large
\maketitle

%\includepdf{cover.pdf}

\begin{abstract}
\color{black}
Let $\tprod$ denote the t-product \cite{kilmer2011factorization} between two third-order tensors. The purpose of this work is to study    fundamental  computation  over the set $\st{n,p,l} := \{\MC{X}\in\T{n}{p}{l} \mid \MC{X}^{\top}\tprod \MC{X} = \mathcal I  \}$, where $\MC{X}$ is a third-order tensor of size $n\times p \times l$ ($n\geq p$) and $\mathcal I$  is the identity tensor.   It is first verified that $\st{n,p,l}$ endowed with the usual Frobenius norm forms  a Riemannian manifold, which is termed as the (third-order) \emph{tensor Stiefel manifold} in this work. We then derive the tangent space, Riemannian gradient, and Riemannian Hessian  on $\st{n,p,l}$.  In addition, formulas of various retractions based on t-QR, t-polar decomposition, t-Cayley transform, and t-exponential,  as well as vector transports, are presented.  It is expected that analogous to their matrix counterparts, the formulas derived in this study may serve as   building blocks for analyzing optimization problems over the tensor Stiefel manifold and     designing Riemannian algorithms %, such as   Riemannian gradient/conjugate gradient/(quasi-)Newton’s methods 
for them.
	
\end{abstract}

\hspace{2mm}\vspace{3mm}
\section{Introduction}
\color{black}
%this is test
%edit from xiamingKK
Higher-order tensors  play   important roles in linear and multilinear algebra, statistics, optimization, machine learning, and engineering \cite{kolda2010tensor,comon2014tensors,cichocki2015tensor,Sidiropoulos2016ten}. However, the notion of multiplication between tensors was unclear based on the traditional tensor computation; this   prevents the    extensions of several matrix  operations to higher-order tensors. Such a problem was addressed by Kilmer, Martin, Braman, and their coauthors, who proposed a type of multiplication, termed   the t-product, between third-order tensors \cite{braman2010third,kilmer2011factorization,kilmer2013third}. The t-product allows the possibility of   usual   notions and properties of matrices living in the tensor world.  For example, the authors of \cite{kilmer2011factorization,kilmer2013third} also defined notions such as  inverse tensors, orthogonal tensors, tensor transpose, and proposed t-SVD and t-QR decomposition.  \cite{lu2019tensor} proposed tensor spectral norm, nuclear norm in the sense of the t-product, and presented an efficient way for computing t-SVD; the authors used these tools to develop tensor robust PCA models. 
Based on the t-product, t‑Jordan canonical form and t‑Drazin inverse were generalized to third-order tensors  \cite{miao2021t}; the tensor t-functions were established in \cite{lund2020tensor,miao2020generalized}; the t-eigenvalues and related properties were studied in \cite{liu2020study}. Recently, the concepts of t-Hessian tensor, t-convexity, and   t-(semi)definiteness were defined in \cite{zheng2021t}. 
The t-SVD was further investigated in \cite{qi2021tubal}. Compared with other tensor (decomposition) models, the t-product based one allows us to deal with   tensors quite similar to their matrix counterparts; moreover, most basic operations can be efficiently implemented via FFT \cite{kilmer2011factorization}. 
\color{black}

Let $\tprod$ denote the t-product and   $\st{n,p,l}$   the set of partially orthogonal tensors:
\begin{align}\label{def:tensor_stiefel_manifold}
	\st{n,p,l} := \{\MC{X}\in\T{n}{p}{l} \mid \MC{X}^{\top}\tprod \MC{X} = \mathcal I, n\geq p  \},
\end{align}
where $\MC{X}$ is a third-order tensor of size $n\times p \times l$ ($n\geq p$% and $l$ arbitrary large
), $^{\top}$ denotes the transpose, and $\mathcal I$ represents the identity tensor that will be detailed later. Several (potential) tensor problems take the form:   
\begin{align}
	\label{prob:prototype}
	\min_{\MC{X}\in\T{n}{p}{l}}\nolimits ~f(\MC{X})~{\rm s.t.}~\MC{X} \in\st{n,p,l},
\end{align}
%where $f:\ \T{n}{p}{l} \rightarrow\mathbb R$, 
such as tensor approximation (with missing entries), joint diagonalization, joint t-SVD,   (sparse) tensor PCA, and beyond; these will be introduced in Sect. \ref{sec:appl}.  In fact,  when $l=1$,  \eqref{prob:prototype}  boils down to   optimization over the orthogonal matrix constraint, namely,   the Stiefel manifold,   which is a special Riemannian manifold. 
In recent years, Riemannian manifold optimization has drawn much attention; see, e.g., \cite{huang2018riemannian,hu2019structured,chen2020proximal,huang2021riemannian,gao2018new,hu2020brief}; fundamental concepts, tools, and algorithms can be found in \cite{absil2009optimization,tuIntroductionManifolds2011,boumal2022intromanifolds}. 
Classical methods     in the Euclidean
space, including the gradient descent/conjugate gradient/(quasi-)Newton’s method/trust region method, have been generalized to   Riemannian manifolds. 
It is known that the  orthogonal projection operator, Riemannian gradient and Riemannian Hessian, the retraction, and the vector transport are fundamental tools for Riemannian manifold optimization.

Riemannian structure and computation have also been studied in the context of tensors. For example, \cite{uschmajew2013geometry} investigated the geometry of  the hierarchical
Tucker format of tensors; \cite{holtz2012manifolds} considered the manifold of tensors of tensor-train (TT) format of fixed TT-rank; a Riemannian conjugate gradient was developed in \cite{kressner2014low} for tensor completion of Tucker format of fixed multilinear-rank; for the same task and the same format, \cite{heidel2018riemannian} proposed a Riemannian trust-region method, while in the TT format of fixed rank,  \cite{steinlechner2016riemannian}   proposed a Riemannian conjugate gradient;  such a method was also developed for the canonical polyadic approximation \cite{breiding2018riemannian}; just to name a few. On the other hand,  
in the t-product sense, \cite{gilman2022grassmannian} recently proposed a Grassmannian optimization based approach for online tensor completion and tracking, and  \cite{song2020riemannian} devised a Riemannian conjugate gradient descent over the manifold of fixed transformed multi-rank tensors for   tensor completion. 

\color{black}However, although optimization over the  (matrix) Stiefel manifold develops rapidly, over the set  $\st{n,p,l}$ in \eqref{def:tensor_stiefel_manifold}, it has not been studied yet. In view of the aforementioned progress on t-product based tensor theory and Riemannian optimization, as well as the real-world demand, this work intends to make a   study concerning $\st{n,p,l}$. % We do not consider specific algorithms; instead, we intend to derive some basic formulas.
Specifically, our progress is: \color{black}

1. We first show that $\st{n,p,l}$ endowed with the Frobenius norm forms a Riemannian manifold, which is termed as the (third-order) \emph{tensor Stiefel manifold} in this paper;

2. The parameterized form of the tangent space   of the tensor Stiefel manifold is established. Furthermore, the orthogonal projector operator  is studied, based on which we   deduce the  Riemannian gradient and Riemannian Hessian of an objective function over  $\st{n,p,l}$ from the  Euclidean gradient and Hessian of an extended objective function on the ambient
Euclidean space;

\color{black}
3. Several retractions based on different tensor decompositions, such as t-QR and t-polar decomposition, t-Cayley transform, and the geodesic based on t-exponential, which map points from the tangent space of $\st{n,p,l}$ to   $\st{n,p,l}$, are derived. The vector transports upon various retractions, which compare tangent vectors at distinct points on the manifold, are also obtained;\color{black}

4. As byproducts, we define skew-symmetric tensors, t-polar decomposition, and related properties; the analytical solution of the tensor Sylvester equation is derived.

Owing to the nice properties of the t-product, the derived formulas have similar forms as their matrix counterparts. It is expected that these formulas can be useful for analyzing optimization problems over the tensor Stiefel manifold and designing    Riemannian algorithms for them. In particular, as these formulas are consistent with their matrix counterparts,
the recently developed algorithms over the matrix Stiefel manifold, such as \cite{chen2020proximal,gao2018new,hu2019structured}, might be parallelly transplanted to the tensor setting without many modifications.

The rest of this work is organized as follows. In Sect. \ref{sec:pre_t_prod}, we summarize   preliminaries on t-product, t-exponential and tensor decompositions, such as t-polar decomposition and t-QR decomposition, 
which are used throughout this paper, while preliminaries on Riemannian manifold are introduced in Sect. \ref{sec:pre_rieman}. 
In Sect. \ref{sec:tensor_stiefel_manifold}, we study the tensor Stiefel manifold \eqref{def:tensor_stiefel_manifold}, the tangent space of \eqref{def:tensor_stiefel_manifold},  the orthogonal projector operator, various retractions and vector transports. %, which provide the most basic tools for unconstrained optimization algorithms on Riemannian manifolds. 
%In Section 5, we present several synthetic data and imaging data sets to demonstrate
%the performance of the proposed algorithms. 
In Sect. \ref{sec:appl}, several examples of \eqref{prob:prototype} and related optimization problems are presented.
In Sect. \ref{sec:experiments}, we conducted preliminary numerical experiments to verify the derived formulas.
Finally, some concluding remarks are given in Sect. \ref{sec:conclusion}.
\color{black}

\section{t-Product based Tensor Computation}
\label{sec:pre_t_prod}
%\subsection{Notation}
\paragraph{Notation.} Throughout this paper,  scalars are written as small letters $(a,b,\cdots)$, 
vectors are written as boldface lowercase letters $(\mathbf{x}, \mathbf{y},\cdots)$, 
matrices correspond to
italic capitals $(A,B,\cdots)$, tensors are written as calligraphic capitals $(\mathcal{A}, \mathcal{B},\cdots)$, and manifolds are written as Ralph Smith's formal script font $(\mathscr{A}, \mathscr{B},\cdots)$. $\mathbb{R}^{n\times p \times l} (\mathbb{C}^{n\times p \times l})$ denotes the
space of $n \times p \times l$ real (complex) tensors. The $(i,j,k)$ entry of $\mathcal{A}$ is denoted as
$a_{ijk}$.
For any positive integer $l$, denote
$[l] := \{1, 2,\cdots, l\}$. 
For a third-order tensor $\mathcal{A}\in \mathbb{R}^{n\times p \times l}$,  $A^{(i)} := \mathcal{A}(:,:,i)$, $i \in [l]$ represents each frontal slice, which is defined by fixing the third index  and varying the first two. 
%The $(i,k)$-th tube fiber in $\mathcal{A}$ is written as $\mathbf{a}_{ik} := \mathcal{A}(i, k, :)$. 
The inner product $\innerprod{\mathcal A}{\mathcal B}$ between two real tensors $\mathcal A$ and $\mathcal B$ of the same size is the sum of entry-wise product and the Frobenius norm $\bigfnorm{\mathcal A} = \innerprod{\mathcal A}{\mathcal A}^{1/2}$. 
%We use the subscript $p$ to emphasize that the dimension of the square matrix $A_p$ is $p\times p$.
% Specifically, 
$I_p(O_p)$ denotes a unit (zero) matrix of dimension $p\times p$.  $A^{\top}(A^H)$,   $\operatorname{conj(A)}$, and $A^{\dagger }$ represent  transpose, conjugate transpose,   conjugate, and   Moore-Penrose generalized inverse of the matrix $A$, respectively.    %represents  conjugate of the matrix $A$. denotes the  of matrix $A$.
% $\operatorname{diag}(x_1,\cdots,x_n)$ means the diagonal matrix with $(x_1,\cdots,x_n)$ as the diagonal elements. 
The diagonal of  $A\in\mathbb R^{n\times n}$ is defined as $\operatorname{diag}(A) \in\mathbb R^{n\times n}$, with all the non-diagonal entries of $A$ zeroed out. $\otimes$ means the Kronecker product of matrices.  $\odot$ represents the   Khatri-Rao product for partitioned matrices \cite{zhang2002inequalities}.
% The (skew-)symmetric part of the tensor $\MC{A}$ is defined by $\operatorname{sym}(\MC{A}) = \frac{\MC{A}+\MC{A}^{\top}}{2}(\operatorname{skew}(\MC{A}) = \frac{\MC{A}-\MC{A}^{\top}}{2})$. 
% and $\lfloor x \rfloor$
% as the one less than  or equal to $x$.

$\mathcal A$ is called f-square if $n=p$.
A tensor $\MC{A}\in \TC{n}{p}{l}$ with $n\geq p$ is called  “f-full rank $p$” if each frontal slice of $\mathcal A$ is of full rank $p$.  The sets of f-full rank $p$ complex tensors  are denoted as $\TC{n}{p}{l}_*$. A tensor $\MC{A}\in \T{n}{p}{l} (\TC{n}{p}{l})$ is called ``f-diagonal'' or ``f-upper triangular'', if each frontal slice of $\mathcal A$ is  diagonal or upper
triangular, respectively. The sets of f-upper triangular real tensors (with strictly positive diagonal elements) are denoted as $\TU{n}{n}{l}(\TUP{n}{n}{l})$. The sets of f-upper triangular complex tensors (with strictly positive diagonal elements) are denoted as $\TCU{n}{n}{l}(\TCUP{n}{n}{l})$.

% \begin{definition}\cite{kilmer2011factorization}\label{def:f-diag and f-upper tri}
	% A tensor $\MC{A}\in \T{n}{p}{l} (\TC{n}{p}{l})$ is called ``f-diagonal'' or ``f-upper triangular'', if each frontal slice of $\mathcal A$ is  diagonal or upper
	% triangular, respectively. The sets of f-upper triangular real tensors (with strictly positive diagonal elements) are denoted as $\TU{n}{n}{l}(\TUP{n}{n}{l})$. The sets of f-upper triangular complex tensors (with strictly positive diagonal elements) are denoted as $\TCU{n}{n}{l}(\TCUP{n}{n}{l})$.
	% \end{definition}
% We need the following definition concerning certain full rankness on the Fourier  domain.
% \begin{definition}\label{def:f-full rank p}
	% A tensor $\MC{A}\in \TC{n}{p}{l}$ with $n\geq p$ is called  “f-full rank $p$” if each frontal slice of $\mathcal A$ is of full rank $p$.  The sets of f-full rank $p$ complex tensors  are denoted as $\TC{n}{p}{l}_*$.
	% \end{definition}

\subsection{t-Product for Third-Order Tensors}

Before giving the definition of the t-product, some preparations are needed first. 
\begin{definition}\cite{kilmer2011factorization}
	The 	``$\operatorname{unfold}$'' command is anchored to the frontal slices of the tensor $\mathcal{A}\in \T{n}{p}{l}$, i.e.,
	\begin{equation*}\small
		\operatorname{unfold}(\mathcal{A}) := \begin{bmatrix}
			\begin{smallmatrix}
				A^{(1)}\\ 
				%			A^{(2)}\\ 
				\vdots \\ 
				A^{(l)}
			\end{smallmatrix}
		\end{bmatrix}\in \M{nl}{p}. 
		%	\color{red}:= [A^{(1)}, A^{(2)}, \cdots, A^{(n)}]^{\top}
	\end{equation*}
	%	where $A^{(i)} = \mathcal{A}(:,:,i)$ for $i \in [l]$.
	%\begin{remark}
	And	the operation takes ``$\operatorname{unfold}$'' back to tensor form is the $\operatorname{fold}$ command:
	$
	\operatorname{fold}(\operatorname{unfold}(\mathcal{A})) = \mathcal{A}.
	$ 
\end{definition}

\begin{definition}\cite{kilmer2011factorization}
	Let $\mathcal{A}\in \mathbb{R}^{n\times p \times l}$; then its circulant matrices is
	$$
	\operatorname{bcirc}(\mathcal{A}) 
	%		:= \operatorname{bcirc}(	\operatorname{unfold}(\mathcal{A})) 
	:= \begin{bmatrix}
		\begin{smallmatrix}
			A^{(1)} & A^{(l)}   & \cdots  & A^{(2)}\\ 
			A^{(2)} & A^{(1)} & \cdots & A^{(3)}\\ 
			\vdots  & \ddots  & \ddots & \vdots \\ 
			A^{(l)}  & \cdots & A^{(2)} & A^{(1)}
		\end{smallmatrix}
	\end{bmatrix}
	\in\mathbb{R}^{nl\times pl}.
	$$
	%where $A^{(i)} = \mathcal{A}(:,:,i)$ for $i \in [l]$.
\end{definition}

\begin{definition}\cite{kilmer2011factorization}\label{def:t-product}
	The t-product 
	between  $\mathcal{A}\in \mathbb{R}^{n\times p \times l}$ and $\mathcal{B}\in \mathbb{R}^{p\times m \times l}$ is defined as
	\begin{equation}\label{eq:t_prod_formular}
		\mathcal{A}\ast\mathcal{B} := \operatorname{fold}\bigxiaokuohao{\operatorname{bcirc}(\mathcal{A})\cdot \operatorname{unfold}(\mathcal{B})}
		%		= \operatorname{bcirc}^{-1}(\operatorname{bcirc}(\mathcal{A})\cdot \operatorname{bcirc}(\mathcal{B}))
		\in \mathbb{R}^{n\times m \times l}.
	\end{equation}
	% 	namely, their last dimension should be coincide and the second dimension of $\mathcal A$ and the first one of $\mathcal B$ should be the same.
\end{definition}

\begin{definition}\cite{kilmer2011factorization}\label{def:tranpose}
	Let $\mathcal{A}\in \mathbb{R}^{n\times p \times l}$; its transpose tensor $\mathcal{A^{\top}}$ is defined as
	\begin{small}
		\begin{equation*}
			\mathcal{A^{\top}}:=\operatorname{fold}
			\begin{bmatrix}
				\begin{smallmatrix}
					(A^{(1)})^{\top}\\ 
					(A^{(l)})^{\top}\\
					\vdots\\ 
					(A^{(2)})^{\top}
				\end{smallmatrix}
			\end{bmatrix}\in \mathbb{R}^{p\times n \times l}.
		\end{equation*} 
	\end{small}
	% i.e., its second to the last frontal slices are placed in a reversing order.
	
\end{definition}

\begin{definition}\cite{kilmer2011factorization}\label{def:inverse}
	The identity tensor $\mathcal{I}\in \mathbb{R}^{n\times n \times l}$ is the tensor whose first frontal slice is the $n\times n$ identity matrix and other frontal faces are zero. For a  f-square tensor $\MC{A}\in \T{n}{n}{l}$,  an inverse $\MC{B}$ exists  if it satisfies $\MC{A}\tprod\MC{B} = \MC{B}\tprod\MC{A} = \MC{I}$. $\MC{B}$ is denoted as $\MC{A}^{-1}$.
\end{definition}
% \begin{proposition}(cf. \cite[Lem. 3]{miao2021t})
	% 	\label{prop:bcirc}
	% 	For any $\mathcal A,\mathcal B\in \T{n}{p}{l}$, 
	% 	\begin{enumerate}
		% 		\item $ \bcirc{\mathcal A  \tprod  \mathcal B} =  \bcirc{\mathcal A}\bcirc{\mathcal B} $;
		% 		\item $\bcirc{\mathcal A}^{\top} = \bcirc{\mathcal A^{\top}}.
		% 		$
		% 	\end{enumerate}
	% \end{proposition}
\subsection{Fourier domain representation}

The t-product based computation can   be efficiently implemented by using   fast Fourier transform (FFT) instead of directly computing \eqref{eq:t_prod_formular}; see \cite{kilmer2011factorization}. Fourier transform is not only useful in implementation, but also important in our later analysis. For this reason, we need the representation of $\mathcal A\in \mathbb R^{n\times p\times l}$ in the Fourier domain, which will be given in Definition \ref{def:DFT}. Before that, the following notations are  introduced first.

For any $\mathcal{A}_{i} \in \mathbb{R}^{m_{i} \times n_{i} \times p}$ and $\mathcal{V}_{i} \in \mathbb{R}^{m_{i} \times n \times p}$ with $i \in[l]$, we denote
$$\small
\operatorname{Diag}\left(\mathcal{A}_{1},  \cdots, \mathcal{A}_{l}\right):=\left[\begin{array}{cccc}
	\begin{smallmatrix}
		\mathcal{A}_{1} & &  \\
		& \ddots & \\
		& & \mathcal{A}_{l}
	\end{smallmatrix}
\end{array}\right], \quad \operatorname{Vec}\left(\mathcal{V}_{1},  \cdots, \mathcal{V}_{l}\right):=\left[\begin{array}{c}
	\begin{smallmatrix}
		\mathcal{V}_{1} \\
		\vdots \\
		\mathcal{V}_{l}
	\end{smallmatrix}
\end{array}\right],
$$
and sometimes, they are abbreviated as $\operatorname{Diag}\left(\mathcal{A}_{i}: i \in[l]\right)$ and $\operatorname{Vec}\left(\mathcal{V}_{i}: i \in[l]\right)$, respectively. When all $\mathcal{A}_{i}\left(\mathcal{V}_{i}\right)$ become matrices (or vectors or scalars), similar symbols are also used.

Using this notation, 
$	\MC{A}=\operatorname{fold}(\Vector{A^{(i)}})
$
is sometimes abbreviated as $\MC{A} =\operatorname{fold}\left({A}^{(i)}: i \in[l]\right)$.

\begin{definition}\cite{kilmer2011factorization}\label{def:DFT}\label{DFT}
	For  $\mathcal{A}\in \mathbb{R}^{n\times p \times l}$,
	the discrete Fourier transform (DFT) of $	\mathcal{A}
	%	= [A^{(1)}, A^{(2)}, \cdots, A^{(n)}]^{\top}
	$ is defined as
	%	\begin{equation*}
		%		\operatorname{unfold}(\hat{\mathcal{A}}) =  {\sqrt{l}}(F_l\otimes I_n)\operatorname{unfold}(\mathcal{A}),
		%	\end{equation*}
	\begin{equation}\label{eq:DFT}
		\hat{\mathcal{A}} 
		:= \operatorname{fold}\bigxiaokuohao{ {\sqrt{l}}(F_l\otimes I_n)\operatorname{unfold}(\mathcal{A})}
		= \operatorname{fold}\bigxiaokuohao{ {\sqrt{l}}(F_l\otimes I_n)\Vector{A^{(i)}}} \in \mathbb C^{n\times p\times l},
	\end{equation}
	where $F_l$ is the normalized Fourier operator and its $(i,j)$-th entry is  $(F_l)_{ij}=\frac{1}{\sqrt{l}} \omega^{(i-1)\cdot (j-1)}$ and  $\omega=e^{-\frac{2\pi \mathrm{i}}{l}}$ is the primitive $l$th root of unity and $\mathrm{i}^2 = -1$. There hold  $\hat{A}^{(1)}\in \M{n}{p}, \hat{A}^{(i)}\in \C{n}{p}, \hat{A}^{(i)} = \operatorname{conj}(\hat{A}^{(l+2-i)}), i\in[l]\setminus \{ 1 \}$.  
	
	Throughout this paper, we represent the linear transform \eqref{eq:DFT} as $\hat{\mathcal A} = L(\mathcal A)$, i.e., we will use the notation $L(\cdot)$ to denote the  DFT in \eqref{eq:DFT}.

	The inverse  discrete Fourier transform (IDFT) of $\hat{\mathcal{A}}
	%	=[\hat{A}^{(1)},
	%	\hat{A}^{(2)},
	%	\hat{A}^{(3)},
	%	\cdots,
	%	\hat{A}^{(n)}]^{\top}
	$ is defined as
	%	\begin{equation*}
		%		\operatorname{unfold}({\mathcal{A}})=\frac{1}{\sqrt{l}}(F^H_l\otimes I_n) \operatorname{unfold}(\hat{\mathcal{A}}),
		%	\end{equation*}
	\begin{equation}\label{eq:IDFT}
		{\mathcal{A}}
		=\operatorname{fold}\bigxiaokuohao{\frac{1}{\sqrt{l}}(F^H_l\otimes I_n) \operatorname{unfold}(\hat{\mathcal{A}})}
		=\operatorname{fold}\bigxiaokuohao{\frac{1}{\sqrt{l}}(F^H_l\otimes I_n) \Vector{\hat{A}^{(i)}}
		},
	\end{equation}
	where $F_l^{H} = F^{-1}_l$.
	
	Likewise, we will use $\mathcal A = L^{-1}(\hat{\mathcal A})$ to represent the IDFT in \eqref{eq:IDFT}. 
	%	where $F_l^H$ is  the conjugate transpose of $F_l$. 
\end{definition}
\begin{remark} It follows from
	\cite{kilmer2011factorization} that obtaining $\hat{\mathcal A} $ can be efficiently done by using FFT in  Matlab:  $\hat{\mathcal{A}} = \texttt{fft}(\mathcal{A},[~],3)$. % as the result of DFT  on all frontal slices  $\mathcal{A}(:,:,i)$. 
	Similarly, one can compute $\mathcal{A}$ from $\hat{\mathcal{A}}$ using the command $\mathcal{A} = \texttt{ifft}(\hat{\mathcal{A}},[~],3)$.
	%	For the
	%	remainder of the paper, hat notation is used to indicate that we are referencing the
	%	object after having taken an DFT along the third dimension.
\end{remark}
\begin{remark}
	Throughout this paper, the notation $\hat{\mathcal A}$ is always referred to the DFT of $\mathcal A$, and $\hat A^{(i)}$ is always referred to the $i$-th frontal slice of $\hat{\mathcal A}$.
\end{remark}

% \color{red}\begin{remark}\label{remark:equality}
	% 	According to  \cite{kernfeld2015tensor}, $\hat{\MC{A}} = L(\MC{A})$ denotes the third-order tensor obtained by applying (non-normalized)
	% 	DFT  to each tube fiber $\MC{A}(j, k, :)$ and $\MC{A} = L^{-1}(\hat{\MC{A}}) $ denotes 
	% 	IDFT of $\MC{A}$ along the each tube fiber.
	% \end{remark} \color{black}
\begin{remark}\label{remark:Linear relationship}
	\cite{zheng2021t} pointed out the following relations between $\mathcal A$ and $\hat{\mathcal A}$:
	\begin{align}
		\left\{\begin{matrix}
			% \left\{\begin{split}
				\hat{A}^{(k)}={\omega}^{(k-1)\cdot 0}{A}^{(1)}+{\omega}^{(k-1)\cdot 1}{A}^{(2)}+\cdots+{\omega}^{(k-1)\cdot (l-1)}{A}^{(l)}\\
				~~~~~A^{(k)}=\frac{1}{l}\bigxiaokuohao{\bar{\omega}^{(k-1)\cdot 0}\hat{A}^{(1)}+\bar{\omega}^{(k-1)\cdot 1}\hat{A}^{(2)}+\cdots+\bar{\omega}^{(k-1)\cdot (l-1)}\hat{A}^{(l)}}
			\end{matrix}\right.,
			%\end{split}\right.
		\end{align}
		where $k
		\in [l]$ and $\bar{\omega} = \omega^{-1}$. %The above relations have also been explicitly pointed out in \cite{zheng2021t}. 
	\end{remark}

	\begin{proposition}(cf. \cite[Sect. 2.3]{lu2019tensor})
		\label{prop:bcirc_properties}
		For any $\mathcal A,\mathcal B\in \T{n}{p}{l}$, 
		\begin{enumerate}
			% 		\item $ \bcirc{\mathcal A  \tprod  \mathcal B} =  \bcirc{\mathcal A}\bcirc{\mathcal B} $;
			\item $
			(F_l\otimes I_n)\operatorname{bcirc}(\mathcal{A})(F^H_l\otimes I_p) 
			= \Diag{\hat{A}^{(i)}}
			%			 =\begin{bmatrix}
				%				\hat{A}^{(1)} &  & \\ 
				%				& \ddots  & \\ 
				%				&  & \hat{A}^{(l)}
				%			\end{bmatrix}
			%		 = \operatorname{diag}(\hat{A}^{(1)},
			%		\cdots,
			%		\hat{A}^{(n)})
			%		=\operatorname{blockdiag}(\mathcal{\hat {A}})
			=:\hat A
			$;
			\item $\mathcal C = \mathcal A\tprod\mathcal B \Leftrightarrow {\rm Diag}\bigxiaokuohao{\hat C^{(i)}:i\in [l]} = {\rm Diag}\bigxiaokuohao{\hat A^{(i)}\hat B^{(i)}: i\in [l] }   \Leftrightarrow \hat C^{(i)} = \hat A^{(i)}\hat B^{(i)}, i \in [l]$.
			% 		\item $\bcirc{\mathcal A}^{\top} = \bcirc{\mathcal A^{\top}}.
			% %		, \bcirc{\mathcal A}^H = \bcirc{\mathcal A^H}   
			% 		$
		\end{enumerate}
	\end{proposition}
	\begin{proposition}\label{prop:DFT tran}
		Let $\mathcal{A}\in \mathbb{R}^{n\times p \times l}$; then
		% 	 \begin{equation*}
			$
			\sqrt{l}(F_l\otimes I_p)\operatorname{unfold}(\mathcal{A}^{\top}) =
			\Vector{(\hat{A}^{(i)})^H}
			%		 = \begin{bmatrix}
				%			(\hat{A}^{(1)})^H\\ 
				%			\vdots\\ 
				%			(\hat{A}^{(l)})^H
				%		\end{bmatrix}
			, $
			% 	\end{equation*}
		which means $L(\MC{A}^{\top}) = \Fold{(\hat{A}^{(i)})^{H}}.$
	\end{proposition}
	The proof of Proposition \ref{prop:DFT tran} is left to Appendix \ref{apx:1}. 
	\begin{remark}\label{remark:DFT equality}
		Combing Proposition \ref{prop:DFT tran} and item 2 of Proposition \ref{prop:bcirc_properties} leads to
		$$\mathcal C = \mathcal A^{\top}\tprod\mathcal B
		% \Leftrightarrow \hat C = \hat A^{H}\hat B
		\Leftrightarrow \hat C^{(i)} = (\hat A^{(i)})^H\hat B^{(i)}, i \in [l].
		$$
	\end{remark}

	\subsection{Trace, t-positive (semi)definiteness, skew-symmetric, and orthogonality}

	\begin{definition}(cf. \cite[Def. 7]{zheng2021t}, \cite[Prop. 1 (b)]{zheng2021t})
		\label{def:trace_tprod}
		Let $\mathcal A\in\T{n}{n}{l}$. The trace of $\mathcal A$, denoted by $\trace{\mathcal A}$, is defined as $\trace{\mathcal A} := \sum^l_{i=1} \trace{\hat A^{(i)}}$ = \trace{\bcirc{\mathcal A}}.
		% 	, where $\hat A^{(i)}$'s are given in item 2 of Proposition \ref{prop:bcirc_properties}.
	\end{definition}
	%According to \cite[Prop. 1 (b)]{zheng2021t}, $\trace{\mathcal A} = \trace{\bcirc{\mathcal A}}$ is always real if $\mathcal A$ is real.
	A symmetric version of the following relation was given in \cite[Rmk. 9]{zheng2021t}. Here we need a nonsymmetric version. 
	\begin{proposition}
		\label{prop:trace_inner_t_prod}
		For any $\mathcal A,\mathcal B\in \T{n}{p}{l}$, 
		$
		\trace{\mathcal A^{\top}\tprod\mathcal B} = 	l\innerprod{\mathcal A}{\mathcal B} .
		$
	\end{proposition}
	The proof of Proposition \ref{prop:trace_inner_t_prod} is left to Appendix \ref{apx:4}. 
	Proposition \ref{prop:trace_inner_t_prod} immediately gives that:
	\begin{proposition}
		\label{prop:trace_switch}
		For any 
		$\mathcal A,\mathcal B\in\T{n}{p}{l}$, 
		\begin{enumerate}
			\item 
			$\trace{\mathcal A^{\top}\tprod\mathcal B} = \trace{\mathcal A\tprod\mathcal B^{\top}} = \trace{\mathcal B^{\top}\tprod\mathcal A} = \trace{\mathcal B\tprod\mathcal A^{\top}}$;
			\item $
			\langle\mathcal{A}\ast \mathcal{B},\mathcal{C}	\rangle~ =~ 	\langle\mathcal{A},\mathcal{C}\ast \mathcal{B}^{\top}	\rangle
			~ =~ 	\langle\mathcal{B},\mathcal{A}^{\top}\ast \mathcal{C}	\rangle
			$;
			\item $	\langle\mathcal{\hat A},\mathcal{\hat B}	\rangle ~=~ l	\langle\mathcal{A},\mathcal{B}	\rangle
			$.
		\end{enumerate}
	\end{proposition}
	
	\begin{definition}(c.f. \cite{zheng2021t,kilmer2013third})
		A tensor $\mathcal A\in\T{n}{n}{l}$ is called symmetric  if $\mathcal A = \mathcal A^{\top}$. The set of symmetric 
		tensors is denoted as $\TS{n}{n}{l}$.	 $\mathcal A\in\TS{n}{n}{l}$ is called symmetric t-positive (semi)definite if $\innerprod{\mathcal X}{\mathcal A\tprod\mathcal X} > (\geq )0$ for any $\mathcal X\in\mathbb R^{n\times 1\times l}\setminus \{ \MC{O} \}$ $\bigxiaokuohao{\mathcal X\in \mathbb R^{n\times 1\times l}}$. The sets of symmetric t-positive (semi)definite tensors is denoted as $\TSPP{n}{n}{l}~ (\TSP{n}{n}{l})$. 
	\end{definition}
	\begin{remark}
		\label{rmk:t_psd}
		$\mathcal A\in \TSPP{n}{n}{l}$ ($\TSP{n}{n}{l}$)  if and only if every frontal slice $\hat{{A}}^{(i)}$ of $\hat{\MC{A}}$ in the Fourier domain is Hermitian positive (semi)definite.
	\end{remark}
	% \begin{definition}
		% 	A tensor $\mathcal A\in \TC{n}{p}{l}$ is called Hermitian if $\mathcal A = \mathcal A^{H}$. The set of Hermitian
		% 	tensors are denoted as $\HTS{n}{n}{l}$.	 $\mathcal A\in\HTS{n}{n}{l}$ is called Hermitian t-positive (semi)definite if $\innerprod{\mathcal X}{\mathcal A\Delta\mathcal X} > (\geq )0$ for any $\mathcal X\in\mathbb C^{n\times 1\times l}\setminus \{ \MC{O} \}$ $\bigxiaokuohao{\mathcal X\in \mathbb C^{n\times 1\times l}}$. The sets of Hermitian t-positive (semi)definite tensors are denoted as $\HTSPP{n}{n}{l} (\HTSP{n}{n}{l})$. 
		% \end{definition}
	Next,  similar to skew-matrices, we define skew-symmetric tensors. Skew-symmetric tensors are important in deriving retractions.
	\begin{definition}
		A tensor $\mathcal A\in\T{n}{n}{l}$ is called skew-symmetric if $\mathcal A = -\mathcal A^{\top}$. The set of skew-symmetric 
		tensors are denoted as $\TSkew{n}{n}{l}$.
	\end{definition}
	% \begin{lemma}\label{lem:positive definite}
		%  $ \mathcal{A}\in \TSPP{n}{n}{l}$ if and only if $\hat{\MC{A}}\in \HTSPP{n}{p}{l}$.
		% \end{lemma}
	\begin{lemma}\label{lem:positive definite}
		$ \mathcal{I} + \mathcal{V}^{\top}\ast\mathcal{V}\in \TSPP{n}{n}{l}$ for all $\MC{V}\in \T{n}{p}{l}$.
	\end{lemma}
	% The proof of Lemma \ref{lem:positive definite} is left to Appendix \ref{apx:2}. 
	\begin{lemma}\label{lemma:orthogonal complement}
		The orthogonal complement of $\TSkew{p}{p}{l}$ is $\TS{p}{p}{l}$.
	\end{lemma}
	% The proof of Lemma \ref{lemma:orthogonal complement} is left to Appendix \ref{apx:3}. 
	The proofs of Lemma \ref{lem:positive definite} and \ref{lemma:orthogonal complement} are left to Appendix \ref{apx:2} $-$ \ref{apx:3}, respectively.

	\begin{definition} \cite{kilmer2011factorization}
		A tensor $\MC{X}\in \mathbb{R}^{n\times n \times l }$ is orthogonal if  $\MC{X}^{\top}\ast\MC{X} = \MC{X}\ast\MC{X}^{\top} = \mathcal{I}\in \mathbb{R}^{n\times n \times l }$. And $\MC{X}\in\mathbb R^{n\times p\times l}$ ($n\geq p$) is partially orthogonal if  $\MC{X}^{\top}\ast\MC{X} = \mathcal{I}\in \mathbb{R}^{p\times p \times l }$. As noted in the introduction, the set of partially orthogonal tensors of size $n\times p\times l$ is denoted as $\st{n,p,l}$.
	\end{definition}
	$\MC{X}$ being orthogonal   implies that every $\hat{X}^{(i)}$ in the Fourier domain is unitary \cite{kilmer2011factorization}. 
	
	Based on Proposition \ref{prop:trace_inner_t_prod} and the definition of the trace,  if
	$\MC{X}\in\st{n,p,l}$ where $n\geq p$,  then $\bigfnorm{\MC{X}}$ is a constant over $\st{n,p,l}$. This means the following: 
	\begin{proposition}\label{prop:for_realization}
		For any given $\mathcal A\in\T{n}{p}{l}$ with $n\geq q$,	$\min_{ \MC{X}\in \st{n,p,l}  } \bigfnorm{\mathcal A-\MC{X}}^2 $ and $\max_{\MC{X}\in \st{n,p,l}}\innerprod{\mathcal A}{\MC{X}}$ are equivalent. 
	\end{proposition}
	
	\subsection{t-SVD}
	
	SVD in the t-product sense is important.
	\begin{theorem}(t-SVD, \cite[Thm. 4.1]{kilmer2011factorization}, \cite[Thm. 2.2]{lu2019tensor})
		\label{th:t_svd}
		Let $\mathcal A\in\T{n}{p}{l}$. Then it can be factorized as $\mathcal A = \mathcal U\tprod\mathcal S\tprod\mathcal V^{\top}$, where $\mathcal U\in\T{n}{n}{l}$ and $\mathcal V\in\T{p}{p}{l}$ are orthogonal tensors, and $\mathcal S\in\T{n}{p}{l}$ is a f-diagonal tensor.
	\end{theorem}
	\begin{remark}
		%From Definition \ref{DFT}, we would only need to compute individual matrix t-SVD's for
		%about half the frontal slice of $\hat{\MC{A}}$, the remainder can be obtained from the conjugate symmetry of the Fourier transform.
		\cite{lu2019tensor} pointed out that to keep $\MC{U},\MC{S}$ and $\MC{V}$ to be real, one need to perform \cite[Alg. 2]{lu2019tensor}.
		Specifically, $\text{for}~ i = 1,\cdots,\lceil  \frac{l+1}{2}  \rceil$, let $\hat A^{(i)} = \hat U^{(i)} \hat S^{(i)} (\hat V^{(i)})^H$ be the SVD of $\hat A^{(i)}$, where $\lceil x \rceil$ is denoted as the nearest integer greater than or equal to $x$. For $i=1+ \lceil  \frac{l+1}{2}  \rceil,\ldots,l$, 
		$
		\hat A^{(i)} = \conj{\hat A^{(l+2-i)}}, \hat U^{(i)} = \conj{\hat U^{(l+2-i)}}, \hat V^{(i)} = \conj{\hat V^{(l+2-i)}}, \hat S^{(i)} = \hat S^{(l+2-i)}. 
		$
	\end{remark}
	As in the matrix case, the frontal slices $\hat S^{(i)}$ of $\hat{\MC{S}}$ are all diagonal, and their diagonal entries can be chosen nonnegative. 
	The compact t-SVD was mentioned in \cite{kilmer2011factorization,lu2019tensor,miao2020generalized}. We formally present it here for later use. 
	\begin{theorem}\cite{kilmer2011factorization,lu2019tensor,miao2020generalized}
		\label{th:t_svd_compact}
		Let $\mathcal A\in\T{n}{p}{l}$ with $n\geq q$. Then it can be factorized as $\mathcal A = \mathcal U\tprod\mathcal S\tprod\mathcal V^{\top}$, where $\mathcal U\in\T{n}{p}{l}$ is partially orthogonal,   
		$\mathcal V\in\T{p}{p}{l}$ is orthogonal, and $\mathcal S\in\T{p}{p}{l}$ is f-diagonal. 
	\end{theorem}

	\subsection{t-QR decomposition}\label{subsec:t-QR}

	\begin{theorem}[t-QR, \cite{kilmer2013third}] \label{th:t_qr}
		Let $\mathcal A\in\T{n}{p}{l}$ with $n\geq p$. Then $\mathcal A$ can be written as
		\begin{align}\label{eq:t_qr}
			\mathcal A = \mathcal Q\tprod\mathcal R,
		\end{align}
		where $\mathcal Q\in \st{n,p,l}$ and $\mathcal R\in\TU{p}{p}{l}$.  
		If $\mathcal A \in L^{-1}\bigxiaokuohao{\mathbb{C}_*^{n\times p \times l}}:= \bigdakuohao{L^{-1}(\hat{\MC{A}})|\hat{\MC{A}}\in \TC{n}{p}{l}_*}$  and we require that  $\mathcal{R}\in L^{-1}\bigxiaokuohao{\TCUP{p}{p}{l}}:= \bigdakuohao{L^{-1}(\hat{\MC{A}})|\hat{\MC{A}}\in \TCUP{n}{p}{l}}$,  then the decomposition \eqref{eq:t_qr} is unique.
		Factorization of the form \eqref{eq:t_qr} is called the t-QR decomposition (t-QR for short). 
	\end{theorem}
	%The proof of Theorem \ref{th:t_qr} is left to Appendix \ref{apx:5}. 
	\begin{lemma}\label{lemma:upp_dim}
		% Denote   $L^{-1}\bigxiaokuohao{\TCUP{p}{p}{l}}$  as the set of $\MC{R}$ factor of t-QR decomposition of $\MC{A}\in \T{n}{p}{l}$. Then
		$L^{-1}\bigxiaokuohao{\TCUP{p}{p}{l}}$	is isomorphic to $\mathbb{R}^{\frac{p^2l+p}{2}}$ if $l$ is odd, and $L^{-1}\bigxiaokuohao{\TCUP{p}{p}{l}}$	is isomorphic to
		$\mathbb{R}^{\frac{p^2l}{2}+p}$ if $l$ is even.
		%$\dim{(\TU{p}{p}{l})}=p\bigxiaokuohao{p\frac{l}{2}+\frac{1}{2^{|\sin(\frac{l\pi}{2})|}}}.$
	\end{lemma}
	%The proof of Lemma \ref{lemma:upp_dim} is left to Appendix \ref{apx:6}. 
	The proofs of Theorem \ref{th:t_qr} and Lemma \ref{lemma:upp_dim} are left to Appendix \ref{apx:5} and \ref{apx:6}, respectively. 
	\subsection{t-polar decomposition}\label{t-PD}
	Similar to the matrix counterpart, based on t-SVD, we can define t-polar decomposition (t-PD for short). 
	\begin{theorem}[t-PD] \label{th:t_pd}
		Let $\mathcal A\in\T{n}{p}{l}$ with $n\geq p$. Then $\mathcal A$ can be written as
		\begin{align}\label{eq:t_pd}
			\mathcal A = \mathcal P\tprod\mathcal H,
		\end{align}
		where $\mathcal P\in \st{n,p,l}$ and $\mathcal H\in\TSP{p}{p}{l}$.  $\mathcal H$ is unique. Furthermore, if $\mathcal A^{\top}\tprod \mathcal A\in \TSPP{p}{p}{l}$, then $\mathcal P$ is unique and $\mathcal H\in\TSPP{p}{p}{l}$.  
		
		%	Factorization of the form \eqref{eq:t_pd} is called the t-polar decomposition (t-PD). 
	\end{theorem}
	%The proof of Theorem \ref{th:t_pd} is left to Appendix \ref{apx:7}. 
	\begin{proposition}
		\label{prop:t_pd_for_retraction}
		If $\mathcal A^{\top}\tprod\mathcal A \in\TSPP{p}{p}{l}$, then $\mathcal P$ and $\mathcal H$ defined in Theorem \ref{th:t_pd} are given by
		\[
		\mathcal P = \mathcal A\tprod\bigxiaokuohao{\mathcal A^{\top}\tprod\mathcal A}^{-\frac{1}{2}},~{\rm and}~\mathcal H = \bigxiaokuohao{\mathcal A^{\top}\tprod\mathcal A  }^{\frac{1}{2}},
		\]
		where the square root notation on tensors was defined in \cite[Sect. 4.4]{zheng2021t}.
		\begin{proposition}\label{prop:sym equal}
			If $\mathcal A \in L^{-1}\bigxiaokuohao{\mathbb{C}_*^{n\times p \times l}}$, then $\mathcal A^{\top}\tprod\mathcal A \in\TSPP{p}{p}{l}$.
		\end{proposition}
	\end{proposition}
	%The proof of Proposition \ref{prop:t_pd_for_retraction} is left to Appendix \ref{apx:8}. 
	\begin{theorem}
		\label{th:max_sol_related_pd}
		Let $\mathcal A\in\T{n}{p}{l}$ with $n\geq q$, admit the compact t-SVD $\mathcal A = \mathcal U\tprod\mathcal S\tprod\mathcal V^{\top}$. Then the optimal solution to $\max_{\mathcal P\in \st{n,p,l}}\innerprod{\mathcal A}{\mathcal P}$ is given by the t-PD of $\mathcal A$, namely, $\mathcal P = \mathcal U\tprod\mathcal V^{\top}$.
	\end{theorem}
	%The proof of Theorem \ref{th:max_sol_related_pd} is left to Appendix \ref{apx:9}. 
	The proofs of Theorem \ref{th:t_pd}, Proposition \ref{prop:t_pd_for_retraction}, and Theorem \ref{th:max_sol_related_pd} are left to Appendix \ref{apx:7} $-$ \ref{apx:9}, respectively. 
	
	\subsection{t-exponential}\label{subsect:t-exp}
	In \cite{lund2020tensor}, the author defined tensor t-functions for a third-order f-square
	tensor \(\mathcal{A} \in \R^{n \times n \times l}\) based on t-product. 
	In particular, the exponential of a third-order tensor
	\(\mathcal{A} \in \R^{n \times n \times l}\) is defined as follows:
	\begin{equation}\label{eq:t-exp2}
		\ex{\mathcal{A}}= \mathrm{fold}\bigxiaokuohao{
			\ex{\mathrm{bcirc}(\mathcal{A})}\mathrm{unfold}\bigxiaokuohao{\mathcal{I}}
		}.\end{equation}
	%where $\mathcal I\in\mathbb R^{n\times n\times l}$. 
	In \cite{miao2020generalized}, the authors extended the definition of tensor t-functions to arbitrary third-order tensors (not necessarily f-square). We present an equivalent definition of the t-exponential of third-order tensors for later use.
	\begin{definition}\label{def:t-exponential}
		The exponential of tensor \(\mathcal{A} \in \R^{n \times n \times l}\) based on t-product ( t-exponential for short) is 
		\begin{equation}\label{eq:exp well defined}
			%e^{\mathcal{A}} =
			\ex{\mathcal{A}} := \sum\nolimits_{k=0}^{\infty} \frac{1}{k!} \mathcal{A}^k,\end{equation}
		where \(\mathcal{A}^k = \mathcal{A} * \mathcal{A} * \cdots * \mathcal{A}\) (\(k\) copies) 
		with the convention that \(\mathcal{A}^0 = \mathcal{I}\).
	\end{definition}
	\begin{remark}
		This is well defined, i.e., the series is convergent. The proof of  well-defined property of \eqref{eq:exp well defined} is left to Appendix \ref{apx:10}. According to the proof, \eqref{eq:exp well defined} can be further written as
		\begin{equation}\label{eq:t-exp}
			\ex{\mathcal{A}} = L^{-1}\bigxiaokuohao{
				\Fold{\ex{\hat{{A}}^{(i)}}}} 
			= L^{-1}\bigxiaokuohao{
				\Fold{\ex{(L(\MC{A}))^{(i)}}}}.
		\end{equation}
	\end{remark}
	The proof of equivalence of \eqref{eq:t-exp} and \eqref{eq:t-exp2} is left to Appendix \ref{apx:11}. 
	
	Properties of the matrix exponential, such as those mentioned above, 
	can be extended to the  t-exponential. We list those that are needed later.
	
	\begin{proposition}\label{prop:exp smooth}
		The exponential map \(\operatorname{exp}: \R^{n\times n \times l} \to \R^{n \times n \times l}, \mathcal{A} \mapsto \ex{\mathcal{A}}\) is smooth.
	\end{proposition}
	%The proof of Proposition \ref{prop:exp smooth} is left to Appendix \ref{apx:12}. 
	\begin{proposition}\label{prop:exp der}
		Let \(\mathcal{A} \in \R^{n \times p \times l}\). Then
		$\frac{\mathrm{d}}{\mathrm{d}t} \ex{t\mathcal{A}}=\ex{t\mathcal{A}}*\mathcal{A}=\mathcal{A}*\ex{t\mathcal{A}}.$
	\end{proposition}
	%The proof of Proposition \ref{prop:exp der} is left to Appendix \ref{apx:13}. 
	\begin{proposition}\label{prop:exp decomp}
		Consider \(\mathcal{A} \in \R^{n \times n \times l}\) and \(\MC{X} \in \st{m,n,l} \) with \(m \ge n\). Then 
		\[\ex{\MC{X}*\mathcal{A}*\MC{X}^{\top}}=\MC{X}*\ex{\mathcal{A}}*\MC{X}^{\top}.\]
	\end{proposition}
	%The proof of Proposition \ref{prop:exp decomp} is left to Appendix \ref{apx:14}. 
	\begin{proposition}\label{prop:exp Diag tensor}
		Let $\MC{D}_j\in \T{m_j}{n_j}{l}, j\in[p]$. Then
		\[\ex{
			\operatorname{Diag}\left(\mathcal{D}_j: j \in[p]\right) 
			%	\begin{pmatrix}
				%		\mathcal{D}_1 & & &\\
				%		& \mathcal{D}_2 & &\\
				%		& & \ddots & \\
				%		& & & \mathcal{D}_m
				%	\end{pmatrix}
		} = 
		\operatorname{Diag}\left(\ex{\mathcal{D}_j}: j \in[p]\right) 
		%	\begin{pmatrix}
			%		\ex[\mathcal{D}_1] & & &\\
			%		& \ex[\mathcal{D}_2] & &\\
			%		& & \ddots & \\
			%		& & & \ex[\mathcal{D}_m]
			%	\end{pmatrix}
		.\]
	\end{proposition}
	%The proof of Proposition \ref{prop:exp Diag tensor} is left to Appendix \ref{apx:15}.
	\begin{proposition}\label{prop:exp tranpose}
		Let \(\mathcal{A} \in \R^{n \times p \times l}\). Then
		$(\ex{\mathcal{A}})^{\top}=\ex{\mathcal{A}^{\top}}.$
	\end{proposition}
	%The proof of Proposition \ref{prop:exp tranpose} is left to Appendix \ref{apx:16}.
	\begin{proposition}\label{prop:exp addition}
		If \(\mathcal{A}*\mathcal{B}=\mathcal{B}*\mathcal{A}\), then 
		$\ex{\mathcal{A}}*\ex{\mathcal{B}}=\ex{\mathcal{A}+\mathcal{B}}.$
	\end{proposition}
	%The proof of Proposition \ref{prop:exp addition} is left to Appendix \ref{apx:17}.
	The proofs of Proposition \ref{prop:exp smooth} $-$ \ref{prop:exp addition} are left to Appendix \ref{apx:12} $-$ \ref{apx:17},  respectively.

\section{Preliminaries on Riemannian Manifold}
Basic definitions and properties concerning the Riemannian manifold can be found in the books \cite{absil2009optimization,tuIntroductionManifolds2011,boumal2022intromanifolds}. To be more convenient and to make the paper self-contained, we summarize the necessary ones in this section.
\label{sec:pre_rieman}
%\begin{definition}\cite{tuIntroductionManifolds2011}\label{def:locally Euclide}
%	A topological space $\MC{M}$ is locally Euclide dimension $n$ if every point $p$ in $\MC{M}$ has a neighborhood $U$ such that there is a homeomorphism $\phi$ from $U$ onto an open subset of $\mathbb{R}^{n}$. The pair $\left(U, \phi: U \rightarrow \mathbb{R}^{n}\right)$ is called a chart.
%\end{definition}
\begin{definition}\cite{tuIntroductionManifolds2011}\label{def:manifold}
	A topological space $\mathscr{M}$ is locally Euclide dimension $n$ if every point $p$ in $\mathscr{M}$ has a neighborhood $U$ such that there is a homeomorphism $\phi$ from $U$ onto an open subset of $\mathbb{R}^{n}$.
	%		 The pair $\left(U, \phi: U \rightarrow \mathbb{R}^{n}\right)$ is called a chart.
	A topological manifold of dimension $n$ is a Hausdorff, second countable, locally Euclidean dimension $n$ space. 
	%	It is said to be of dimension $n$ if it is locally Euclidean of dimen$\operatorname{sion} n$. 
	Especially, every vector space is a linear manifold.
\end{definition}
%\begin{remark}
%	Every vector space is a linear manifold.
%\end{remark}
\begin{definition}\cite{absil2009optimization}\label{def:embeded submanifold}
	Let $\mathscr{N}$ be a submanifold of $\mathscr{M}$. If the mani­fold topology of $\mathscr{N}$ coincides with its subspace topology induced from the topological space $\mathscr{M}$, then $\mathscr{N}$ is called an embedded submanifold 
	% 	or a regular submanifold 
	of the manifold $\mathscr{M}$.
\end{definition}
%A $d$-dimensional manifold $\mathcal{M}$ is a Hausdorff and second-countable topological space, which is homeomorphic to the $d$-dimensional Euclidean space locally via a family of charts. 
%When the transition maps of intersecting charts are smooth, the manifold $\mathcal{M}$ is called a smooth manifold. 
%Intuitively, the tangent space $T_{x} \mathcal{M}$ at a point $x$ of a manifold $\mathcal{M}$ is the set of the tangent vectors of all the curves at $x$. 
\begin{definition}\cite{absil2009optimization}\label{def:tangent vector}
	A tangent vector $\xi_{x}$ to $\mathscr{M}$ at $x$ is a mapping such that there exists a curve $\gamma$ on $\mathscr{M}$ with $\gamma(0)=x$, satisfying
	$$
	\xi_{x} f:=\left.\dot{\gamma}(0) f \triangleq \frac{\mathrm{d}(f(\gamma(t)))}{\mathrm{d} t}\right|_{t=0}, \quad \forall f \in \mathfrak{F}_{x}(\mathscr{M}),
	$$
	where $\mathfrak{F}_{x}(\mathscr{M})$ is the set of all real-valued functions $f$ defined in a neighborhood of $x$ in $\mathscr{M}$. The tangent space $T_{x} \mathscr{M}$ to $\mathscr{M}$ is defined as the set of all tangent vectors to $\mathscr{M}$ at $x$. 
	% 	The union of all tangent spaces 
	$T\mathscr{M}  :=\bigcup_{x \in \mathscr{M}} T_{x} \mathscr{M}.$	is called the tangent bundle of the manifold.
	% 	, denoted as $T\mathcal{M}$.
\end{definition}
\begin{definition}\cite{absil2009optimization}\label{def:differential}
	The differential of $F: \mathscr{M} \rightarrow \mathscr{N}$ at $x$ is a linear operator $\mathrm{D}F(x): T_{x} \mathscr{M} \rightarrow T_{F(x)} \mathscr{N}$ defined by:
	$$
	\mathrm{D} F(x)[v]:=\left.\frac{\mathrm{d}}{\mathrm{d} t} F(\gamma(t))\right|_{t=0},
	$$
	where $\gamma(t)$ is any curve on the manifold that satisfies $\gamma(0)=x$ and $\dot{\gamma}(0)=v$. 
\end{definition}
\begin{definition}\cite{absil2009optimization}\label{def:Riemannian}
	A Riemannian metric $g$ is defined on each tangent space of $x$ as an inner product $g_{x}: T_{x} \mathscr{M} \times T_{x} \mathscr{M} \rightarrow \mathbb{R}$. 
	$
	g_{x}(\eta, \xi)=\langle\eta, \xi\rangle_{x}
	$
	where $\eta, \xi \in T_{x} \mathscr{M}$ and the $x$ is dropped when context permits.  
	A Riemannian manifold is the combination $(\mathscr{M}, g)$.
	%If $\mathcal{M}$ is equipped with a smoothly varied inner product $g_{x}(\cdot, \cdot):=\langle\cdot, \cdot\rangle_{x}$ on the tangent space, then $(\mathcal{M}, g)$ is a Riemannian manifold. 
\end{definition}
\begin{definition}\cite{absil2009optimization}\label{defin:vector field}
	A smooth scalar field on a manifold $\mathscr{M}$ is a smooth function $f: \mathscr{M} \rightarrow  \mathbb{R}$. 
	% If $f$ is a smooth function, we say it is a smooth scalar field. 
	The set of smooth scalar fields on $\mathscr{M}$ is denoted by $\mathfrak{F}(\mathscr{M})$. A smooth vector field $\xi$ on a manifold $\mathscr{M}$ is a smooth function from $\mathscr{M}$ to the tangent bundle $T\mathscr{M}$ that assigns to each point $x\in\mathscr{M}$ a tangent vector $\xi_x\in T_x\mathscr{M}$.
	Let $\mathfrak{X}(\mathscr{M})$ denote the set of smooth vector fields on $\mathscr{M}$.
\end{definition}
\begin{definition}\cite{absil2009optimization} An affine connection $\nabla$ on a manifold $\mathscr{M}$ is a mapping 
	$
	\nabla: \mathfrak{X}(\mathscr{M}) \times \mathfrak{X}(\mathscr{M}) \rightarrow \mathfrak{X}(\mathscr{M}),
	$
	which is denoted by $(\eta, \xi) \stackrel{\nabla}{\longrightarrow} \nabla_{\eta} \xi$ and satisfies the following properties. 
	For $\eta,  \xi, \zeta \in \mathfrak{X}(\mathscr{M}), f, g \in \mathfrak{F}(\mathscr{M})$, and $a, b \in \mathbb{R}$: $(i) ~\nabla_{f \eta+g \chi} \xi=f \nabla_{\eta} \xi+g \nabla_{\chi} \xi$,   $ (ii)~\nabla_{\eta}(a \xi+b \zeta)=a \nabla_{\eta} \xi+b \nabla_{\eta} \zeta$, and  $(iii)~\nabla_{\eta}(f \xi)=(\eta f) \xi+f \nabla_{\eta} \xi.$
	% \begin{itemize}
		% \item[(1)]
		% % $\mathfrak{F}(\mathcal{M})$-linearity in $\eta$ :
		% $\nabla_{f \eta+g \chi} \xi=f \nabla_{\eta} \xi+g \nabla_{\chi} \xi$,
		% \item[(2)]
		% % $\mathbb{R}$-linearity in $\xi: \quad
		% $\nabla_{\eta}(a \xi+b \zeta)=a \nabla_{\eta} \xi+b \nabla_{\eta} \zeta,$
		% \item[(3)] 
		% % Product rule (Leibniz' law ) : \quad
		% $\nabla_{\eta}(f \xi)=(\eta f) \xi+f \nabla_{\eta} \xi.$
		% \end{itemize}
	
	The vector field $\nabla_{\eta} \xi$ is called the covariant derivative of $\xi$ with respect to $\eta$ for the affine connection $\nabla$. For a Riemannian manifold, one of the affine connections, called Riemannian connection or Levi-Civita connection, uniquely satisfies the following two additional conditions: $(i)~\left(\nabla_{\eta} \xi-\nabla_{\xi} \eta\right) f=\eta(\xi f)-\xi(\eta f)$, and $(ii)~\zeta\langle\eta, \xi\rangle=\left\langle\nabla_{\zeta} \eta, \xi\right\rangle+\left\langle\eta, \nabla_{\zeta} \xi\right\rangle$.
	% \begin{itemize}
		% \item[(1)]$\left(\nabla_{\eta} \xi-\nabla_{\xi} \eta\right) f=\eta(\xi f)-\xi(\eta f)$,
		%  \item[(2)] $\zeta\langle\eta, \xi\rangle=\left\langle\nabla_{\zeta} \eta, \xi\right\rangle+\left\langle\eta, \nabla_{\zeta} \xi\right\rangle$.
		% %  (compatibility with the Riemannian metric).
		% \end{itemize}
\end{definition}
\begin{definition}
	Let $c: I \rightarrow \mathscr{M}$ be a smooth curve on a manifold equipped with a connection $\nabla$. There exists a unique operator $\frac{\mathrm{D}}{\mathrm{d} t}: \mathfrak{X}(c) \rightarrow \mathfrak{X}(c)$ which satisfies the following properties for all $Y, Z \in \mathfrak{X}(c), U \in \mathfrak{X}(\mathcal{M}), g \in \mathfrak{F}(I)$, and $a, b \in \mathbb{R}$ :
	(i) $\mathbb{R}$-linearity: $\frac{\mathrm{D}}{\mathrm{d} t}(a Y+b Z)=a \frac{\mathrm{D}}{\mathrm{d} t} Y+b \frac{\mathrm{D}}{\mathrm{d} t} Z$;
	(ii) Leibniz rule: $\frac{\mathrm{D}}{\mathrm{d} t}(g Z)=g^{\prime} Z+g \frac{\mathrm{D}}{\mathrm{d} t} Z$;
	(iii) Chain rule: $\left(\frac{\mathrm{D}}{\mathrm{d} t}(U \circ c)\right)(t)=\nabla_{c^{\prime}(t)} U$ for all $t \in I$.
	$\frac{\mathrm{D}}{\mathrm{d} t}$ is called the induced covariant derivative.
	%	 If moreover $\mathcal{M}$ is a Riemannian manifold and $\nabla$ is compatible with its metric $\langle\cdot, \cdot\rangle$ (e.g., if $\nabla$ is the Riemannian connection), then the induced covariant derivative also satisfies:
	%	4. Product rule: $\frac{\mathrm{d}}{\mathrm{d} t}\langle Y, Z\rangle=\left\langle\frac{\mathrm{D}}{\mathrm{d} t} Y, Z\right\rangle+\left\langle Y, \frac{\mathrm{D}}{\mathrm{d} t} Z\right\rangle$, where $\langle Y, Z\rangle \in \mathfrak{F}(I)$ is defined by $\langle Y, Z\rangle(t)=\langle Y(t), Z(t)\rangle_{c(t)}$.
\end{definition}
\begin{definition}\cite{huang2013optimization}
	%	Let $c: I \rightarrow \mathcal{M}$ be a smooth curve. The acceleration of $c$ is the smooth vector field $c^{\prime \prime} \in \mathfrak{X}(c)$ defined by:
	%	$$
	%	c^{\prime \prime}=\frac{\mathrm{D}}{\mathrm{d} t} c^{\prime} .
	%	$$
	%	$A$ geodesic is a smooth curve $c: I \rightarrow \mathcal{M}$ such that $c^{\prime \prime}(t)=0$ for all $t \in I$.
	The geodesic $\gamma(t)$ defined by an affine connection is a curve that satisfies
	$$
	\ddot{\gamma}(t):=\nabla_{\dot{\gamma}(t)} \dot{\gamma}(t)
	:=\frac{\mathrm{D}^{2}}{\mathrm{d} t^{2}} \gamma(t):=\frac{\mathrm{D}}{\mathrm{d} t} \dot{\gamma}(t)
	=0.
	$$
\end{definition}
\begin{definition}\cite{absil2009optimization}
	The Riemannian gradient $\operatorname{grad}f(x)$ of a function $f$ at $x$ is an unique vector in $T_x\mathscr{M}$ satisfying
	$\langle\operatorname{grad} f(x), \xi_x\rangle_{x}=\mathrm{D} f(x)[\xi_x], \quad \forall \xi_x \in T_{x} \mathscr{M}.$
\end{definition}
\begin{definition}\cite{absil2009optimization}
	The Riemannian Hessian  $	\operatorname{Hess}f(x)$ is a mapping from the tangent space $T_{x} \mathscr{M}$ to the tangent space $T_{x} \mathscr{M}$ :
	$
	\operatorname{Hess} f(x)[\xi]:={\nabla}_{\xi} \operatorname{grad} f(x),
	$
	where ${\nabla}$ is the Riemannian connection on $\mathscr{M}$.
\end{definition}
\begin{lemma}\label{Lemma:grad Hess proj}
	\cite{boumal2022intromanifolds}
	For a function $f$ defined on a submanifold $\mathscr{M}$ with the Euclidean metric on its tangent space, if it can be extended to the ambient Euclidean space denoted as $\bar{f}$, one has 
	% 	its Riemannian gradient grad $f$ and Riemannian Hessian Hess $f$ :
	$$
	\begin{aligned}
		\operatorname{grad} f(x) &=\mathbf{P}_x(\operatorname{grad} \bar f(x)), \\
		\operatorname{ Hess } f(x)[u] &=\mathbf{P}_x(\mathrm{D} \bar G(x)[u]), u \in T_{x} \mathscr{M},
	\end{aligned}
	$$
	where $\mathrm{D}$ is the Euclidean derivative, $\mathbf{P}_x(y)
	%:=\arg \min _{z \in T_{x} %\mathcal{M}}\|y-z\|^{2}
	$ denotes the orthogonal projection operator from Euclidean space $\mathscr{E}$ to $T_{x} \mathscr{M}$, and $\bar G(x)$ denote a smooth extension of the $\operatorname{grad} f(x)$ to a neighborhood of $\mathscr{M}$ in the ambient Euclidean space.
\end{lemma}
%For Riemannian optimization, in each step, the need arises to map points from a tangent space to the manifold in order to generate the new iterate. 
%%The early methods used to tackle with the problem exploit the underlying geometry of that optimization problem by relying on mainstream differential-geometric concepts, such as
% The Riemannian exponential map is used to tackle with the problem, which moves a point on the manifold along the Riemannian geodesic locally defined by a vector in the tangent space. 
%However,  it is defined as the solution of a nonlinear ordinary differential equation, resulting that computing is prohibitively expensive. As shown in \cite{absilProjectionlikeRetractionsMatrix2012},  retractions, which  approximates the exponential map to the first order, is sufficient for many convergence results.
Retraction provides a method to map the tangent vector to the next iterate on the manifold.
\begin{definition}(cf. \cite[Def. 4.1.1]{absil2009optimization})\label{def:retraction}
	A retraction on a manifold $\mathscr{M}$ is a smooth mapping $R$ from the tangent bundle $T\mathscr{M} $
	onto $\mathscr{M}$ with the following properties. Let $R_x$ denote the restriction of $R$ to $T_x\mathscr{M}$, $(i)~R_x(0_x) = x$, where $0_x$ denotes the zero element of $T_x\mathscr{M} $, and $(ii)~\mathrm{D} R_x(0_x):T_x\mathscr{M}\mapsto T_x\mathscr{M}$
	is the identity map: $\mathrm{D} R_x(0_x)[v] = v$.
	% 	\begin{itemize}
		% 		\item [(1)] $R_x(0_x) = x$, where $0_x$ denotes the zero element of $T_x\mathscr{M} $.
		% 		\item [(2)] $D R_x(0_x):T_x\mathscr{M}\mapsto T_x\mathscr{M}$
		% 		is the identity map: $D R_x(0_x)[v] = v$.
		% 	\end{itemize}
\end{definition}
For the embedded submanifold of a vector space, there is a simple way to construct retractions,  as specified in the following lemma.
\begin{lemma}(cf. \cite[Prop. 4.1.2]{absil2009optimization})\label{lemm:retraciton}
	Let $\mathscr{M}$ be an embedded manifold of a vector space $\mathscr{E}$ and let $\mathscr{N}$ be an
	abstract manifold such that $\dim{\mathscr M}+\dim{\mathscr N}=\dim{\mathscr E}$. Assume that
	there is a diffeomorphism
	$
	\phi :\mathscr{M} \times \mathscr{N}\rightarrow \mathscr{E}_{*}:(F,G)\mapsto\phi(F,G),
	$
	where $\mathscr{E}_{*}$ is an open subset of $\mathscr{E}$(thus  $\mathscr{E}_{*}$ is an open submanifold of $\mathscr{E}$), with a neutral element $I \in \mathscr{N}$ satisfying 
	$
	\phi(F,I) = F,~~ \forall F \in \mathscr{M}.
	$
	Then the retraction is 
	$
	R_{X}(\xi) = \pi_1\bigxiaokuohao{\phi^{-1}(X+\xi)},
	$
	where $	\pi_1:\mathscr{M} \times \mathscr{N}\rightarrow \mathscr{M}:(F,G)\mapsto F$ is the projection onto the first
	component, defines a retraction on $\mathscr{M}$.
\end{lemma}
To  compare tangent vectors at distinct points on the manifold, the vector transport upon retraction $R$
% , as specified in the following definition,  
gives us a way to transport a tangent vector $\xi \in T_x\mathscr{M}$ to the tangent space
$T_{R_x(\eta)}\mathscr{M} $ for some $\eta \in T_x\mathscr{M} $.
% and some retraction $R$.
\begin{definition}(cf. \cite[Def. 8.1.1]{absil2009optimization})\label{def:vector_transport}
	A vector transport $\MC{T}:	T\mathscr{M} \oplus T\mathscr{M} \rightarrow T\mathscr{M} :(\eta_x ,\xi_x ) \mapsto \mathcal{T}_{\eta_x}\xi_x$ associated with a retraction $R$ is a smooth mapping
	% 	\begin{equation*}
		% 		T\mathscr{M} \oplus T\mathscr{M} \rightarrow T\mathscr{M} :(\eta_x ,\xi_x ) \mapsto \mathcal{T}_{\eta_x}\xi_x,
		% 	\end{equation*}
	satisfying the following properties for all $x\in\mathscr{M}$:  $(i)~\MC{T}_{\eta_x}\xi_x\in T_{R_x(\eta_x)}\mathscr{M}$,   $(ii)~\MC{T}_{0_{x}}\xi_x = \xi_x$ for all $\xi_x\in T_x\mathscr{M}$, and $(iii)~\MC{T}_{\eta_{x}}(a\xi_x+b\zeta) = a\mathcal{T}_{\eta_{x}}\xi_x+b\mathcal{T}_{\eta_{x}}\zeta$.
	% 	\begin{itemize}
		% 		\item [(1)] 
		% %		(Associated retraction)
		% % 		 There exists a retraction $R$ associated with $\MC{T}$, i.e.,
		% %		such that, for all $\eta_x ,\xi_x$, it holds that
		% 		 $\MC{T}_{\eta_x}\xi_x\in T_{R_x(\eta_x)}\mathscr{M}$.
		% 		\item [(2)] 
		% %		(Consistency)
		%  $\MC{T}_{0_{x}}\xi_x = \xi_x$ for all $\xi_x\in T_x\mathscr{M}$.
		% 		\item [(3)] 
		% 		 $\MC{T}_{\eta_{x}}(a\xi_x+b\zeta) = a\mathcal{T}_{\eta_{x}}\xi_x+b\mathcal{T}_{\eta_{x}}\zeta$
		% %		(Linearity) 
		% %		The mapping $\mathcal T_{\eta_x}:T_x\MC{M} \rightarrow T_{R_x(\eta_x)}\MC{M},\xi_x\mapsto \mathcal T_{\eta_x}\xi_x$ is linear.
		% 	\end{itemize}
	Vector transport by differentiated retraction is defined as 
	\begin{equation}\label{eq:vec transport}
		\mathcal{T}_{\eta_x}\xi_x := D R_{x}(\eta_x)[\xi_x] = \frac{\mathrm{d}}{\mathrm{d}t}R_{x}(\eta_x+t\xi_x)\big|_{t=0}.
	\end{equation}
\end{definition}
%\begin{figure}[H] 
%	\centering
%	\includegraphics[width=3.5in]{vector_transport.png} 
%	\caption{The vector transport $\mathcal{T}_{\eta_x}(\xi_x)$}\label{fig:1} 
%\end{figure}
\begin{lemma}(cf. \cite[Sect. 8.1.3]{absil2009optimization})\label{lemma:vector transport}
	A vector transport on $\mathscr{M}$
	associated with a retraction $R$ is given by the orthogonal projection onto the tangent space, i.e.,
	$
	\mathcal{T}_{\eta_x}{\xi_x} = \mathbf{P}_{R_x(\eta_x)}\xi_x,
	$
	where $\mathbf{P}_x$ denotes the orthogonal projector onto $T_x\mathscr{M}$.
\end{lemma}
\begin{definition}
	A vector transport $\mathcal{T}$ is called isometric if it satisfies
	$
	\left\langle\mathcal{T}_{\eta}(\xi), \mathcal{T}_{\eta}(\xi)\right\rangle_{R_{x}(\eta)}=\langle\xi, \xi\rangle_{x}
	$
	for all $\eta, \xi \in T_{x} \mathscr{M}$, where $R$ is the retraction associated with $\mathcal{T}$.
\end{definition}

\section{Computation over Tensor Stiefel Manifold}  \label{sec:tensor_stiefel_manifold}
In this Section, we first prove that $\st{n,p,l}$ with $n\geq p$ is a manifold and establish the parmeterized form of its tangent space in Sect. \ref{subsec:manifold setting}. Next, in Sect. \ref{subsection:grad Hess} we further show that $\st{n,p,l}$ is a Riemannian manifold and obtain the Riemannian gradient and Riemannian Hessian by the orthogonal projector operator. In Sect. \ref{sec:retraction}, the geodesic and some retractions based on various tensor decompositions are derived. Finally, we construct various vector transports based on the projector operator and by differentiated the retration in Sec. \ref{subsec:vec trans}.
\subsection{Manifold setting}\label{subsec:manifold setting}
The following theorem shows that the set $\st{n,p,l}$ with $n\geq p$  is indeed a manifold.
\begin{theorem}[Manifold]\label{Th:tensor_stiefel_manifold}
	For $n\geq p$, let $\st{n,p,l} =\bigdakuohao{\MC{X}\in \mathbb{R}^{n\times p \times l }| \MC{X}^{\top}\ast\MC{X} = \mathcal{I}}$. Then $\st{n,p,l} $ is an embedded  submanifold of ${\mathbb{R}^{n\times p \times l }}$ of dimension
	\begin{equation*}
		\mathrm{dim}(\st{n,p,l} )
		%		 = \operatorname{dim}(\mathbb{R}^{n\times p \times l })-\operatorname{dim}(\mathbb{S}^{p\times p\times l}) 
		= p\bigxiaokuohao{nl-\frac{pl}{2}-\frac{1}{2^{|\sin(\frac{l\pi}{2})|}}}.
	\end{equation*}
\end{theorem}
\begin{proof}
	It follows from \cite[Prop. 2]{zheng2021t} that
	\begin{equation}\label{eq:dim of sym}
		\operatorname{dim}(\TS{p}{p}{l}) = \begin{cases}
			p(pk+1),& l = 2k \\ 
			p(p(k+\frac{1}{2})+\frac{1}{2}),&  l = 2k+1
		\end{cases} = p\bigxiaokuohao{p\frac{l}{2}+\frac{1}{2^{|\sin(\frac{l\pi}{2})|}}}.
	\end{equation}
	Consider the following function:
	\begin{equation*}
		h:\mathbb{R}^{n\times p \times l}\rightarrow \TS{p}{p}{l}:\MC{X}\mapsto h(\MC{X})= \MC{X}^{\top}\ast\MC{X} - \mathcal{I}.
	\end{equation*}
	%	We claim that $h$ is a local defining function for $\st{n,p,l} $. Indeed, 
	Obviously, $h$ is smooth and the zero level set $h^{-1}(\mathcal{O}) = \st{n,p,l} $. Note that $\TS{p}{p}{l}$ is a vector space and so a linear manifold.  Since $h$ is a smooth map of manifolds, \cite[Thm.  9.9]{tuIntroductionManifolds2011} shows that if the level set $h^{-1}(\mathcal{O})$ is a regular level set of $h$, then  $\st{n,p,l} $ is an embedded submanifold 
	% 	(or called regular submanifold) 
	of $\mathbb{R}^{n\times p \times l }$ of dimension equal to $\operatorname{dim}(\mathbb{R}^{n\times p \times l })-\operatorname{dim}(\TS{p}{p}{l})$. Recall that   $h^{-1}(\mathcal{O})$ is a regular level set of $h$ if only if the differential of $h$
	is surjective (cf. \cite[Sect. 9.2]{tuIntroductionManifolds2011}). 
	%	Thus it remains to check that the differential of $h$ has rank $d$ for all $\MC{X}\in \st{n,p,l} $. 
	To this end, for all $\MC{X}\in \st{n,p,l} $, consider $\mathrm{D}h(\MC{X}):\mathbb{R}^{n\times p \times l}\rightarrow \TS{p}{p}{l}:$
	\begin{eqnarray*}
		\mathrm{D}h(\MC{X})[\mathcal{V}] &=& \lim\limits_{t\rightarrow0} \frac{h(\MC{X}+t\mathcal{V})-h(\MC{X})}{t}\\
		&=& \lim\limits_{t\rightarrow0} \frac{(\MC{X}+t\mathcal{V})^{\top}\ast(\MC{X}+t\mathcal{V})-\MC{X}^{\top}\ast\MC{X}}{t}\\
		&=& \MC{X}^{\top}\ast\mathcal{V} + \mathcal{V}^{\top}\ast\MC{X}.
	\end{eqnarray*}
	%	To show $Dh(\MC{X})$ has rank $d$, we must be show its range is linear subspace of dimension $d$. Since the codomain $\sym{\mathbb{R}^{p\times p \times l }}$ has dimension $d$, we must show that the range of $Dh(\MC{X})$ is all of $\sym{\mathbb{R}^{p\times p \times l }}$, that is, $Dh(\MC{X})$ is surjective. 
	For $\mathcal{V} = \frac{1}{2}\MC{X}\ast\mathcal{A}$ with $\mathcal{A} \in \TS{p}{p}{l}$ arbitrary, there holds
	\begin{equation*}
		\mathrm{D}h(\MC{X})[\mathcal{V}] = \frac{1}{2} \MC{X}^{\top}\ast\MC{X}\ast\mathcal{A} + \frac{1}{2}\mathcal{A}^{\top}\ast\MC{X}^{\top}\ast\MC{X} = \mathcal{A}.
	\end{equation*}
	In other words, for any tensor $\mathcal{A} \in \TS{p}{p}{l}$, there exists a tensor $\mathcal{V}\in \mathbb{R}^{n\times p \times l }$ such that $\mathrm{D}h(\MC{X})[\mathcal{V}] = \mathcal{A}$. This confirms that the range of $\mathrm{D}h(\MC{X})$ is  $\TS{p}{p}{l}$. Thus, $h^{-1}(\mathcal{O}) = \st{n,p,l} $ is the  a nonempty regular  level set, making $\st{n,p,l} $ an embedded submanifold of ${\mathbb{R}^{n\times p \times l }}$ of dimension
	\begin{equation*}
		\operatorname{dim}(\st{n,p,l} ) = \operatorname{dim}(\mathbb{R}^{n\times p \times l })-\operatorname{dim}(\TS{p}{p}{l}) = p\bigxiaokuohao{nl-\frac{pl}{2}-\frac{1}{2^{|\sin(\frac{l\pi}{2})|}}}.
	\end{equation*}
\end{proof}
To apply optimization algorithms based on line search, we must consider a direction on a manifold, that is the tangent vector. 
%The set of all tangent vectors at $\MC{X}$ is called the tangent space $T_{\MC{X}}\st{n,p,l} $ and 
%the union of all tangent spaces is called the tangent bundle of the manifold $T\st{n,p,l} $. 
The tangent space of $\st{n,p,l}$  can be
parametrized as follows.
\begin{theorem}[Tangent space]\label{th:dim}
	The tangent space $T_{\MC{X}}\st{n,p,l} $ is a subspace of $\mathbb{R}^{n\times p \times l }$:
	\begin{equation*}	T_{\MC{X}}\st{n,p,l}  
		%		= \operatorname{ker}( Dh(\MC{X})) 
		= \bigdakuohao{ \MC{X}\ast \mathcal{W}+\MC{X}_{\perp }\ast \mathcal{B}\in\mathbb{R}^{n\times p \times l }\bigg|\mathcal{W}\in \TSkew{p}{p}{l},~\mathcal{B}\in \mathbb{R}^{(n-p)\times p \times l}},
	\end{equation*}
	where ${\MC{X}_{\perp }} = L^{-1}\bigxiaokuohao{\hat{\MC{X}}_{\perp }},~{\hat{\MC{X}}} = L\bigxiaokuohao{{\MC{X}}}$ and  
	%$\hat{\mathcal {Q}}_{\perp }(:,:,i) = \hat{X}^{(i)}_{\perp }$ are orthogonal complementary matrix of $\hat{\mathcal {Q}}(:,:,i) =\hat{X}^{(i)},i \in [l]$.  
	$ \hat{X}^{(i)}_{\perp }\in \mathbb{C}^{n\times (n-p)}$ is any matrix such that $\operatorname{span}(\hat{X}^{(i)}_{\perp})=\{\hat{X}^{(i)}_{\perp}\alpha\big|\alpha\in \mathbb{C}^{n-p}\}$ is the or­thogonal complement of $\operatorname{span}(\hat{X}^{(i)})=\{\hat{X}^{(i)}\beta\big|\beta\in \mathbb{C}^{p}\}$.
\end{theorem}
\begin{proof}
	The rank of $h$ at $\MC{X}$
	is defined as the dimension of the range of  the differential  $\mathrm{D}h(\MC{X})$ (cf. \cite[Sect. 8.9]{tuIntroductionManifolds2011}), which is equal to  $\text{dim}(\TS{p}{p}{l})$ given in \eqref{eq:dim of sym}.
	Since $\st{n,p,l} $ is defined as a level set of a constant-rank
	function  $	h:\mathbb{R}^{n\times p \times l}\rightarrow \TS{p}{p}{l}:\MC{X}\mapsto h(\MC{X})= \MC{X}^{\top}\ast\MC{X} - \mathcal{I}$, it follows from \cite[Sect. 3.5.7]{absil2009optimization} that 
	\begin{equation*}
		T_{\MC{X}}\st{n,p,l}  = \operatorname{ker}( Dh(\MC{X})) = \bigdakuohao{\mathcal{V}\in\mathbb{R}^{n\times p \times l }|\MC{X}^{\top}\ast\mathcal{V} + \mathcal{V}^{\top}\ast\MC{X} = \mathcal{O}}.
	\end{equation*} 
	Using  Remark \ref{remark:DFT equality}, we get
	\begin{equation}\label{eq:4.2.1}
		\MC{X}^{\top}\ast\mathcal{V} + \mathcal{V}^{\top}\ast\MC{X} = \mathcal {O} \Leftrightarrow (\hat{X}^{(i)})^H\hat{V}^{(i)} + (\hat{V}^{(i)})^H\hat{X}^{(i)} = O_{p}, i\in [l].
	\end{equation}
	Similarly, $	\MC{X}^{\top}\ast\MC{X}
	% = \operatorname{fold}(\operatorname{bcirc}(\MC{X})^{\top}\cdot \operatorname{unfold}(\MC{X})) 
	= \mathcal{I}$ leads to $	(\hat{X}^{(i)})^H\hat{X}^{(i)} = I_p,i\in[l]$ which means $\hat{X}^{(i)}$ is full rank $p$. Let $\hat{X}^{(i)}_{\perp}$ be any $n\times (n-p)$ complex matrix such that $\operatorname{span}(\hat{X}^{(i)}_{\perp})=\{\hat{X}^{(i)}_{\perp}\alpha\big|\alpha\in \mathbb{C}^{n-p}\}$ is the or­thogonal complement of $\operatorname{span}(\hat{X}^{(i)})=\{\hat{X}^{(i)}\beta\big|\beta\in \mathbb{C}^{p}\}$. By the definition, $\innerprod{\hat{X}^{(i)}\beta}{\hat{X}^{(i)}_{\perp}\alpha} = 0$ for any $\alpha\in \mathbb{C}^{n-p}$ and $\beta\in \mathbb{C}^{p}$,  leading to 
	\begin{equation}\label{eq:orthogonal}
		(\hat{X}^{(i)})^H\hat{X}^{(i)}_{\perp } = O_{p\times (n-p)}.
	\end{equation}
	Since $[\hat{X}^{(i)}, \hat{X}^{(i)}_{\perp }]\in \mathbb{C}^{n\times n}$ is  invertible, any matrix $\hat{V}^{(i)}$ can be
	written as
	\begin{equation}\label{eq:4.2.3}
		\hat{V}^{(i)} =  [\hat{X}^{(i)}, \hat{X}^{(i)}_{\perp }]\begin{bmatrix}
			\hat W^{(i)}\\ 
			\hat B^{(i)}
		\end{bmatrix} = \hat{X}^{(i)}\hat W^{(i)}+\hat{X}^{(i)}_{\perp}\hat B^{(i)},
	\end{equation}
	for a unique choice of $\hat W^{(i)} \in \mathbb{C}^{p\times p}$ and $\hat B^{(i)}\in \mathbb{C}^{(n-p)\times p}$. 
	Combing equation \eqref{eq:4.2.1} and \eqref{eq:4.2.3} yields
	\begin{equation} 
		O_p = (\hat{X}^{(i)})^H(\hat{X}^{(i)}\hat W^{(i)}+\hat{X}^{(i)}_{\perp }\hat B^{(i)})+	(\hat{X}^{(i)}\hat W^{(i)}+\hat{X}^{(i)}_{\perp }\hat B^{(i)})^H\hat{X}^{(i)} = 	\hat W^{(i)}+	(\hat W^{(i)})^H, i\in [l],
	\end{equation}
	which together with Remark \ref{remark:DFT equality} gives \begin{equation}\label{eq:proof:tangent_space:1}
		\mathcal{W}+\mathcal{W}^{\top} =  \mathcal{O}.\end{equation}
	
	Remark \ref{remark:DFT equality} shows that  \eqref{eq:orthogonal} is equivalent to  $$ 
	%	({\MC{X}})^{\top}\ast{\MC{X}}= \mathcal {I},~({\MC{X}}_{\perp })^{\top}\ast{\MC{X}}_{\perp } = \mathcal I_{(n-p)(n-p)l},~
	{\MC{X}}^{\top}\ast {\MC{X}}_{\perp } = \mathcal{O},
	$$
	where ${\MC{X}_{\perp }} = L^{-1}\bigxiaokuohao{\hat{\MC{X}}_{\perp }},~{\hat{\MC{X}}} = L\bigxiaokuohao{{\MC{X}}}$ and  $ \hat{X}^{(i)}_{\perp }\in \mathbb{C}^{n\times (n-p)}$ are any matrix such that $\operatorname{span}(\hat{X}^{(i)}_{\perp})=\{\hat{X}^{(i)}_{\perp}\alpha\big|\alpha\in \mathbb{C}^{n-p}\}$ is the or­thogonal complement of $\operatorname{span}(\hat{X}^{(i)})=\{\hat{X}^{(i)}\beta\big|\beta\in \mathbb{C}^{p}\}$.
	Thus it holds that $\mathcal{V} = \MC{X}\ast \mathcal{W}+\MC{X}_{\perp }\ast \mathcal{B}\in\mathbb{R}^{n\times p \times l }$ where $\mathcal{W}\in \TSkew{p}{p}{l}$ by \eqref{eq:proof:tangent_space:1}, and so
	\begin{equation*}	T_{\MC{X}}\st{n,p,l}  
		= \operatorname{ker}( Dh(\MC{X})) 
		= \bigdakuohao{ \MC{X}\ast \mathcal{W}+\MC{X}_{\perp }\ast \mathcal{B}\in\mathbb{R}^{n\times p \times l }\bigg|\mathcal{W}\in \TSkew{p}{p}{l},~\mathcal{B}\in \mathbb{R}^{(n-p)\times p \times l}}.
	\end{equation*}
\end{proof}
\subsection{Riemannian metric, gradient and Hessian on $\st{n,p,l}$}\label{subsection:grad Hess}
For the smooth manifold $\st{n,p,l} $, the Riemannian metric $g_\MC{X}$ is defined as
\begin{equation}\label{eq:inner product}
	g_\MC{X}(\mathcal{V},\mathcal{U}) := 	\langle\mathcal{V},\mathcal{U}	\rangle_{\MC{X}} = \sum_{i = 1}^{n}\sum_{j = 1}^{p}\sum_{k = 1}^{l}v_{ijk}\cdot u_{ijk}~~\operatorname{with}~\MC{X}\in \st{n,p,l} ~\operatorname{and}~\mathcal{V},\mathcal{U}\in T_{\MC{X}}\st{n,p,l} ,
\end{equation}
which is indeed the Euclidean metric from the embedded space induced by the Frobenius norm.
Equipped with this metric,
% smoothly varied inner product $g_{\MC{X}}(\cdot, \cdot):=\langle\cdot, \cdot\rangle_{\MC{X}}$ on the tangent space, then  
$(\st{n,p,l} ,g_\MC{X})$ becomes a Riemannian manifold. In the sequel, we write $\st{n,p,l} $
as a Riemannian manifold for simplicity. The norm induced by the Riemannian metric $g_\MC{X}$ is defined as $||\MC{V}||_{\MC{X}} := 	\langle\mathcal{V},\mathcal{V}	\rangle_{\MC{X}}$.
%According to \cite[Sect. 3.6.1]{absil2009optimization}, since $\st{n,p,l} $ is an embedded smooth manifold of ${\mathbb{R}^{n\times p \times l }}$, then
%the Riemannian gradient can be seen as the orthogonal projection onto the tangent space of the gradient of ${\mathbb{R}^{n\times p \times l }}$.
\begin{theorem}[Orthogonal projector operator]
	The orthogonal projector operator from Euclidean space $\T{n}{p}{l}$ to  $T_{\MC{X}}\st{n,p,l}$ is 
	\begin{equation}\label{eq:proj}
		\mathbf{P}_{\MC{X}}(\mathcal{U}) = \mathcal{U} - \MC{X}\ast\frac{1}{2}\bigxiaokuohao{\MC{X}^{\top}\ast\mathcal{U}+\mathcal{U}^{\top}\ast\MC{X}} = \mathcal{U} -\MC{X}\ast\operatorname{sym}(\MC{X}^{\top}\ast\mathcal{U})	= (\mathcal{I} - \MC{X}\ast\MC{X}^{\top})\ast\mathcal{U}+\MC{X}\ast\operatorname{skew}(\MC{X}^{\top}\ast\mathcal{U}),
	\end{equation}
	where $\operatorname{sym}(\mathcal{A}) = \frac{\mathcal{A}+\mathcal{A}^{\top}}{2}$ and $\operatorname{skew}(\mathcal{A}) = \frac{\mathcal{A}-\mathcal{A}^{\top}}{2}$.
\end{theorem}
\begin{proof}
	From Lemma \ref{Lemma:grad Hess proj}, there holds $\innerprod{\MC{U}-\mathbf{P}_{\MC{X}}(\mathcal{U})}{\MC{V}} = 0$ for all $\MC{V}\in T_{\MC{X}}\st{n,p,l}$ and $\MC{U} \in \T{n}{p}{l}$.
	Hence we consider the  normal space of $T_{\MC{X}}\st{n,p,l} $ as follows:
	\begin{eqnarray*}
		N_{\MC{X}}\st{n,p,l}  
		%		&=& (T_{\MC{X}}\st{n,p,l} )^{\top}\\
		&=& \bigdakuohao{\mathcal{N}\in \mathbb{R}^{n\times p \times l}|	\langle\mathcal{N},\mathcal{V}	\rangle = \mathcal{O}~ \operatorname{for~all}~ \mathcal{V}\in T_{\MC{X}}\st{n,p,l} }\\
		&=& \bigdakuohao{\mathcal{N}\in \mathbb{R}^{n\times p \times l}\bigg|	\langle\mathcal{N},\MC{X}\ast \mathcal{W}+\MC{X}_{\perp }\ast \mathcal{B}	\rangle= \mathcal{O},
			%			\mathcal{W}(:,:,i)=
			\mathcal{W}\in \TSkew{p}{p}{l},\mathcal{B}\in \mathbb{R}^{(n-p)\times p\times l}},
	\end{eqnarray*}
	where we use the parameter expression of $	T_{\MC{X}}\st{n,p,l}$ in Theorem \ref{th:dim}. Since $\hat{X}^{(i)}_{\perp }\in \mathbb{C}^{n\times (n-p)}$ is any matrix such that $\operatorname{span}(\hat{X}^{(i)}_{\perp})=\{\hat{X}^{(i)}_{\perp}\alpha\big|\alpha\in \mathbb{C}^{n-p}\}$ is the or­thogonal complement of $\operatorname{span}(\hat{X}^{(i)})=\{\hat{X}^{(i)}\beta\big|\beta\in \mathbb{C}^{p}\}$,
	we can expand the DFT of normal vectors as ${\hat N}^{(i)} = \hat{X}^{(i)}\hat A^{(i)}+\hat{X}^{(i)}_{\perp}\hat C^{(i)}$  with $\hat A^{(i)} \in \mathbb{C}^{p\times p}$ and $\hat C^{(i)}\in \mathbb{C}^{(n-p)\times p}$, which together with Remark \ref{remark:DFT equality} leads to $\mathcal{N} = \MC{X}\ast \mathcal{A}+\MC{X}_{\perp }\ast \mathcal{C}$. Then it holds that
	\begin{eqnarray*}
		N_{\MC{X}}\st{n,p,l}  
		&=& \bigdakuohao{\mathcal{N}\in \mathbb{R}^{n\times p \times l}\bigg|	\langle\MC{X}\ast \mathcal{A}+\MC{X}_{\perp }\ast \mathcal{C},\MC{X}\ast \mathcal{W}+\MC{X}_{\perp }\ast \mathcal{B}	\rangle= 0,~\operatorname{for~ all}~
			%			\mathcal{W}(:,:,i)=
			\mathcal{W}\in \TSkew{p}{p}{l},\mathcal{B}\in \mathbb{R}^{(n-p)\times p\times l}}\\
		&=& \bigdakuohao{\mathcal{N}\in \mathbb{R}^{n\times p \times l}\bigg|	\langle \MC{X}^{\top}\ast \MC{X}\ast \mathcal{A}, \mathcal{W}	\rangle+	\langle\MC{X}_{\perp }^{\top}\ast \MC{X}_{\perp }\ast\mathcal{C}, \mathcal{B}	\rangle= 0,~\operatorname{for~all}~	\mathcal{W}\in \TSkew{p}{p}{l},\mathcal{B}\in \mathbb{R}^{(n-p)\times p\times l}}\\
		&=& \bigdakuohao{\mathcal{N}\in \mathbb{R}^{n\times p \times l}\bigg|	\langle\mathcal{A}, \mathcal{W}	\rangle=0 ~\operatorname{and}~	\langle\mathcal{C}, \mathcal{B}	\rangle=0,~\operatorname{for~ all}~	\mathcal{W}\in \TSkew{p}{p}{l},\mathcal{B}\in \mathbb{R}^{(n-p)\times p\times l}}\\
		%			&=& \bigdakuohao{\mathcal{U}\in \mathbb{R}^{n\times p \times l}\bigg|	\langle\mathcal{\hat A}, \mathcal{\hat W}	\rangle=0,~\operatorname{for~ all}~\hat W^{(i)} \in \operatorname{Skew}(\mathbb{C}^{p\times p}),\mathcal{C}=\mathcal{O}_{(n-p)\times p\times l}}\\
		&=& \bigdakuohao{\MC{X}\ast\mathcal{A}\in \mathbb{R}^{n\times p \times l}\bigg|
			%			\mathcal{A}(:,:,i)= 
			\mathcal{A}\in \TS{p}{p}{l}
			%			&=& \bigdakuohao{\MC{X}\ast\mathcal{A}\in \mathbb{R}^{n\times p \times l}\bigg|
				%			 A^{(i)}\in \operatorname{Sym}({\mathbb{C}^{p\times p}}),i \in [l]
			},
		\end{eqnarray*}
		where the last equation comes from Lemma \ref{lemma:orthogonal complement}. 
		%	 and the last equation comes from the item 2 of Proposition \ref{prop:bcirc_properties}
		%	 .
		Thus, the orthogonal projector obeys:
		\begin{equation*}
			\mathbf{P}_{\MC{X}}(\mathcal{U}) = \mathcal{U} - \mathcal{N} = \mathcal{U} - \MC{X}\ast\mathcal{A}
		\end{equation*}
		for some symmetric tensor $
		\mathcal{A}\in \TS{p}{p}{l}$. Since the projected tensor must lie in $	T_{\MC{X}}\st{n,p,l}  $, we have
		\begin{equation*}
			\MC{X}^{\top}\ast\bigxiaokuohao{\mathcal{U} - \MC{X}\ast\mathcal{A}}+	\bigxiaokuohao{\mathcal{U} - \MC{X}\ast\mathcal{A}}^{\top}\ast\MC{X} = \mathcal{O},
		\end{equation*}
		that is, $ \MC{X}^{\top}\ast\mathcal{U}+\mathcal{U}^{\top}\ast\MC{X}  = \mathcal{A}+\mathcal{A}^{\top} = 2\mathcal{A}$. 
		Thus,  $\mathcal{A} = \frac{1}{2}\bigxiaokuohao{\MC{X}^{\top}\ast\mathcal{U}+\mathcal{U}^{\top}\ast\MC{X}}$ and 
		\begin{equation*}
			\mathbf{P}_{\MC{X}}(\mathcal{U}) = \mathcal{U} - \MC{X}\ast\operatorname{sym}(\MC{X}^{\top}\ast\mathcal{U})
			%		= \mathcal{U} -\MC{X}\ast\operatorname{sym}(\MC{X}^{\top}\ast\mathcal{U}) 
			= (\mathcal{I} - \MC{X}\ast\MC{X}^{\top})\ast\mathcal{U}+\MC{X}\ast\operatorname{skew}(\MC{X}^{\top}\ast\mathcal{U}).
		\end{equation*}
		%	where $\operatorname{sym}(\mathcal{A}) = \frac{\mathcal{A}+\mathcal{A}^{\top}}{2}$ and $\operatorname{skew}(\mathcal{A}) = \frac{\mathcal{A}-\mathcal{A}^{\top}}{2}$.
	\end{proof}
	Lemma \ref{Lemma:grad Hess proj} tells us that the orthogonal projector yields a convenient formula for the Riemannian gradient and Riemannian Hessian.
	\begin{proposition}[Gradient]
		The gradient of smooth functions  $f$ defined on $\st{n,p,l} $ is
		\begin{equation*}
			\operatorname{grad}f(\MC{X}) = \mathbf{P}_{\MC{X}}(	\operatorname{grad}\bar{f}(\MC{X})) = \operatorname{grad}\bar{f}(\MC{X})-\MC{X}\ast\operatorname{sym}(\MC{X}^{\top}\ast\operatorname{grad}\bar{f}(\MC{X})),
		\end{equation*}
		where $\bar{f}$ defined on $\T{n}{p}{l}$ coincides with $f$ on $\st{n,p,l}: f = \bar{f}\big|_{\st{n,p,l}}$.
	\end{proposition}
	The definitions and computation of the Euclidean gradient $\operatorname{grad}f(\MC{X})$ above and the Euclidean directional derivative $\mathrm{D}f(\MC{X})[\MC{H}]$ that will appear in the following are given in Appendix \ref{apx:18}. 
	\begin{theorem}[Hessian]
		The Riemannian Hessian of a real-valued function $f$ at a point $\MC{X}$ on  $\st{n,p,l} $ is
		\begin{equation*}
			\operatorname{Hess}f(\MC{X})[\mathcal{V}] = \mathbf{P}_{\MC{X}}\bigxiaokuohao{\operatorname{Hess}\bar{f}(\MC{X})[\mathcal{V}]}-\mathcal{V}\ast \operatorname{sym}(\MC{X}^{\top}\ast \operatorname{grad}\bar{f}(\MC{X})).
		\end{equation*}
	\end{theorem}
	
	\begin{proof}
		Let $\bar{G}(\MC{X}) =\operatorname{grad}\bar{f}(\MC{X})-\MC{X}\ast\operatorname{sym}(\MC{X}^{\top}\ast\operatorname{grad}\bar{f}(\MC{X}))$
		denote a smooth extension of  $\operatorname{grad}f(\MC{X})$ to a neighborhood of $\st{n,p,l} $ in $\T{n}{p}{l}$. 
		\begin{eqnarray*}
			D\bar{G}(\MC{X})[\mathcal{V}] &=& \lim_{t\rightarrow 0}\frac{1}{t}\bigzhongkuohao{\bar{G}(\MC{X}+t\mathcal{V})-\bar{G}(\MC{X})}\\
			&=&\operatorname{Hess}\bar{f}(\MC{X})[\mathcal{V}] - \lim_{t\rightarrow 0}\frac{1}{t}\bigzhongkuohao{
				(\MC{X}+t\MC{V})\ast\operatorname{sym}\bigxiaokuohao{(\MC{X}+t\MC{V})^{\top}\ast\operatorname{grad}\bar{f}(\MC{X}+t\MC{V})}
				-\MC{X}\ast\operatorname{sym}(\MC{X}^{\top}\ast\operatorname{grad}\bar{f}(\MC{X}))}\\
			&=&\operatorname{Hess}\bar{f}(\MC{X})[\mathcal{V}] - \mathcal{V}\ast \operatorname{sym}\bigxiaokuohao{\MC{X}^{\top}\ast \operatorname{grad}\bar{f}(\MC{X})}
			\\
			&\;&
			-\MC{X}\ast\operatorname{sym}\bigxiaokuohao{ \lim_{t\rightarrow 0}\frac{1}{t}\bigzhongkuohao{(\MC{X}+t\MC{V})^{\top}\ast\operatorname{grad}\bar{f}(\MC{X}+t\MC{V})-\MC{X}^{\top}\ast\operatorname{grad}\bar{f}(\MC{X})}}\\
			&=&\operatorname{Hess}\bar{f}(\MC{X})[\mathcal{V}] - \mathcal{V}\ast \operatorname{sym}\bigxiaokuohao{\MC{X}^{\top}\ast \operatorname{grad}\bar{f}(\MC{X})}
			- \MC{X}\tprod{\operatorname{sym}\bigxiaokuohao{\mathcal{V}^{\top}\ast \operatorname{grad}\bar{f}(\MC{X})+\MC{X}^{\top}\operatorname{Hess}\bar{f}(\MC{X})[\mathcal{V}]}}.
		\end{eqnarray*}
		$\st{n,p,l} $ is a Riemannian submanifold of ${\mathbb{R}^{n\times p \times l }}$, and its   Riemannian Hessian is 
		\begin{eqnarray*}
			\operatorname{Hess}f(\MC{X})[\mathcal{V}] 
			% 		&=& \nabla_{\mathcal{V}}\operatorname{grad}f(\MC{X})\\
			&=&\mathbf{P}_{\MC{X}}\bigxiaokuohao{D\bar{G}(\MC{X})[\mathcal{V}]}\\
			&=& \mathbf{P}_{\MC{X}}\bigxiaokuohao{\operatorname{Hess}\bar{f}(\MC{X})[\mathcal{V}]-\mathcal{V}\ast \operatorname{sym}(\MC{X}^{\top}\ast \operatorname{grad}\bar{f}(\MC{X}))-\MC{X}\ast \mathcal{S}}\\
			&=&\mathbf{P}_{\MC{X}}\bigxiaokuohao{\operatorname{Hess}\bar{f}(\MC{X})[\mathcal{V}]}-\mathcal{V}\ast \operatorname{sym}(\MC{X}^{\top}\ast \operatorname{grad}\bar{f}(\MC{X})),
		\end{eqnarray*}
		where $\mathcal{S} = \operatorname{sym}\bigxiaokuohao{\mathcal{V}^{\top}\ast \operatorname{grad}\bar{f}(\MC{X})+\MC{X}^{\top}\operatorname{Hess}\bar{f}(\MC{X})[\mathcal{V}]}$ and $\MC{X}\tprod\MC{S}\in N_{\MC{X}}\st{n,p,l}$ vanishes through $\mathbf{P}_{\MC{X}}$.
	\end{proof}
	%For $\mathcal X \in\mathbb R^{n\times 1\times l}$, the definitions and computation of the Euclidean gradient $\operatorname{grad}f(\MC{X})$ and the Euclidean directional derivative $Df(\MC{X})[\MC{H}]$  in the sense of t-product  were given in \cite{zheng2021t}; For $\mathcal X \in\mathbb R^{n\times p\times l}$, the results of the corresponding promotions
	%are left to Appendix \ref{apx:18}.
	
	\subsection{Retraction on  $\st{n,p,l}$} \label{sec:retraction}
	According to Lemma \ref{lemm:retraciton}, we can  construct different retractions based on various tensor decompositions. Based on t-QR in Subsection \ref{subsec:t-QR}, we get the following retraction.
	\begin{theorem}[t-QR based retraction]\label{th:t-qr based retraction}
		The
		retraction on $\st{n,p,l}$ based on the t-QR decomposition is
		\begin{equation}\label{eq:t-qr based retraction}
			R_{\MC{X}}(\mathcal{V}) = qf\bigxiaokuohao{\MC{X}+\mathcal{V}},
		\end{equation}
		where $\MC{X}\in \st{n,p,l}, \MC{V}\in T_{\MC{X}}\st{n,p,l}$, and $qf(\mathcal A)$ denotes the $\mathcal Q$ factor of the t-QR decomposition of $\mathcal A \in L^{-1}\bigxiaokuohao{\mathbb{C}_*^{n\times p \times l}}$ as
		$\mathcal A = \mathcal Q\ast \mathcal R$, where  $\mathcal Q\in\st{n,p,l} $ and $\mathcal R \in L^{-1}\bigxiaokuohao{\TCUP{p}{p}{l}}$. 
	\end{theorem}
	\begin{proof}
		According to Theorem \ref{th:t_qr}, if $\mathcal A \in L^{-1}\bigxiaokuohao{\mathbb{C}_*^{n\times p \times l}}$, then t-QR decomposition of $\MC{A}$ is unique.
		Hence the inverse of t-QR decomposition of  $\mathcal A \in L^{-1}\bigxiaokuohao{\mathbb{C}_*^{n\times p \times l}}$ is a one-to-one mapping
		\begin{equation*}
			\phi:\st{n,p,l} \times L^{-1}\bigxiaokuohao{\TCUP{p}{p}{l}}
			%		\TUP{p}{p}{l} 
			\rightarrow  L^{-1}(\mathbb{C}_*^{n\times p \times l}):(\mathcal Q, \mathcal R)\mapsto \mathcal Q \ast \mathcal R = \mathcal A.
		\end{equation*}
		where 
		% $L^{-1}(\mathbb{C}_*^{n\times p \times l}) = \bigdakuohao{L^{-1}(\hat{\MC{A}})|\hat{\MC{A}}\in \TC{n}{p}{l}_*}$ and 
		$L^{-1}\bigxiaokuohao{\TCUP{p}{p}{l}}$ represents the set of $\MC{R}$ factor of t-QR decomposition of $\MC{A}\in \T{n}{p}{l}$.
		Since $L$ is a continuous function and $\mathbb{C}_*^{n\times p \times l}$ is an open set in $\mathbb{C}^{n\times p \times l}$, it follows that the preimage $L^{-1}(\mathbb{C}_*^{n\times p \times l})$ is open in $\T{n}{p}{l}$.
		Combing Theorem \ref{Th:tensor_stiefel_manifold} and Lemma \ref{lemma:upp_dim} gives rise to
		\begin{equation*}
			\operatorname{dim}(\st{n,p,l})+	\operatorname{dim}\bigxiaokuohao{L^{-1}\bigxiaokuohao{\TCUP{p}{p}{l}}}=\operatorname{dim}(\mathbb{R}^{n\times p \times l })-\operatorname{dim}( \TS{p}{p}{l})+\operatorname{dim}\bigxiaokuohao{L^{-1}\bigxiaokuohao{\TCUP{p}{p}{l}}}=\operatorname{dim}(\mathbb{R}^{n\times p \times l }).
		\end{equation*} The identity tensor is the neutral element. 
		% 	It follows from the existence and uniqueness properties of the t-QR decomposition for a tensor $\mathcal A \in L^{-1}\bigxiaokuohao{\mathbb{C}_*^{n\times p \times l}}$ that $\phi$ is bijective. 
		The mapping $\phi$ is smooth since it is the restriction of a
		smooth map (tensor product) to a submanifold. For $\phi^{-1}$, notice that
		its first tensor component $\mathcal Q$ is obtained by the Gram-Schmidt process and the inverse Fourier transform according to \cite[Alg. 1]{kilmer2013third}, which
		are $C^{\infty}$. Since the second component $\mathcal R$ is
		obtained as $\mathcal{Q}^{-1}\ast \mathcal{A}$, it follows that $\phi^{-1}$ is $C^{\infty}$. Thus $\phi$ is a diffeomorphism. From Lemma \ref{lemm:retraciton}, we have
		\begin{equation*}
			R_{\MC{X}}(\mathcal{V})
			= \pi_1\bigxiaokuohao{\phi^{-1}(\MC{X}+\mathcal{V})} 
			= qf(\MC{X}+\mathcal{V}),
		\end{equation*}
		where $qf(\mathcal A) = \pi_1 \circ \phi^{-1} $ denotes the mapping that sends a tensor to the $\mathcal{ Q}$ factor of its t-QR decomposition. This is
		well defined, i.e., the t-QR decomposition $\MC{X}+\MC{V}$ is unique. To see this,
		it follows from Remark \ref{remark:DFT equality} that 
		$$\MC{A}^{\top}\tprod\MC{A}=(\MC{X}+\mathcal{V})^{\top}\tprod(\MC{X}+\mathcal{V}) = \mathcal{I} + \mathcal{V}^{\top}\ast\mathcal{V}\Leftrightarrow (\hat{A}^{(i)})^H\hat{A}^{(i)}
		% 	= (\hat{X}^{(i)}+\hat{V}^{(i)})^{H}(\hat{X}^{(i)}+\hat{V}^{(i)}) 
		= {I}_p + (\hat{V}^{(i)})^H\hat{V}^{(i)}, i\in [l],$$
		confirming that ${\MC{A}}={\MC{X}}+{\MC{V}}\in L^{-1}\bigxiaokuohao{\TC{n}{p}{l}_*}$.
		% 	The mapping $qf$ can be computed using the
		% 	Gram-Schmidt orthonormalization on the Fourier domanin of $\mathcal{A}$ and then returning to the original domain.
	\end{proof}
	In the same vein, we get the following retraction based on t-PD in Subsection \ref{t-PD}.
	\begin{theorem}[t-PD based retraction]\label{t-pd based retraction}
		The
		retraction on $\st{n,p,l}$ based on   t-PD is
		\begin{equation}\label{eq:t-pd based retraction}
			R_{\MC{X}}(\mathcal{V}) = (\MC{X}+\mathcal{V})\ast(\mathcal{I}+\mathcal{V}^{\top}\ast\mathcal{V})^{-\frac{1}{2}},
		\end{equation}
		where $\MC{X}\in \st{n,p,l}$ and $\MC{V}\in T_{\MC{X}}\st{n,p,l}$.
	\end{theorem}
	\begin{proof}
		Proposition \ref{prop:sym equal} shows that if $\mathcal A \in L^{-1}\bigxiaokuohao{\mathbb{C}_*^{n\times p \times l}}$, then $\mathcal A^{\top}\tprod\mathcal A \in\TSPP{p}{p}{l}$.
		Hence the t-PD of  $\mathcal A \in L^{-1}\bigxiaokuohao{\mathbb{C}_*^{n\times p \times l}}$ is unique due to Theorem \ref{th:t_pd}. The inverse of t-PD is a mapping
		\begin{equation*}
			\phi:\st{n,p,l} \times \TSPP{p}{p}{l} \rightarrow  L^{-1}(\mathbb{C}_*^{n\times p \times l}):(\MC{X}, \mathcal H)\mapsto \MC{X}\ast \mathcal H = \mathcal A.
		\end{equation*}
		% 	where $L^{-1}(\mathbb{C}_*^{n\times p \times l}) = \bigdakuohao{L^{-1}(\hat{\MC{A}})|\hat{\MC{A}}\in \TC{n}{p}{l}_*}$.
		Since $L$ is a continuous function and $\mathbb{C}_*^{n\times p \times l}$ is a open set in $\mathbb{C}^{n\times p \times l}$, it follows that the preimage $L^{-1}(\mathbb{C}_*^{n\times p \times l})$ is open in $\T{n}{p}{l}$.
		Theorem \ref{Th:tensor_stiefel_manifold} shows that
		\begin{equation*}
			\operatorname{dim}(\st{n,p,l})+	\operatorname{dim}( \TSPP{p}{p}{l})=\operatorname{dim}(\mathbb{R}^{n\times p \times l })-\operatorname{dim}( \TSPP{p}{p}{l})+\operatorname{dim}( \TSPP{p}{p}{l})=\operatorname{dim}(\mathbb{R}^{n\times p \times l }).
		\end{equation*} The identity tensor is the neutral element. 
		By Proposition \ref{prop:t_pd_for_retraction}, if $\mathcal A^{\top}\tprod\mathcal A \in \TSPP{p}{p}{l}$, then  $\phi^{-1}(\mathcal A) = (\mathcal A\ast(\mathcal A^{\top}\ast \mathcal A)^{-\frac{1}{2}},( \mathcal A^{\top}\ast \mathcal A)^{\frac{1}{2}})$ which shows that $\phi$ is a diffeomorphism, and thus
		\begin{eqnarray*}
			R_{\MC{X}}(\mathcal{V})
			&=& \pi_1\bigxiaokuohao{\phi^{-1}(\MC{X}+\mathcal{V})} \\
			&=& (\MC{X}+\mathcal{V})\ast((\MC{X}+\mathcal{V})^{\top}\ast (\MC{X}+\mathcal{V}))^{-\frac{1}{2}} \\
			&=& (\MC{X}+\mathcal{V})\ast\bigxiaokuohao{\mathcal{I}+\mathcal{V}^{\top}\ast\mathcal{V}+\MC{X}^{\top}\ast \mathcal{V} +\mathcal{V}^{\top}\ast \MC{X}}^{-\frac{1}{2}}\\
			&=& (\MC{X}+\mathcal{V})\ast(\mathcal{I}+\mathcal{V}^{\top}\ast\mathcal{V})^{-\frac{1}{2}}.
		\end{eqnarray*}
		This is well defined. To see this, 
		it follows from \ref{remark:DFT equality} that 
		$$
		\MC{A}^{\top}\tprod\MC{A}=
		(\MC{X}+\mathcal{V})^{\top}\tprod(\MC{X}+\mathcal{V}) 
		= \mathcal{I} + \mathcal{V}^{\top}\ast\mathcal{V}
		% 	\in \TSPP{p}{p}{l}.
		\Leftrightarrow
		(\hat{A}^{(i)})^H\hat{A}^{(i)}
		% 	= (\hat{X}^{(i)}+\hat{V}^{(i)})^{H}(\hat{X}^{(i)}+\hat{V}^{(i)}) 
		= {I}_p + (\hat{V}^{(i)})^H\hat{V}^{(i)}, i\in [l].
		$$
		confirming that $\mathcal A ={\MC{X}}+{\MC{V}} \in L^{-1}\bigxiaokuohao{\mathbb{C}_*^{n\times p \times l}}$.
	\end{proof}
	We can also construct retractions directly from the Definition \ref{def:retraction}.
	\begin{theorem}[t-Cayley based retraction]\label{cayley based retraction}
		The retraction on $\st{n,p,l}$ based on  t-Cayley transform is
		\begin{equation}\label{eq:cayley based retraction}
			R_{\MC{X}}(\mathcal{V}) = \bigxiaokuohao{\mathcal{I} - \frac{1}{2}\mathcal{W}_{\mathcal{V}}}^{-1}\ast\bigxiaokuohao{\mathcal{I} + \frac{1}{2}\mathcal{W}_{\mathcal{V}}}\ast\MC{X},
		\end{equation}
		where $\MC{X}\in \st{n,p,l}, \MC{V}\in T_{\MC{X}}\st{n,p,l}$, $\mathcal{W}_{\mathcal{V}} = \MC{P}\ast \mathcal{V}\ast \MC{X}^{\top} - \MC{X}\ast \mathcal{V}^{\top}\ast\MC{P}\in \TSkew{n}{n}{l}$ and $\MC{P} = \mathcal{I} - \frac{1}{2}\MC{X}\ast\MC{X}^{\top}$.
	\end{theorem}
	\begin{proof}
		According to Lemma \ref{lem:positive definite}, for any $\mathcal{W}_{\mathcal{V}}\in \TSkew{n}{n}{l}$, it holds that
		$$\bigxiaokuohao{\mathcal{I} - \frac{1}{2}\mathcal{W}_{\mathcal{V}}}^{\top}\tprod\bigxiaokuohao{\mathcal{I} - \frac{1}{2}\mathcal{W}_{\mathcal{V}}} 
		%	= \bigxiaokuohao{\mathcal{I} + \frac{1}{2}\mathcal{W}_{\mathcal{V}}}\tprod\bigxiaokuohao{\mathcal{I} - \frac{1}{2}\mathcal{W}_{\mathcal{V}}}
		=\mathcal{I} + \frac{1}{4}\mathcal{W}_{\mathcal{V}}^{\top}\tprod\mathcal{W}_{\mathcal{V}}\in \TSPP{n}{n}{l},$$ namely, $\mathcal{I} - \frac{1}{2}\mathcal{W}_{\mathcal{V}}$ is invertible.
		Since $\mathcal{W}_{\mathcal{V}}$ is  skew-symmetric and $(\MC{I}-\MC{A})*(\MC{I}+\MC{B}) = (\MC{I}+\MC{B})*(\MC{I}-\MC{A})$ for all $\MC{A},\MC{B}\in \T{n}{p}{l}$, we obtain $$\bigxiaokuohao{\mathcal{I} - \frac{1}{2}\mathcal{W}_{\mathcal{V}}}^{-1}\ast\bigxiaokuohao{\mathcal{I} + \frac{1}{2}\mathcal{W}_{\mathcal{V}}}\in \st{n,n,l},$$ and thus $R_{\MC{X}}(\mathcal{V})\in \st{n,p,l} $.  Consider the following curve on $\st{n,p,l} $: $$\mathcal{F}(t) = R_{\MC{X}}(t\mathcal{V})= \bigxiaokuohao{\mathcal{I} - \frac{t}{2}\mathcal{W}_{\mathcal{V}}}^{-1}\ast\bigxiaokuohao{\mathcal{I} + \frac{t}{2}\mathcal{W}_{\mathcal{V}}}\ast\MC{X},$$  with $\mathcal{F}(0) = \MC{X}$. Differentiating both sides of the following equation 
		$$\bigxiaokuohao{\mathcal{I} - \frac{t}{2}\mathcal{W}_{\mathcal{V}}}\ast\mathcal{F}(t) = \bigxiaokuohao{\mathcal{I} + \frac{t}{2}\mathcal{W}_{\mathcal{V}}}\ast\MC{X}$$ with respect to  $t$  gives  $$\dot{{\mathcal{F}}}(0) =  \mathcal{W}_{\mathcal{V}}\ast \MC{X}  =\mathcal{V}-\frac{1}{2}\MC{X}\tprod\bigxiaokuohao{\MC{X}^{\top}\tprod\MC{V}+\MC{V}^{\top}\tprod\MC{X}}=\mathcal{V}$$ for all $\MC{V}\in T_{\MC{X}}\st{n,p,l} $. Hence $R_{\MC{X}}(\MC{O}_{\MC{X}}) = \MC{X}$ and $$D R_{\MC{X}}(\MC{O}_{\MC{X}})[\mathcal{V}] = \frac{\mathrm{d}}{\mathrm{d}t}R_{\MC{X}}(\MC{O}_{\MC{X}}+t\mathcal{V})\big|_{t=0} = \mathcal{V},$$ which shows that $R_{\MC{X}}(\mathcal{V})$ is a retraction.
	\end{proof}
	Based on t-exponential introduced in Subsection \ref{subsect:t-exp}, we get the following retraction.
	\begin{theorem}[t-exponential based retraction]\label{t-exp based retraction}
		The retraction on $\st{n,p,l}$ based on     t-exponential is
		\[R_{\MC{X}}(\mathcal{V}) = (\MC{X} \,\,\,\, \mathcal{Q}) *\ex{
			\begin{pmatrix}
				\begin{smallmatrix}
					\MC{X}^{\top}*\mathcal{V} & -\mathcal{R}^{\top}\\
					\mathcal{R} & \mathcal{O}
				\end{smallmatrix}
		\end{pmatrix}}*
		\begin{pmatrix}
			\begin{smallmatrix}
				\mathcal{I}\\
				\mathcal{O}
			\end{smallmatrix}
		\end{pmatrix},
		\]
		where \(\mathcal{Q}*\mathcal{R}=\MC{C}\) is the t-QR decomposition of 
		\(\mathcal{C} = (\mathcal{I}-\MC{X}*\MC{X}^{\top})*\mathcal{V}\) if $\hat{\MC{C}}\in \TC{n}{p}{l}_*$
		% when $\hat{\MC{C}}\in \TC{n}{p}{l}_*$
		; alternatively, if  $\hat{\MC{C}}\notin \TC{n}{p}{l}_*$, then
		%$\hat{\MC{C}}$
		%% \(\mathcal{C} = (\mathcal{I}_{nnl}-\MC{X}*\MC{X}^{\top})*\mathcal{V}\)
		%is not of f-full rank $p$, 
		$\mathcal{Q} \in \R^{n \times (n-p) \times l}, \mathcal{R} \in \R^{(n-p) \times p \times l},$
		in such a way that \(\mathcal{Q}\) is partially orthogonal and \(\MC{X}^{\top}*\mathcal{Q}=\MC{O}\).
		%Note \(\mathcal{R}\) is not required to be f-upper triangular.
	\end{theorem}
	\begin{proof}
		When $\hat{\MC{C}}\in \TC{n}{p}{l}_*$,  \(\hat{{R}}^{(i)}\) are invertible and therefore \(\mathcal{R}\) is  invertible and 
		\[\MC{X}^{\top}*\mathcal{Q}=\MC{X}^{\top}*\mathcal{C}*\mathcal{R}^{-1}=\MC{O},\] since
		$\MC{X}^{\top}*\mathcal{C}=\MC{X}^{\top}*(\mathcal{V}-\MC{X}*\MC{X}^{\top}*\mathcal{V})
		=\MC{X}^{\top}*\mathcal{V}-\MC{X}^{\top}*\MC{X}*\MC{X}^{\top}*\mathcal{V}=\MC{O}.$
		
		Now consider the case that when 
		$\hat{\MC{C}}$
		is not of f-full rank $p$. Such a choice is always available as we show below. Recall that every \(\mathcal{V} \in T_{\MC{X}}\mathrm{St}(n,p,l)\) can be written in the form
		\(\mathcal{V}=\MC{X}*\mathcal{A}+\MC{X}_{\perp}*\mathcal{B}\), where
		\(\MC{X}_{\perp} \in \R^{n \times (n-p) \times l}\) is partially orthogonal and \(\MC{X}^{\top}*\MC{X}_{\perp}=\MC{O}\).
		According to Remark \ref{remark:DFT equality}, for each \(i\in[l]\), the complex matrix \(\hat{X}_{\perp}^{(i)}\) is partially unitary and 
		\((\hat{X}^{(i)})^H\hat{X}_{\perp}^{(i)}=O_{p\times (n-p)}\).
		Note that
		\begin{equation*}
			\mathcal{C}	=\mathcal{V}-\MC{X}*\MC{X}^{\top}*\mathcal{V}
			=\mathcal{V}-\MC{X}*\MC{X}^{\top}*(\MC{X}*\mathcal{A}+\MC{X}_{\perp}*\mathcal{B})
			=\mathcal{V}-\MC{X}*\mathcal{A}
			=\MC{X}_{\perp}*\mathcal{B}.
		\end{equation*}
		Therefore, using Remark \ref{remark:DFT equality} again, for \(i\in[l]\), we have \(\hat{{C}}^{(i)}=\hat{X}_{\perp}^{(i)}\hat{{B}}^{(i)}\) and consequently
		$\mathrm{span}\,\hat{{C}}^{(i)} \subset \mathrm{span}\,\hat{X}_{\perp}^{(i)},$
		with \(\mathrm{dim}\bigxiaokuohao{\mathrm{span}\,\hat{X}_{\perp}^{(i)}}=n-p\).
		Pick an orthonormal basis of \(\mathrm{span}\,\hat{X}_{\perp}^{(i)}\) and 
		form the partially unitary matrix \(\hat{Q}^{(i)} \in \mathbb{C}^{n \times (n-p)}\). Then we have 
		\((\hat{X}^{(i)})^H\hat{Q}^{(i)} = O\) since the column vectors of \(\hat{X}_i\) are orthogonal to any vector in
		\(\mathrm{span}\,\hat{X}_{\perp}^{(i)}\).
		Since \(\mathrm{span}\,\hat{{C}}_i \subset \mathrm{span}\,\hat{X}_{\perp}^{(i)}\), we have
		\(\hat{{C}}_i=\hat{Q}^{(i)}\hat{{R}}^{(i)}\) for some \(\hat{{R}}^{(i)} \in \mathbb{C}^{(n-p)\times p}\).
		Applying inverse DFT to the third-order tensors \(
		\Fold{\hat{Q}^{(i)}}
		%\mathrm{fold}(\begin{pmatrix}
			%	\hat{\mathcal{U}}^{(1)}\\
			%	\vdots\\
			%	\hat{\mathcal{U}}^{(l)}
			%\end{pmatrix})
			\)
			and  \(
			\Fold{\hat{{R}}^{(i)}}
			%\mathrm{fold}(\begin{pmatrix}
				%	\hat{\mathcal{R}}^{(1)}\\
				%	\vdots\\
				%	\hat{\mathcal{R}}^{(l)}
				%\end{pmatrix})
				\),
				we obtain a partially orthogonal tensor \(\mathcal{Q} \in \R^{n \times (n-p) \times l}\) with
				\(\MC{X}^{\top}*\mathcal{Q}=\MC{O}\) and a tensor \(\mathcal{R}\in \R^{(n-p)\times p \times l}\) with \(\mathcal{Q}*\mathcal{R}=\mathcal{C}\).
				Indeed we can simply take \(\mathcal{Q}=\MC{X}_{\perp}\) and \(\mathcal{R}=\mathcal{B}\).  
				
				In conclusion, there always holds that \(\MC{X}^{\top}*\mathcal{Q}=\MC{O}\) regardless of whether $\hat{\MC{C}}$ is f-full rank or not.
				Now we prove that the exponential retraction on \(\mathrm{St}(n,p,l)\) as defined above is indeed a retraction.
				
				First of all, we prove that \(R_{\MC{X}}(\mathcal{V}) \in \mathrm{St}(n,p,l)\).
				Denotes $\mathcal{A}=\begin{pmatrix}
					\begin{smallmatrix}
						\MC{X}^{\top}*\mathcal{V} & -\mathcal{R}^{\top}\\
						\mathcal{R}& \MC{O}
					\end{smallmatrix}
				\end{pmatrix}$. Since $\MC{A}$ is a skew-symmetric tensor, it follows from Proposition \ref{prop:exp addition} that
				$\ex{\mathcal{A}^{\top}}*\ex{\mathcal{A}}=\ex{-\mathcal{A}}*\ex{\mathcal{A}}=\ex{-\mathcal{A}+\mathcal{A}}=\mathcal{I}.$ 
				Thus we have
				\begin{eqnarray*}
					R_{\MC{X}}(\mathcal{V})^{\top}*R_{\MC{X}}(\mathcal{V})
					&=&\begin{pmatrix}
						\mathcal{I}&\MC{O}
					\end{pmatrix}*
					\ex{
						\MC{A}
						%			\begin{pmatrix}
							%			\MC{X}^T*\mathcal{V} & -\mathcal{R}^T\\
							%			\mathcal{R}& O
							%		\end{pmatrix}
					}^{\top}*\begin{pmatrix}
						\begin{smallmatrix}
							\MC{X}^{\top}\\
							\mathcal{Q}^{\top}
						\end{smallmatrix}
					\end{pmatrix}
					*
					(\MC{X}\,\,\,\,\,\mathcal{Q}) *\ex {
						\MC{A}
						%			\begin{pmatrix}
							%			\MC{X}^T*\mathcal{V} & -\mathcal{R}^T\\
							%			\mathcal{R}& O
							%		\end{pmatrix}
					}*
					\begin{pmatrix}
						\begin{smallmatrix}
							\mathcal{I}\\
							\MC{O}
						\end{smallmatrix}
					\end{pmatrix}\\
					&=&\begin{pmatrix}
						\mathcal{I}&\MC{O}
					\end{pmatrix}*
					\ex{
						\MC{A}
						%			\begin{pmatrix}
							%			\MC{X}^T*\mathcal{V} & -\mathcal{R}^T\\
							%			\mathcal{R}& O
							%		\end{pmatrix}
						^{\top}}
					*
					\begin{pmatrix}
						\begin{smallmatrix}
							\MC{X}^{\top}*\MC{X} & 	\MC{X}^{\top}*\mathcal{Q}\\
							\mathcal{Q}^{\top}*\MC{X}& \mathcal{Q}^{\top}*\mathcal{Q}
						\end{smallmatrix}
					\end{pmatrix}
					*
					\ex{
						\MC{A}
						%			\begin{pmatrix}
							%			\MC{X}^T*\mathcal{V} & -\mathcal{R}^T\\
							%			\mathcal{R}& O
							%		\end{pmatrix}
					}*
					\begin{pmatrix}
						\begin{smallmatrix}
							\mathcal{I}\\
							\MC{O}
						\end{smallmatrix}
					\end{pmatrix}\\
					&=&\begin{pmatrix}
						\mathcal{I}&\MC{O}
					\end{pmatrix}*
					\ex{
						\MC{A}
						%			\begin{pmatrix}
							%			\MC{X}^T*\mathcal{V} & -\mathcal{R}^T\\
							%			\mathcal{R}& O
							%		\end{pmatrix}
						^{\top}}
					*
					\ex{
						\MC{A}
						%			\begin{pmatrix}
							%			\MC{X}^T*\mathcal{V} & -\mathcal{R}^T\\
							%			\mathcal{R}& O
							%		\end{pmatrix}
					}*
					\begin{pmatrix}
						\begin{smallmatrix}
							\mathcal{I}\\
							\MC{O}
						\end{smallmatrix}
					\end{pmatrix}
					=
					%		\begin{pmatrix}
						%			\mathcal{I}&0
						%		\end{pmatrix}*\begin{pmatrix}
						%			\mathcal{I}\\
						%			O
						%		\end{pmatrix}
					\mathcal{I},
					%	\\
					%		=&\mathcal{I}
				\end{eqnarray*}
				%	where we have used the fact that for the skew-symmetric tensor 
				%	$\mathcal{A}=\begin{pmatrix}
					%		\MC{X}^T*\mathcal{V} & -\mathcal{R}^T\\
					%		\mathcal{R}& O
					%	\end{pmatrix}$ is skew-symmetric and 
				%	we have
				%	$\ex{\mathcal{A}^T}*\ex{\mathcal{A}}=\ex{-\mathcal{A}}*\ex{\mathcal{A}}=\ex{-\mathcal{A}+\mathcal{A}}=\mathcal{I}_{2p,2p,l}.$
				where the second equality comes from Proposition \ref{prop:exp tranpose}.
				We now derive an equivalent formula for the exponential retraction as follows:
				\begin{eqnarray*}
					R_{\MC{X}}(\mathcal{V}) 
					&=& (\MC{X}\,\,\mathcal{Q}) *\ex {
						\MC{A}
						%			\begin{pmatrix}
							%			\MC{X}^T*\mathcal{V} & -\mathcal{R}^T\\
							%			\mathcal{R}& O
							%		\end{pmatrix}
					}*
					\begin{pmatrix}
						\begin{smallmatrix}
							\mathcal{I}\\
							\MC{O}
						\end{smallmatrix}
					\end{pmatrix}\\
					&=&(\MC{X}\,\,\mathcal{Q}) *\ex {
						\MC{A}
						%			\begin{pmatrix}
							%			\MC{X}^T*\mathcal{V} & -\mathcal{R}^T\\
							%			\mathcal{R}& O
							%		\end{pmatrix}
					}*
					\begin{pmatrix}
						\begin{smallmatrix}
							\MC{X}^{\top}\\
							\mathcal{Q}^{\top}
						\end{smallmatrix}
					\end{pmatrix}*
					(\MC{X}\,\,\mathcal{Q})*
					\begin{pmatrix}
						\begin{smallmatrix}
							\mathcal{I}\\
							\MC{O}
						\end{smallmatrix}
					\end{pmatrix}\\
					&=&\ex{(\MC{X}\,\,\,\mathcal{Q})*
						\MC{A}
						%		\begin{pmatrix}
							%			\MC{X}^T*\mathcal{V} & -\mathcal{R}^T\\
							%			\mathcal{R}& O
							%		\end{pmatrix}
						*
						\begin{pmatrix}
							\begin{smallmatrix}
								\MC{X}^{\top}\\
								\mathcal{Q}^{\top}
							\end{smallmatrix}
						\end{pmatrix}
					}*
					(\MC{X}\,\,\,\mathcal{Q})*
					\begin{pmatrix}
						\begin{smallmatrix}
							\mathcal{I}\\
							\MC{O}
						\end{smallmatrix}
					\end{pmatrix}\\
					&=&\ex{(\MC{X}*\MC{X}^{\top}*\mathcal{V}+\mathcal{Q}*\mathcal{R},\,\,\,-\MC{X}*\mathcal{R}^{\top})*
						\begin{pmatrix}
							\begin{smallmatrix}
								\MC{X}^{\top}\\
								\mathcal{Q}^{\top}
							\end{smallmatrix}
					\end{pmatrix}}*\MC{X}\\
					&=&\ex{(\MC{X}*\MC{X}^{\top}*\mathcal{V}+(\mathcal{I}-\MC{X}*\MC{X}^{\top})*\mathcal{V},\,\,\,-\MC{X}*\mathcal{R}^{\top})
						*\begin{pmatrix}
							\begin{smallmatrix}
								\MC{X}^{\top}\\
								\mathcal{Q}^{\top}
							\end{smallmatrix}
					\end{pmatrix}}*\MC{X}\\
					&=&\ex{(\mathcal{V},\,\,\,-\MC{X}*\mathcal{R}^{\top})*
						\begin{pmatrix}
							\begin{smallmatrix}
								\MC{X}^{\top}\\
								\mathcal{Q}^{\top}
							\end{smallmatrix}
					\end{pmatrix}}*\MC{X}\\
					&=& \ex{\mathcal{V}*\MC{X}^{\top}-\MC{X}*(\mathcal{Q}*\mathcal{R})^{\top}}*\MC{X}\\
					&=& \ex {\mathcal{V}*\MC{X}^{\top}+\MC{X}*\mathcal{V}^{\top}*(\MC{X}*\MC{X}^{\top}-\mathcal{I})}*\MC{X},
				\end{eqnarray*} 
				where the third equality uses Proposition \ref{prop:exp decomp}.
				This equivalent expression for \(R:T\mathrm{St}(n,p,l) \to \mathrm{St}(n,p,l)\) proves its smoothness since it involves only the exponential and the t-product of \(\MC{X}\)
				and \(\mathcal{V}\),
				which from Proposition \ref{prop:exp smooth} are both smooth operations.
				
				We then verify that \(R_{\MC{X}}(\MC{O})=\MC{X}\). Note that in this case (\(\mathcal{V}=\MC{O}\)), we have
				\(\MC{X}^{\top}*\mathcal{V}=\MC{O}\) and
				\(\mathcal{Q}*\mathcal{R}=
				(\mathcal{I}-\MC{X}*\MC{X}^{\top})*\mathcal{V}=\MC{O}\). Therefore \(\mathcal{R} = \mathcal{Q}^{\top}*\mathcal{Q}*\mathcal{R}=
				\MC{O}\).
				Then it holds that
				\[
				\begin{aligned}
					R_\MC{X}(\MC{O})=& (\MC{X}\,\,\,\,\mathcal{Q})*\ex {\begin{pmatrix}
							\begin{smallmatrix}
								\MC{O} & \MC{O} \\
								\MC{O} & \MC{O}
							\end{smallmatrix}
					\end{pmatrix}}*\begin{pmatrix}
						\begin{smallmatrix}
							\mathcal{I}\\
							\MC{O}
						\end{smallmatrix}
					\end{pmatrix}
					=(\MC{X}\,\,\,\,\mathcal{Q})*\begin{pmatrix}
						\begin{smallmatrix}
							\mathcal{I}\\
							\MC{O}
						\end{smallmatrix}
					\end{pmatrix}
					=\MC{X}.
				\end{aligned}
				\]
				Finally we show that the derivative \(\mathrm{D}R_{\MC{X}}(\MC{O}): T_{\MC{X}}\mathrm{St}(n,p,l) \to T_{\MC{X}}\mathrm{St}(n,p,l)\) is the identity mapping.
				For any \(\mathcal{V} \in T_{\MC{X}}\mathrm{St}(n,p,l)\), we have
				\begin{eqnarray*}
					\mathrm{D}R_{\MC{X}}(\MC{O})[\mathcal{V}]
					&=&\frac{\mathrm{d}}{\mathrm{d}t}(R_{\MC{X}}(t\mathcal{V}))\bigg|_{t=0}\\
					&=&\bigdakuohao{(\MC{X}\,\,\,\,\mathcal{Q}) *\ex {t\begin{pmatrix}
								\begin{smallmatrix}
									\MC{X}^{\top}*\mathcal{V} & -\mathcal{R}^{\top}\\
									\mathcal{R} & \MC{O}
								\end{smallmatrix}
						\end{pmatrix}}*
						\begin{pmatrix}
							\begin{smallmatrix}
								\MC{X}^{\top}*\mathcal{V} & -\mathcal{R}^{\top}\\
								\mathcal{R} & \MC{O}
							\end{smallmatrix}
						\end{pmatrix}*
						\begin{pmatrix}
							\begin{smallmatrix}
								\mathcal{I}\\
								\MC{O}
							\end{smallmatrix}
					\end{pmatrix}}\bigg|_{t=0}\\
					&=& \MC{X}*\MC{X}^{\top}*\mathcal{V}+\mathcal{Q}*\mathcal{R}\\
					&=& \MC{X}*\MC{X}^{\top}*\mathcal{V}+(\mathcal{I}-\MC{X}*\MC{X}^{\top})*\mathcal{V}= \mathcal{V},
					%		\\
					%		=& \mathcal{V}.
				\end{eqnarray*}
				where the second equality is due to Proposition \ref{prop:exp der}.
			\end{proof}
			\begin{remark}
				It follows from this formula that when \(n=p\), we have the simpler formula for the exponential retraction on the group
				\(\mathrm{St}(n,n,l)\) of orthogonal tensors: $	R_{\MC{X}}(\mathcal{V})=\ex{\mathcal{V}*\MC{X}^{\top}}*\MC{X}
				=\MC{X}*\ex{\MC{X}^{\top}*\mathcal{V}}.$
				%\[
				%\begin{aligned}
				%	R_{\MC{X}}(\mathcal{V})=&\ex{\mathcal{V}*\MC{X}^{\top}}*\MC{X}\\
				%	=&\ex{\MC{X}*\MC{X}^{\top}*\mathcal{V}*\MC{X}^{\top}}*\MC{X}\\
				%	=&\MC{X}*\ex{
					%	\MC{X}^{\top}*\mathcal{V}
					%}*\MC{X}^{\top}*\MC{X}\\
				%	=&\MC{X}*\ex{\MC{X}^{\top}*\mathcal{V}}.
				%\end{aligned}  
				%\]
			\end{remark}
			The following proposition indicates that the retraction based on t-exponential is actually a geodesic on $\st{n,p,l}$.
			\begin{proposition}[Geodesic]
				The geodesic on $\st{n,p,l}$ emanating from $\MC{X}$ in direction $\MC{V}$ is given by the curve $\MC{C}(t) = R_{\MC{X}}(t\mathcal{V})=	(\MC{X}\,\,\,\,\,\mathcal{Q}) *\ex {
					t\MC{A}
				}*
				\begin{pmatrix}
					\begin{smallmatrix}
						\mathcal{I}\\
						\MC{O}
					\end{smallmatrix}
				\end{pmatrix}$, where  $\mathcal{A}=
				\begin{pmatrix}
					\begin{smallmatrix}
						\MC{X}^{\top}*\mathcal{V} & -\mathcal{R}^{\top}\\
						\mathcal{R}& \MC{O}
					\end{smallmatrix}
				\end{pmatrix}$ is a skew-symmetric tensor.
			\end{proposition}
			\begin{proof}
				The proof of Theorem \ref{t-exp based retraction} shows that $\MC{C}(0) = \MC{X}$ and $\dot{\MC{C}}(0) = \MC{V}$.
				Since $\st{n,p,l}$ is a Riemannian submanifold of $\T{n}{p}{l}$, \cite[Sect. 5.8]{boumal2022intromanifolds} shows that the acceleration of $\MC{C}(t)$ is
				\begin{equation*}
					\ddot{\MC{C}}(t) =  \mathbf{P}_{\MC{C}(t)}\bigxiaokuohao{\frac{\mathrm{d}^2}{\mathrm{d}t^2}\MC{C}(t)},
				\end{equation*}
				where
				$ \mathbf{P}_{\MC{X}}(\mathcal{U}) = \mathcal{U} - \MC{X}\ast\operatorname{sym}\bigxiaokuohao{\MC{X}^{\top}\ast\mathcal{U}}
				$
				and $\operatorname{sym}(\mathcal{A}) = \frac{\mathcal{A}+\mathcal{A}^{\top}}{2}$.
				Hence 
				$$
				\ddot{\MC{C}}(t) = \frac{\mathrm{d}^2}{\mathrm{d}t^2}\MC{C}(t)-\MC{C}(t)\ast \operatorname{sym}\bigxiaokuohao{\MC{C}^{\top}(t)\ast\frac{\mathrm{d}^2}{\mathrm{d}t^2}\MC{C}(t)},
				$$
				where $
				\frac{\mathrm{d}^2}{\mathrm{d}t^2}\MC{C}(t)=
				(\MC{X}\,\,\,\,\,\mathcal{Q}) *\ex {
					t\MC{A}
				}*\MC{A}^2 \ast
				\begin{pmatrix}
					\begin{smallmatrix}
						\mathcal{I}\\
						\MC{O}
					\end{smallmatrix}
				\end{pmatrix}.
				$
				Since $\MC{A}$ is a skew-symmetric tensor, it is easy to check $\ddot{\MC{C}}(t) = \MC{O}$, which means $\MC{C}(t)$ is the geodesic on $\st{n,p,l}$.
			\end{proof}
			\subsection{Vector transport on  $\st{n,p,l}$}\label{subsec:vec trans}
			By Lemma \ref{lemma:vector transport} and the orthogonal projector operator \eqref{eq:proj}, we obtain a series of vector transports as follows.
			\begin{theorem}[Orthogonal projector based vector transport]
				The vector transport on $\st{n,p,l}$ is
				\begin{eqnarray}\label{eq:Orthogonal projector based vector transport}
					\mathcal{T}_{\MC{U}}{\MC{V}} 
					&=&\mathbf{P}_{R_{\MC{X}}(\MC{U})}\MC{V}\nonumber\\
					&=& (\mathcal{I}-\mathcal{Y}\ast\mathcal{Y}^{\top})\ast\MC{V}+\mathcal{Y}\ast\operatorname{skew}(\mathcal{Y}^{\top}\ast \MC{V})\nonumber\\
					&=& \MC{V}-\mathcal{Y}\ast\operatorname{sym}(\mathcal{Y}^{\top}\ast \MC{V})\in T_{\mathcal{Y}}\st{n,p,l},
				\end{eqnarray}
				where $\mathcal{Y} = R_{\MC{X}}(\MC{U})$ is any retraction on $\st{n,p,l}$.
			\end{theorem}
			% \begin{lemma}\label{lemma:upper tri mannifold} $\TU{p}{p}{l}$ is  a linear manifold and $T_{\MC{X}}\TU{p}{p}{l} = \TU{p}{p}{l}, \MC{X}\in \TU{p}{p}{l}$.
				%  $\TUP{p}{p}{l}$ is an open submanifold of linear manifold $\TU{p}{p}{l}$, and $T_{\MC{Y}}\TUP{p}{p}{l} = \TU{p}{p}{l}, \MC{Y}\in \TUP{p}{p}{l}$.
				% \end{lemma}
			% \begin{proof}
				% 	Note that $\TU{p}{p}{l}$ is a vector space and so a linear manifold, hence  $T_{\MC{X}}\TU{p}{p}{l} = \TU{p}{p}{l}, \MC{X}\in \TU{p}{p}{l}$. The set of all $p\times p \times l$ real tensors, whose each frontal slice is invertible,  is open in $\T{p}{p}{l}$; hence the intersection of this set with $\TU{p}{p}{l}$, which equals to $\TUP{p}{p}{l}$, is also open in $\TU{p}{p}{l}$.  According to \cite[Section 3.5.2]{absil2009optimization}, since $\TUP{p}{p}{l}$ is an
				% 	open subspace of  manifold $\TU{p}{p}{l}$, thus it is a manifold and its tangent space at any point $\MC{U}$ is just  $T_{\MC{X}}$\TUP{p}{p}{l} = \TU{p}{p}{l}.
				% \end{proof}
			To derive the vector transport by  differentiated retraction  based on t-QR decomposition, the following two lemmas are necessary.
			\begin{lemma}\label{lemma:complex upp manifold}
				$\TCUP{p}{p}{l}$ is an open submanifold of linear manifold $\TCU{p}{p}{l}$ with real diagonal elements, and its tangent space at any point ${\MC{Y}}\in \TCUP{p}{p}{l}$ is just $\TCU{p}{p}{l}$ with real diagonal elements.
			\end{lemma}
			\begin{proof}
				Note that $\TCU{p}{p}{l}$ with real diagonal elements is a vector space and so a linear manifold. Since $\mathbb{R}^{+}$ is open in $\mathbb{R}$, it follows that $\TCUP{p}{p}{l}$  is open in $\TCU{p}{p}{l}$ with real diagonal elements.
				% The set of all $p\times p \times l$ real tensors, whose each frontal slice is invertible,  is open in $\T{p}{p}{l}$; hence the intersection of this set with $\TU{p}{p}{l}$, which equals to $\TUP{p}{p}{l}$, is also open in $\TU{p}{p}{l}$.  
				% Since $\TCUP{p}{p}{l}$ is an
				%	open subspace of  manifold $\TCU{p}{p}{l}$ with real diagonal elements,
				Then it follows from 
				\cite[Sect. 3.5.2]{absil2009optimization} that $\TCUP{p}{p}{l}$ is a manifold and its tangent space at any point ${\MC{Y}}\in \TCUP{p}{p}{l}$ is just $\TCU{p}{p}{l}$ with real diagonal elements.
			\end{proof}
			\begin{lemma}\label{prop:curve}
				Let $\MC{C}(t) = \mathcal{A}(t)\ast\mathcal{B}(t)\in \T{m}{n}{l}
				$.
				% 	be a curve on an abstract manifold. 
				Then the tangent vector to the curve $\MC{C}(t)$ is
				$$
				\dot{\MC{C}}(t)= \dot{\MC{A}}(t)\ast{\MC{B}}(t) + {\MC{A}}(t)\ast\dot{\MC{B}}(t).
				$$
			\end{lemma}
			\begin{proof}
				By Definition \ref{def:DFT}, there holds  
				\begin{equation*}
					\operatorname{unfold}(\mathcal{C}(t)) =	\operatorname{unfold}(\mathcal{A}(t)\ast\mathcal{B}(t))\\ = \operatorname{bcirc}(\mathcal{A}(t))\cdot \operatorname{unfold}(\mathcal{B}(t)).
				\end{equation*}
				Hence it follows that
				\begin{equation*}
					C^{(k)}(t) = \sum\nolimits_{i =1}^{l}A^{(h_i)}(t)B^{(i)}(t),~k\in[l],
				\end{equation*}
				where $$h_i = \begin{cases}
					l+k+1-i,& i> k \\ 
					k+1-i,&  i\leq k
				\end{cases}.$$ 
				Using the corresponding property of the matrix case \cite{absil2009optimization}, we obtain 
				\begin{equation*}
					\dot{C}^{(k)}(t) = \sum\nolimits_{i =1}^{l}\bigxiaokuohao{
						\dot{A}^{(h_i)}(t)B^{(i)}(t)
						+A^{(h_i)}(t)\dot{B}^{(i)}(t)
					},
				\end{equation*}
				which means $
				\operatorname{unfold}(\dot{\mathcal{C}}(t)) =	 \operatorname{bcirc}(\dot{\mathcal{A}}(t))\cdot \operatorname{unfold}(\mathcal{B})(t)+\operatorname{bcirc}(\mathcal{A}(t))\cdot \operatorname{unfold}(\dot{\mathcal{B}}(t)).
				$ Thus it follows that
				\[
				\dot{\MC{C}}(t)= \dot{\MC{A}}(t)\ast{\MC{B}}(t) + {\MC{A}}(t)\ast\dot{\MC{B}}(t).
				\]
			\end{proof}
			We are now in a position to derive computational formulae from  \eqref{eq:vec transport} for the vector transport as the differentiated retraction $	R_{\MC{X}}(\mathcal{U}) = qf\bigxiaokuohao{\MC{X}+\mathcal{U}}$.
			\begin{theorem}[t-QR based vector transport]	The vector transport on $\st{n,p,l}$ by differentiated retraction $	R_{\MC{X}}(\mathcal{U}) = qf\bigxiaokuohao{\MC{X}+\mathcal{U}}$  is
				\begin{equation*}	
					\mathcal{T}_{\mathcal{U}}{\mathcal{V}} =
					\MC{Q}\ast L^{-1}\bigxiaokuohao{\Fold{ \mathbf{P}_{\operatorname{skew}}\bigxiaokuohao{\hat{B}^{(i)}}}} +
					\bigxiaokuohao{\mathcal{I}-\MC{Q}\ast \MC{Q}^{\top}}\ast 	\MC{C},
				\end{equation*}
				where $\MC{Q}=R_{\MC{X}}(\mathcal{U})$, $\MC{C} = \mathcal V\ast (\MC{Q}^{\top}\ast (\MC{X}+\mathcal{U}))^{-1}$, $\MC{B} =\MC{Q}^{\top} \ast \MC{C}$ and $\mathbf{P}_{\operatorname{skew}}(\hat{B}^{(i)})$	denotes   the skew-symmetric term of the decomposition of the complex matrix
				$\hat{B}^{(i)} = \bigxiaokuohao{L(\MC{B})}^{(i)}$ into the sum of a skew-symmetric term and an upper triangular term with real diagonal elements. Specifically,
				\begin{equation}
					%(\mathbf{P}_{\operatorname{skew}}(\MC{B}))_{::i} := %\operatorname{ifft}(\mathbf{P}_{\operatorname{skew}}(%\hat{B}^{(i)}),[~],3), ~~ \text{and} ~~
					(\mathbf{P}_{\operatorname{skew}}(\hat{B}^{(i)}))_{mn} = \begin{cases}
						-\operatorname{conj}(\hat{B}^{(i)}_{mn}),& m<n \\ 
						\operatorname{Im}(\hat{B}^{(i)}_{mn}),&  m=n\\
						\hat{B}^{(i)}_{mn},& m>n
					\end{cases}
				\end{equation}
				where $\operatorname{Im}(c)$ represents the imaginary part of the complex number $c$.
			\end{theorem}
			\begin{proof}
				From \eqref{eq:vec transport} and Theorem \ref{th:t-qr based retraction},
				for $\mathcal{U},\mathcal{V}\in T_{\MC{X}}\st{n,p,l},$ we have
				\begin{equation*}
					\mathcal{T}_{\mathcal{U}}{\mathcal{V}}
					= D R_{\MC{X}}(\mathcal{U})[\mathcal{V}]
					= D qf(\MC{X}+\mathcal{U})[\mathcal{V}]
					= \frac{\mathrm{d}}{\mathrm{d}t}qf(\MC{X}+\mathcal{U}+t\mathcal{V})\big|_{t=0}.
				\end{equation*}
				This is well defined,  i.e., the t-QR decomposition of  $\MC{W}(t):=\MC{G} + t\MC{V}$ is unique, where $\MC{G}:=\MC{X}+\MC{U}$. To see this, it follows from Remark \ref{remark:DFT equality} that
				$$\MC{W}(t)^{\top}\ast\MC{W}(t) = \mathcal{I} + \bigxiaokuohao{\mathcal{U}+t\mathcal{V}}^{\top}\ast\bigxiaokuohao{\mathcal{U}+t\mathcal{V}}\Leftrightarrow
				(\hat{W}(t)^{(i)})^H\hat{W}(t)^{(i)}
				= {I} + \bigxiaokuohao{\hat{U}^{(i)}+t\hat{V}^{(i)}}^{H}\bigxiaokuohao{\hat{U}^{(i)}+t\hat{V}^{(i)}}, i\in [l],$$
				showing that $\hat{\MC{W}}(t)\in \TC{n}{p}{l}_*$, which together with Theorem \ref{th:t-qr based retraction} gives the desired result.
				% 	Thus we need a formula for $D qf(\mathcal{Q})[\mathcal{V}]$ with $\hat{\mathcal{Q}}\in\mathbb{C}_*^{n\times p \times l}$ and $\hat{\mathcal{Q}}+t\hat{\mathcal{V}}\in\mathbb{C}_*^{n\times p \times l}$.
				Hence $\mathcal{W}(t) 
				=\mathcal{G}+t\mathcal{V}
				$ is a curve on $L^{-1}\bigdakuohao{\mathbb{C}_*^{n\times p \times l}}$ with $\mathcal{W}(0) = \mathcal{G}$ and $\dot{\mathcal{W}}(0) = \mathcal{V}$.
				Let $\mathcal W(t) = \mathcal Q(t) \ast\mathcal R(t)$ denote the t-QR decomposition of $\mathcal W(t)$, where $\mathcal Q(t)\in \st{n,p,l}$ and $\mathcal R(t)\in L^{-1}\bigxiaokuohao{\TCUP{p}{p}{l}}$. 	Hence $\mathcal{Q}(0)=qf(\mathcal{G}),~\mathcal{R}(0) = qf(\mathcal{G})^{\top}\ast\mathcal{G}
				%	,\mathcal{V}=\MC{X}'(0)\ast qf(\mathcal{Q})^{\top}\ast\mathcal{Q}+qf(\mathcal{Q})\ast \mathcal{R}'(0)
				$. Our task now is to compute $\mathcal{T}_{\mathcal{U}}{\mathcal{V}} = \frac{\mathrm{d}}{\mathrm{d}t}\MC{Q}(t)\big|_{t=0}=\dot{\mathcal Q}(0).$
				Since $\mathcal{I}=\mathcal{Q}(t)\ast \mathcal{Q}(t)^{\top} + \bigxiaokuohao{\mathcal{I}-\mathcal{Q}(t)\ast \mathcal{Q}(t)^{\top}}$, we have the decomposition
				\begin{equation}\label{eq:2}
					\dot{\mathcal Q}(t)=\mathcal{Q}(t)\ast \mathcal{Q}(t)^{\top}\ast \dot{\mathcal Q}(t)  + \bigxiaokuohao{\mathcal{I}-\mathcal{Q}(t)\ast \mathcal{Q}(t)^{\top}}\ast \dot{\mathcal Q}(t).
				\end{equation}
				It follows from Lemma \ref{prop:curve} that
				\begin{equation}\label{eq:1}
					\dot{\mathcal W}(t) = 	\dot{\mathcal Q(t)} \ast	\mathcal R(t) + 	\mathcal Q(t) \ast 	\dot{\mathcal R}(t). 
				\end{equation}
				Since $\hat{\MC{R}}(t)\in \TCUP{p}{p}{l}$ which means $\hat{{R}}^{(i)}(t)$ are invertible, it follows from Remark \ref{remark:DFT equality} and Definition \ref{def:inverse} that ${\MC{R}}(t)$ are invertible.
				Multiplying \eqref{eq:1} by $\mathcal{I}-\mathcal{Q}(t)\ast \mathcal{Q}(t)^{\top}$ on the left and by $(\mathcal{R}(t))^{-1}$ on the right yields 
				\begin{equation}\label{eq:3}
					\bigxiaokuohao{\mathcal{I}-\mathcal{Q}(t)\ast \mathcal{Q}(t)^{\top}}\ast \dot{\mathcal Q}(t) = \bigxiaokuohao{\mathcal{I}-\MC{Q}(t)\ast \mathcal{Q}(t)^{\top}}\ast 	\dot{\mathcal W}(t)\ast (\mathcal{R}(t))^{-1}.
				\end{equation} 
				which is the second term of \eqref{eq:2}. It remains to derive the computational formulae for the first term of \eqref{eq:2}.
				Since $\dot{\mathcal Q}(t)$ is a tangent vector at the point $\mathcal{Q}(t)$, it follows that \eqref{eq:2} satisfies the form:
				\begin{equation*}	T_{\mathcal{Q}(t)}\st{n,p,l}  
					%		= \operatorname{ker}( Dh(\MC{X})) 
					= \bigdakuohao{ \mathcal{Q}(t)\ast \mathcal{W}(t)+\mathcal{Q}_{\perp }(t)\ast \mathcal{B}(t)\in\mathbb{R}^{n\times p \times l }\bigg|\mathcal{W}(t)\in\TSkew{p}{p}{l},~\mathcal{B}(t)\in \mathbb{R}^{(n-p)\times p \times l}},
				\end{equation*}
				where $ \hat{Q}^{(i)}_{\perp }(t)\in \mathbb{C}^{n\times (n-p)}$ is any matrix such that $\operatorname{span}(\hat{Q}^{(i)}_{\perp}(t))=\{\hat{Q}^{(i)}_{\perp}(t)\alpha\big|\alpha\in \mathbb{C}^{n-p}\}$ is the orthogonal complement of $\operatorname{span}(\hat{Q}^{(i)}(t))=\{\hat{Q}^{(i)}(t)\beta\big|\beta\in \mathbb{C}^{p}\}$.
				It is easy to check that $\innerprod{\hat{Q}^{(i)}(t)\beta}{\bigxiaokuohao{I-\hat{Q}^{(i)}(t)(\hat{Q}^{(i)}(t))^H}\alpha} = 0$ for any $\alpha\in \mathbb{C}^{n-p}$ and $\beta\in \mathbb{C}^{p}$,
				which means matrix $\bigxiaokuohao{I-\hat{Q}^{(i)}(t)(\hat{Q}^{(i)}(t))^H}$ is a choice of matrix $\hat{Q}^{(i)}_{\perp }(t)$,
				hence the
				term $\mathcal{Q}(t)^{\top} \ast \dot{\mathcal{Q}}(t)\in\TSkew{p}{p}{l}$. Next it is sufficient for us to obtain the formula for $\mathcal{Q}(t)^{\top} \ast \dot{\mathcal{Q}}(t).$

				Multiplying \eqref{eq:1} on the left
				by $\mathcal{Q}(t)^{\top}$ and on the right by $(\mathcal{R}(t))^{-1}$ leads to
				\begin{equation}\label{eq:33}
					\MC{B}(t):=\mathcal{Q}(t)^{\top} \ast \dot{\mathcal W}(t)\ast  (\mathcal{R}(t))^{-1} 
					= \mathcal{Q}(t)^{\top} \ast \dot{\mathcal{Q}}(t) + \dot{\mathcal R}(t) \ast (\mathcal{R}(t))^{-1}.
				\end{equation}	
				Since $\hat{\MC{R}}(t)\in \TCUP{p}{p}{l}$, Lemma \ref{lemma:complex upp manifold} shows that $\dot{\hat{\MC{R}}}(t)\in \TCU{p}{p}{l}$ with real diagonal elements.
				% , which combines Lemma \ref{lemma:upper tri mannifold} leads to the term $\dot{\mathcal R}(t)\in \TU{p}{p}{l}$. 
				% Note that $\TU{p}{p}{l}$ is a vector space and so a linear manifold, hence 
				% % $\dot{\mathcal R}(t)$ reduces to its classical definition
				% % $\MC{R}'(t) = \lim\limits_{\tau\rightarrow 0} \frac{1}{\tau}\bigxiaokuohao{\MC{R}(t+\tau)-\MC{R}(t)}$ 
				% % and 
				% $\dot{\mathcal R}(t)\in T_{\MC{X}}\TU{p}{p}{l} =\TU{p}{p}{l}, ~\text{for}~\MC{X}\in\TU{p}{p}{l}$. 
				% It follows from Remark \ref{remark:DFT equality} and Remark \ref{remark:Linear relationship} that $ (\hat{\mathcal{R}}^{(i)}(t))^{-1}$ are upper triangular complex matrix with real diagonal elements.
				% belongs to the set of f-upper triangular tensors.
				Note that $\hat{{R}}^{(i)}(t)\in \mathbb{C}^{p\times p}_{upp+}$ are invertible, hence $\bigxiaokuohao{\hat{R}^{(i)}(t)}^{-1}\in \mathbb{C}^{p\times p}_{upp+}$.
				Thus $\dot{\hat R}^{(i)}(t) \bigxiaokuohao{\hat{R}^{(i)}(t)}^{-1}$ is upper triangular complex matrix with real diagonal elements. 
				Using Remark \ref{remark:DFT equality} again, we get
				\begin{equation}\label{eq:13}
					\hat{B}^{(i)}=\bigxiaokuohao{\hat{Q}^{(i)}(0)}^{H}\dot{\hat W}^{(i)}(0)  \bigxiaokuohao{\hat{R}^{(i)}(0)}^{-1} 
					= \bigxiaokuohao{\hat{Q}^{(i)}(0)}^{H} \dot{\hat{Q}}^{(i)}(0) + \dot{\hat R}^{(i)}(0) \bigxiaokuohao{\hat{R}^{(i)}(0)}^{-1}.
				\end{equation}
				
				%	Let $\mathbf{P}_{\operatorname{skew}}(\hat{B}^{(i)})$
				%	denote the  skew-symmetric term of the decomposition of a complex matrix
				%	$\hat{B}^{(i)}$ into the sum of a skew-symmetric term and an upper triangular term with real diagonal elements. Specifically,
				%	\begin{equation}
					%		%(\mathbf{P}_{\operatorname{skew}}(\MC{B}))_{::i} := %\operatorname{ifft}(\mathbf{P}_{\operatorname{skew}}(%\hat{B}^{(i)}),[~],3), ~~ \text{and} ~~
					%		(\mathbf{P}_{\operatorname{skew}}(\hat{B}^{(i)}))_{mn} = \begin{cases}
						%			-\operatorname{conj}(\hat{B}^{(i)}_{mn}),& m<n \\ 
						%			\operatorname{Im}(\hat{B}^{(i)}_{mn}),&  m=n\\
						%			\hat{B}^{(i)}_{mn},& m>n
						%		\end{cases}
					%	\end{equation}
				%where $\operatorname{Im}(c)$ represents the imaginary part of the complex number $c$.
				Recalling that $\bigxiaokuohao{\hat{Q}^{(i)}(t)}^{H} \dot{\hat{Q}}^{(i)}(t)$ is skew-symmetric and 
				applying the operator $\mathbf{P}_{\operatorname{skew}}$ which denotes the skew-symmetric term of the decomposition of the complex matrix into the sum of a skew-symmetric term and an upper triangular term with real diagonal elements, we obtain
				\begin{equation}\label{eq:4.31}
					\bigxiaokuohao{L\bigxiaokuohao
						{\mathcal{Q}(0)^{\top} \ast \dot{\mathcal{Q}}(0)}}^{(i)}  =\bigxiaokuohao{\hat{Q}^{(i)}(t)}^{H} \dot{\hat{Q}}^{(i)}(t) 
					=  \mathbf{P}_{\operatorname{skew}}\bigxiaokuohao{\hat{B}^{(i)}}
				\end{equation} 
				And  Remark \ref{remark:DFT equality} tell us that \eqref{eq:4.31} can be equivalently rewritten as
				\begin{equation}\label{eq:first term}
					\mathcal{Q}(0)^{\top} \ast \dot{\mathcal{Q}}(0) = L^{-1}\bigxiaokuohao{\Fold{ \mathbf{P}_{\operatorname{skew}}\bigxiaokuohao{\hat{B}^{(i)}}}},
				\end{equation} 
				where $\MC{B} = \mathcal{Q}(0)^{\top} \ast \dot{\mathcal W}(0)\ast  (\mathcal{R}(0))^{-1}.$
				Replacing \eqref{eq:first term} and \eqref{eq:3}
				in \eqref{eq:2} gives
				\begin{eqnarray*}\label{eq:4}
					\dot{\mathcal Q}(0)  
					% 		&=&
					% 		\mathcal{Q}(0)\ast \mathcal{Q}(0)^{\top}\ast \dot{\mathcal Q}(0)  + 
					% 		\bigxiaokuohao{\mathcal{I}-\mathcal{Q}(0)\ast \mathcal{Q}(0)^{\top}}\ast \dot{\mathcal Q}(0) \\
					&=& 
					\mathcal{Q}(0)\ast L^{-1}\bigxiaokuohao{\Fold{ \mathbf{P}_{\operatorname{skew}}\bigxiaokuohao{\hat{B}^{(i)}}}} +
					\bigxiaokuohao{\mathcal{I}-\mathcal{Q}(0)\ast \mathcal{Q}(0)^{\top}}\ast 	\dot{\mathcal W}(0)\ast (\mathcal{R}(0))^{-1}\\
					&=& 
					qf(\mathcal{G})\ast L^{-1}\bigxiaokuohao{\Fold{ \mathbf{P}_{\operatorname{skew}}\bigxiaokuohao{\hat{B}^{(i)}}}} +
					\bigxiaokuohao{\mathcal{I}-qf(\mathcal{G})\ast qf(\mathcal{G})^{\top}}\ast 	\mathcal V\ast (qf(\mathcal{G})^{\top}\ast \mathcal{G})^{-1},
				\end{eqnarray*}
				% 	Therefor there holds that
				% 	\begin{eqnarray*}\label{eq:44}
					% 		\mathrm{D} qf(\mathcal{G})[\mathcal{V}]
					% 		&=&
					% % 		\frac{\mathrm{d}}{\mathrm{d}t}qf(\mathcal{G}+t\mathcal{V})\big|_{t=0} 
					% % 		= \frac{\mathrm{d}}{\mathrm{d}t}qf(\mathcal{W}(t))\big|_{t=0}     
					% % 		=
					% 		\dot{\mathcal{Q}}(0)     \\
					% 		&=& 
					% 		qf(\mathcal{G})\ast L^{-1}\bigxiaokuohao{\Fold{ \mathbf{P}_{\operatorname{skew}}\bigxiaokuohao{\hat{B}^{(i)}}}} +
					% 		 \bigxiaokuohao{\mathcal{I}-qf(\mathcal{G})\ast qf(\mathcal{G})^{\top}}\ast 	\mathcal V\ast (qf(\mathcal{G})^{\top}\ast \mathcal{G})^{-1},
					% 	\end{eqnarray*}
				where $\MC{B} = qf(\mathcal{G})^{\top} \ast \mathcal V\ast  (qf(\mathcal{G})^{\top}\ast \mathcal{G})^{-1}.$
				% Consider $\MC{Q} = \MC{X}+\MC{U}$, then $\MC{W}(t)^{\top}\tprod\MC{W}(t) = \MC{I} + (\MC{U}+t\MC{V})^{\top}\tprod(\MC{U}+t\MC{V})\in \TSPP{p}{p}{l}$, hence $\MC{W}(t) \in \T{n}{p}{l}_*$.
				Finally, we have, for $\mathcal{U},\mathcal{V}\in T_{\MC{X}}\st{n,p,l} ,$
				\begin{equation}
					\mathcal{T}_{\mathcal{U}}{\mathcal{V}} 
					% 		&=& \mathrm{D} R_{\MC{X}}(\mathcal{U})[\mathcal{V}]
					% 		= \mathrm{D} qf(\MC{X}+\mathcal{U})[\mathcal{V}]\\
					= \dot{\mathcal Q}(0) =
					\MC{Q}\ast L^{-1}\bigxiaokuohao{\Fold{ \mathbf{P}_{\operatorname{skew}}\bigxiaokuohao{\hat{B}^{(i)}}}} +
					\bigxiaokuohao{\mathcal{I}-\MC{Q}\ast \MC{Q}^{\top}}\ast 	\mathcal C,
				\end{equation}
				where $\MC{Q}=R_{\MC{X}}(\mathcal{U})$, $\MC{C} = \mathcal V\ast  (\MC{Q}^{\top}\ast (\MC{X}+\mathcal{U}))^{-1}$ and $\MC{B} =\MC{Q}^{\top} \ast \MC{C}$.
			\end{proof}
			To derive the vector transport by differentiated retraction $	R_{\MC{X}}(\mathcal{U}) $ based on t-PD decomposition, we then need the circulant matrices of $\mathcal A$ in another order as follows.
			\begin{definition}
				Let $\mathcal{A}\in \mathbb{R}^{n\times p \times l}$; then the circulant matrices in another order is defined as
				$$\small
				\widetilde{\operatorname{bcirc}}(\mathcal{A}) := \begin{bmatrix}
					\begin{smallmatrix}
						A^{(1)} &  A^{(2)} & \cdots & A^{(l)}\\ 
						A^{(l)} & A^{(1)}  & \cdots & A^{(l-1)} \\ 
						\vdots  & \ddots  & \ddots & \vdots \\ 
						A^{(2)}& \cdots & A^{(l)} & A^{(1)}
					\end{smallmatrix}
				\end{bmatrix}\in\mathbb{R}^{nl\times pl}.
				$$
			\end{definition}
			\begin{theorem}[t-Sylvester equation]\label{th:sylvester}
				Let $\mathcal{A}\in \mathbb{R}^{n\times n \times l},\mathcal{X}\in \mathbb{R}^{n\times p\times l},\mathcal{B}\in \mathbb{R}^{p\times p\times l}$. The analytical solution of the t-Sylvester equation
				\begin{equation}\label{eq:sylvester}
					\mathcal{A}\ast \mathcal{X} + \mathcal{X}\ast \mathcal{B} = \mathcal{C},
				\end{equation}
				is $\operatorname{vec}(\mathcal{X})  = \bigxiaokuohao{ \widetilde{\operatorname{bcirc}}(\mathcal{B})^{\top}\otimes I_n+[I_p]_{l\times l}\odot\operatorname{bcirc}(\mathcal{A}) }^{\dagger }\cdot \operatorname{vec}(\mathcal{C}),$ where 
				$[I_p]_{l\times l}\in\mathbb{R}^{pl\times pl}$ is a  block matrix whose $(i,j)$ submatrice is $I_p,~i\in [l], j\in [l]$,
				%	$[I_k]_{l\times l}  = \begin{bmatrix}
					%		I_k & \cdots  & I_k\\ 
					%		\vdots  & \ddots  & \vdots \\ 
					%		I_k& \cdots  & I_k
					%	\end{bmatrix}_{l\times l},$
				% $M^{\dagger }$ denotes the Moore-Penrose generalized inverse of matrix $M$,
				and $\operatorname{vec}(\MC{C}) = \MC{C}(:)$ denotes the vectorrized of $\MC{C}$ in the meaning of lexicographical ordering.
			\end{theorem}
			Necessary lemmas for proving Theorem \ref{th:sylvester} are provided in Appendix \ref{apx:19}.
			The proof of Theorem \ref{th:sylvester} is also left to Appendix \ref{apx:19}. 
			\begin{lemma}\label{lemma:t-PSD manifold}
				$\TSPP{n}{p}{l}$ is an open submanifold of linear manifold $\TS{n}{p}{l}$ and its tangent space  at any point $\MC{Y}\in\TSPP{n}{p}{l}$ is just $ \TS{n}{p}{l} $.
			\end{lemma}
			\begin{proof}
				Note that $\TS{p}{p}{l}$ is a vector space thus a linear manifold. 	\cite[Rmk. 10]{zheng2021t} shows that $\TSPP{n}{p}{l}$ is an open subset of $\TS{n}{p}{l}$ thus an open submanifold of $\TS{n}{p}{l}$.  
				%Since $\TSPP{n}{p}{l}$ is an open submanifold of $\TS{n}{p}{l}$, 
				Then it follows from
				\cite[Sect. 3.5.2]{absil2009optimization} that  $T_{\MC{Y}}\TSPP{n}{p}{l} = \TS{n}{p}{l} $.
			\end{proof}
			With the help of Theorem \ref{th:sylvester} and Lemma \ref{lemma:t-PSD manifold}, we can now derive the following t-PD based vector transport.
			\begin{theorem}[t-PD based vector transport]The vector transport on $\st{n,p,l}$ by differentiated retraction based on t-PD is
				\begin{equation*}
					\mathcal{T}_{\mathcal{U}}{\mathcal{V}} = \mathcal{Y}\ast \mathcal{S} + (\mathcal{I}-\mathcal{Y}\ast \mathcal{Y}^{\top})\ast \mathcal{V}\ast (\mathcal{Y}^{\top}\ast (\MC{X}+\mathcal{U}))^{-1},
				\end{equation*}
				where $\mathcal{Y}  = 	R_{\MC{X}}(\mathcal{U}) = (\MC{X}+\mathcal{U})\ast\MC{P}^{-1}$, $\operatorname{vec}(\mathcal S)=\bigxiaokuohao{ \widetilde{\operatorname{bcirc}}(\mathcal P)^{\top}\otimes I_p+[I_p]_{l\times l}\odot\operatorname{bcirc}(\mathcal P) }^{\dagger }\cdot \operatorname{vec}(\mathcal{Y}^{\top}\ast \mathcal V - \mathcal V^{\top}\ast 
				\mathcal Y)$ and $\mathcal P =  (\mathcal{I}+\mathcal{U}^{\top}\ast \mathcal{U})^{\frac{1}{2}}$.
			\end{theorem}
			\begin{proof}
				From \eqref{eq:vec transport} and Theorem \ref{t-pd based retraction},
				for $\mathcal{U},\mathcal{V}\in T_{\MC{X}}\st{n,p,l},$ we have 
				\begin{equation*}
					\mathcal{T}_{\mathcal{U}}\mathcal{V} = D R_{\MC{X}}(\mathcal{U})[\mathcal{V}]
					= \frac{\mathrm{d}}{\mathrm{d}t} R_{\MC{X}}(\mathcal{U}+t\mathcal{V})\big|_{t=0}.
					% 		\\
					% 		&=& \frac{d}{dt} \mathcal{Y}(t)\big|_{t=0}.
				\end{equation*}
				This is well defined,  i.e., the t-PD decomposition of  $\MC{W}(t):=\MC{X}+\MC{U} + t\MC{V}$ is unique. To see this, it follows from Remark \ref{remark:DFT equality} that
				$$\MC{W}(t)^{\top}\ast\MC{W}(t) = \mathcal{I} + \bigxiaokuohao{\mathcal{U}+t\mathcal{V}}^{\top}\ast\bigxiaokuohao{\mathcal{U}+t\mathcal{V}}\Leftrightarrow
				(\hat{W}(t)^{(i)})^H\hat{W}(t)^{(i)}
				= {I}_{p} + \bigxiaokuohao{\hat{U}^{(i)}+t\hat{V}^{(i)}}^{H}\bigxiaokuohao{\hat{U}^{(i)}+t\hat{V}^{(i)}}, i\in [l],$$
				showing that $\hat{\MC{W}}(t)\in \TC{n}{p}{l}_*$, which together with Theorem \ref{th:t-qr based retraction} gives the desired result.
				Hence  $\mathcal W(t)$
				% 	It is easy to check that $\MC{W}(t)^{\top}\tprod\MC{W}(t) = \MC{I} + (\MC{U}+t\MC{V})^{\top}\tprod(\MC{U}+t\MC{V})\in \TSPP{p}{p}{l}$ for all $\MC{U},\MC{V}\in T_{\MC{X}}\st{n,p,l}$, which means that $\MC{W}(t) \in \T{n}{p}{l}_*$. 
				is a curve on $L^{-1}\bigxiaokuohao{\mathbb{C}_*^{n\times p \times l}}$ with $\mathcal{W}(0) = \MC{X}+\mathcal{U}$ and $\dot{\mathcal{W}}(0) = \mathcal{V}$.  
				Let $\mathcal W(t) = \mathcal Y(t) \ast\mathcal P(t)$ denote the t-PD of $\mathcal W(t)$.  Theorem \ref{th:t_pd} shows that $\mathcal Y(t)\in \st{n,p,l}$ and $\mathcal P(t)\in \TSPP{p}{p}{l}$. Hence $\mathcal{Y}(0)=R_{\MC{X}}(\mathcal{U}),\mathcal{P}(0) = R_{\MC{X}}(\mathcal{U})^{\top}\ast (\MC{X} + \mathcal U)
				%,\mathcal{V}=\dot{\mathcal{Y}}(0)\ast R_{\MC{X}}(\mathcal{U})^{\top}\ast (\MC{X} + \mathcal U)+R_{\MC{X}}(\mathcal{U})\ast \dot{\mathcal{P}}(0)
				$. Our task now is to compute $\mathcal{T}_{\mathcal{U}}{\mathcal{V}} = \frac{\mathrm{d}}{\mathrm{d}t}\MC{Y}(t)\big|_{t=0}=\dot{\mathcal Y}(0).$ 
				Since $\mathcal{I}=\mathcal{Y}(t)\ast \mathcal{Y}(t)^{\top} + \bigxiaokuohao{\mathcal{I}-\mathcal{Y}(t)\ast \mathcal{Y}(t)^{\top}}$, we have the decomposition
				\begin{equation}\label{eq:22}
					\dot{\mathcal{Y}}(t)=	\mathcal{Y}(t)\ast \mathcal{Y}(t)^{\top}\ast \dot{\mathcal{Y}}(t) + \bigxiaokuohao{\mathcal{I}-\mathcal{Y}(t)\ast \mathcal{Y}(t)^{\top}}\ast \dot{\mathcal{Y}}(t).
				\end{equation}
				It follows from Lemma \ref{prop:curve} that
				\begin{equation}\label{eq:21}
					\dot{\mathcal W}(t) = 	\dot{\mathcal Y}(t) \ast	\mathcal P(t) + 	\mathcal Y(t) \ast 	\dot{\mathcal P}(t). 
				\end{equation}
				Multiplying \eqref{eq:21} by $\mathcal{I}-\mathcal{Y}(t)\ast \mathcal{Y}(t)^{\top}$ on the left and $\mathcal P(t)^{-1}$ on the right yields
				\begin{equation}\label{eq:sec term}
					\bigxiaokuohao{\mathcal{I}-\mathcal{Y}(t)\ast \mathcal{Y}(t)^{\top}}\ast \dot{\mathcal Y}(t) = \bigxiaokuohao{\mathcal{I}-\mathcal{Y}(t)\ast \mathcal{Y}(t)^{\top}}\ast 	\dot{\mathcal W}(t)\ast (\mathcal{P}(t))^{-1},
				\end{equation}
				which is the second term of \eqref{eq:22}. It remains to derive the computational formulae for the first term of \eqref{eq:22}.
				Since $\dot{\mathcal Y}(t)$ is a tangent vector at the point $\mathcal{Y}(t)$, \eqref{eq:22} satisfies the form:
				\begin{equation*}	T_{\mathcal{Y}(t)}\st{n,p,l}  
					%		= \operatorname{ker}( Dh(\MC{X})) 
					= \bigdakuohao{ \mathcal{Y}(t)\ast \mathcal{W}(t)+\mathcal{Y}_{\perp }(t)\ast \mathcal{B}(t)\in\mathbb{R}^{n\times p \times l }\bigg|\mathcal{W}(t)\in\TSkew{p}{p}{l},~\mathcal{B}(t)\in \mathbb{R}^{(n-p)\times p \times l}},
				\end{equation*}
				where $ \hat{Y}^{(i)}_{\perp }(t)\in \mathbb{C}^{n\times (n-p)}$ is any matrix such that $\operatorname{span}(\hat{Y}^{(i)}_{\perp}(t))=\{\hat{Y}^{(i)}_{\perp}(t)\alpha\big|\alpha\in \mathbb{C}^{n-p}\}$ is the orthogonal complement of $\operatorname{span}(\hat{Y}^{(i)}(t))=\{\hat{Y}^{(i)}(t)\beta\big|\beta\in \mathbb{C}^{p}\}$.
				It is easy to check that $\innerprod{\hat{Y}^{(i)}(t)\beta}{\bigxiaokuohao{I-\hat{Y}^{(i)}(t)(\hat{Y}^{(i)}(t))^H}\alpha} = 0$ for any $\alpha\in \mathbb{C}^{n-p}$ and $\beta\in \mathbb{C}^{p}$,
				which means matrix $\bigxiaokuohao{I-\hat{Y}^{(i)}(t)(\hat{Y}^{(i)}(t))^H}$ is a choice of matrix $\hat{Y}^{(i)}_{\perp }(t)$,
				hence the
				term $\mathcal{S}(t):=\mathcal{Y}(t)^{\top} \ast \dot{\mathcal{Y}}(t)\in\TSkew{p}{p}{l}$.
				% 	Note the expression of the tangent space of the Stiefel manifold. 
				% 	$\mathcal{Y}^{\top}(t)\ast \dot{\mathcal{Y}}(t)$ is a skew symmetric
				% 	tensor and is denoted by $\mathcal{S}(t)$. 
				Next it is sufficient to obtain the formula for $\mathcal{S}(t)$.
				Multiplying \eqref{eq:21} by $\mathcal{Y}(t)^{\top}$ on the left gives
				\begin{equation}
					\mathcal{Y}(t)^{\top}\ast \dot{\mathcal W}(t) = 	\mathcal{Y}(t)^{\top}\ast \dot{\mathcal Y}(t) \ast	\mathcal P(t) + 	\dot{\mathcal P}(t) = \mathcal{S}(t) \ast	\mathcal P(t) + 	\dot{\mathcal P}(t). 
				\end{equation}
				From Lemma \ref{lemma:t-PSD manifold}, $\dot{\mathcal P}(t)\in T_{\MC{X}}\TSPP{n}{p}{l} = \TS{n}{p}{l}$ for any $\MC{X}\in \TSPP{n}{p}{l}$, hence
				\begin{equation}
					\mathcal{Y}(t)^{\top}\ast \dot{\mathcal W}(t) - \dot{\mathcal{W}}(t)^{\top}\ast 
					\mathcal Y(t)= \mathcal{S}(t) \ast	\mathcal P(t) - \mathcal{P}(t)^{\top} \ast	\mathcal S(t)^{\top} =  \mathcal{S}(t) \ast	\mathcal P(t) + \mathcal{P}(t) \ast	\mathcal S(t).
				\end{equation}
				Therefore, according to Theorem \ref{th:sylvester}, a analytical solution for $\mathcal S(t)$
				exists and is given by
				\begin{equation}\label{eq:vecs}
					\operatorname{vec}(\mathcal S(t))=\bigxiaokuohao{ \widetilde{\operatorname{bcirc}}(\mathcal P(t))^{\top}\otimes I_p+[I_p]_{l\times l}\odot\operatorname{bcirc}(\mathcal P(t)) }^{\dagger }\cdot \operatorname{vec}(\mathcal{Y}(t)^{\top}\ast \dot{\mathcal W}(t) - \dot{\mathcal{W}}(t)^{\top}\ast 
					\mathcal Y(t)).
				\end{equation}
				By substitute \eqref{eq:vecs} and \eqref{eq:sec term} into \eqref{eq:22}, we obtain
				\begin{equation*}
					\dot{\mathcal{Y}}(t)=\mathcal{Y}(t)\ast \mathcal S(t) + \bigxiaokuohao{\mathcal{I}-\mathcal{Y}(t)\ast \mathcal{Y}(t)^{\top}}\ast 	\dot{\mathcal W}(t)\ast (\mathcal{P}(t))^{-1}.
				\end{equation*}
				Thus, there holds
				\begin{equation*}
					\mathcal{T}_{\mathcal{U}}\mathcal{V} 
					% 		&=& \mathrm{D} R_{\MC{X}}(\mathcal{U})[\mathcal{V}]\\
					% 		&=& \frac{\mathrm{d}}{\mathrm{d}t} R_{\MC{X}}(\mathcal{U}+t\mathcal{V})\big|_{t=0}\\
					=
					\dot{\mathcal{Y}}(0)
					= \mathcal{Y}\ast \mathcal{S} + (\mathcal{I}-\mathcal{Y}\ast \mathcal{Y}^{\top})\ast \mathcal{V}\ast (\mathcal{Y}^{\top}\ast (\MC{X}+\mathcal{U}))^{-1},
				\end{equation*}
				where $\mathcal{Y}  = R_{\MC{X}}(\mathcal{U}),$  $\operatorname{vec}(\mathcal S)=\bigxiaokuohao{ \widetilde{\operatorname{bcirc}}(\mathcal P)^{\top}\otimes I_p+[I_p]_{l\times l}\odot\operatorname{bcirc}(\mathcal P) }^{\dagger }\cdot \operatorname{vec}(\mathcal{Y}^{\top}\ast \mathcal V - \mathcal V^{\top}\ast 
				\mathcal Y)$ and $\mathcal P = R_{\MC{X}}(\mathcal{U})^{\top}\ast (\MC{X} + \mathcal U) = (\mathcal{I}+\mathcal{U}^{\top}\ast \mathcal{U})^{\frac{1}{2}}$.
			\end{proof}
			Next we derive the vector transport as the differentiated retraction based on t-Cayley transfrom.
			\begin{theorem}[t-Cayley based vector transport]
				The vector transport on $\st{n,p,l}$ by differentiated retraction based on t-Cayley transfrom is
				\begin{equation*}
					\mathcal{T}_{\mathcal{U}}{\mathcal{V}} = \bigxiaokuohao{\mathcal{I} - \frac{1}{2}\mathcal{W}_{\mathcal{U}}}^{-1}\ast \mathcal{W}_{\mathcal{V}}\ast \bigxiaokuohao{\mathcal{I} - \frac{1}{2}\mathcal{W}_{\mathcal{U}}}^{-1}\ast \MC{X}.
				\end{equation*}
				where $\MC{X}\in \st{n,p,l}, \MC{U},\MC{V}\in T_{\MC{X}}\st{n,p,l}$, $\mathcal{W}_{\mathcal{U}} = \MC{P}\ast \mathcal{U}\ast \MC{X}^{\top} - \MC{X}\ast \mathcal{U}^{\top}\ast\MC{P}\in \TSkew{n}{n}{l}$,  $\mathcal{W}_{\mathcal{V}} = \MC{P}\ast \mathcal{V}\ast \MC{X}^{\top} - \MC{X}\ast \mathcal{V}^{\top}\ast\MC{P}\in \TSkew{n}{n}{l}$, and $\MC{P} = \mathcal{I} - \frac{1}{2}\MC{X}\ast\MC{X}^{\top}$.
			\end{theorem}
			\begin{proof}
				Consider the following retraction based on
				t-Cayley transform:
				\begin{equation*}
					R_{\MC{X}}(\mathcal{U}+t\mathcal{V}) = \bigxiaokuohao{\mathcal{I} - \frac{1}{2}\mathcal{W}_{\mathcal{U}}-\frac{t}{2}\mathcal{W}_{\mathcal{V}}}^{-1}\ast\bigxiaokuohao{\mathcal{I} + \frac{1}{2}\mathcal{W}_{\mathcal{U}}+\frac{t}{2}\mathcal{W}_{\mathcal{V}}}\ast\MC{X}.
				\end{equation*}
				Differentiating both sides of
				$$
				\bigxiaokuohao{\mathcal{I} - \frac{1}{2}\mathcal{W}_{\mathcal{U}}-\frac{t}{2}\mathcal{W}_{\mathcal{V}}}\ast R_{\MC{X}}(\mathcal{U}+t\mathcal{V}) = \bigxiaokuohao{\mathcal{I} + \frac{1}{2}\mathcal{W}_{\mathcal{U}}+\frac{t}{2}\mathcal{W}_{\mathcal{V}}}\ast\MC{X}
				$$
				with respect to $t$, we have
				$$
				-\frac{1}{2}\mathcal{W}_{\mathcal{V}}\ast R_{\MC{X}}(\mathcal{U}+t\mathcal{V})
				+\bigxiaokuohao{\mathcal{I} - \frac{1}{2}\mathcal{W}_{\mathcal{U}}-\frac{t}{2}\mathcal{W}_{\mathcal{V}}}
				\ast \frac{\mathrm{d}}{\mathrm{d} t}R_{\MC{X}}(\mathcal{U}+t\mathcal{V}) 
				= \frac{1}{2}\mathcal{W}_{\mathcal{V}}\ast\MC{X}.
				$$
				According to \eqref{eq:vec transport}, the vector transport by differentiated retraction is 
				\begin{eqnarray*}
					\mathcal{T}_{\mathcal{U}}{\mathcal{V}} &=& \frac{\mathrm{d}}{\mathrm{d} t}R_{\MC{X}}(\mathcal{U}+t\mathcal{V})\big|_{t=0}\\
					&=& \frac{1}{2}\bigxiaokuohao{\mathcal{I} - \frac{1}{2}\mathcal{W}_{\mathcal{U}}}^{-1}\ast \mathcal{W}_{\mathcal{V}}\ast \bigxiaokuohao{\MC{X}+R_{\MC{X}}(\mathcal{U})}\\
					&=& \frac{1}{2}\bigxiaokuohao{\mathcal{I} - \frac{1}{2}\mathcal{W}_{\mathcal{U}}}^{-1}\ast \mathcal{W}_{\mathcal{V}}\ast \bigxiaokuohao{\MC{X}+\bigxiaokuohao{\mathcal{I} - \frac{1}{2}\mathcal{W}_{\mathcal{U}}}^{-1}\ast\bigxiaokuohao{\mathcal{I} + \frac{1}{2}\mathcal{W}_{\mathcal{U}}}\ast\MC{X}}\\
					&=& \bigxiaokuohao{\mathcal{I} - \frac{1}{2}\mathcal{W}_{\mathcal{U}}}^{-1}\ast \mathcal{W}_{\mathcal{V}}\ast \bigxiaokuohao{\mathcal{I} - \frac{1}{2}\mathcal{W}_{\mathcal{U}}}^{-1}\ast \MC{X},
				\end{eqnarray*}
				where the last equality comes from $\MC{W}_{\MC{U}}\ast \MC{X} = \MC{U}$.
			\end{proof}
			Finally, we introduce a  isometric vector transport.
			\begin{theorem}[isometric vector transport]
				The following formulae is an isometric vector transport on $\st{n,p,l}$ 
				\begin{equation}\label{eq:cayley vec transport}
					\mathcal{T}_{\mathcal{U}}{\mathcal{V}} = \bigxiaokuohao{\mathcal{I} - \frac{1}{2}\mathcal{W}_{\mathcal{U}}}^{-1}\ast\bigxiaokuohao{\mathcal{I} + \frac{1}{2}\mathcal{W}_{\mathcal{U}}}\ast\mathcal{V},
				\end{equation}
				where $\MC{X}\in \st{n,p,l}, \MC{U},\MC{V}\in T_{\MC{X}}\st{n,p,l}$, $\mathcal{W}_{\mathcal{U}} = \MC{P}\ast \mathcal{U}\ast \MC{X}^{\top} - \MC{X}\ast \mathcal{U}^{\top}\ast\MC{P}\in \TSkew{n}{n}{l}$ and $\MC{P} = \mathcal{I} - \frac{1}{2}\MC{X}\ast\MC{X}^{\top}$. 
			\end{theorem}
			\begin{proof}Consider the following retraction based on t-Cayley transform:
				\begin{equation*}
					R_{\MC{X}}(\mathcal{U}) = \bigxiaokuohao{\mathcal{I} - \frac{1}{2}\mathcal{W}_{\mathcal{U}}}^{-1}\ast\bigxiaokuohao{\mathcal{I} + \frac{1}{2}\mathcal{W}_{\mathcal{U}}}\ast\MC{X}.
				\end{equation*}
				Since  $\mathcal{W}_{\mathcal{U}}\in \TSkew{n}{n}{l}$ and $(\MC{I}-\MC{A})*(\MC{I}+\MC{B}) = (\MC{I}+\MC{B})*(\MC{I}-\MC{A})$ for all $\MC{A},\MC{B}\in \T{n}{p}{l}$, we have 
				\begin{equation*}
					\mathcal{T}_{\mathcal{U}}{\mathcal{V}}^{\top}\ast R_{\MC{X}}(\mathcal{U}) + R_{\MC{X}}(\mathcal{U})^{\top}\ast \mathcal{T}_{\mathcal{U}}{\mathcal{V}} = \MC{V}^{\top}\ast \MC{X}+\MC{X}^{\top}\ast \MC{V} = \MC{O},
				\end{equation*}
				which combines Theorem \ref{th:dim} lead to $\mathcal{T}_{\mathcal{U}}{\mathcal{V}}\in T_{R_{\MC{X}}(\mathcal{U})}\st{n,p,l}$.
				% and therefore $R_{\MC{X}}(\mathcal{U})$ is the foot of $\mathcal{T}_{\mathcal{U}}{\mathcal{V}}$. 
				% Thus we have proved that the t-Cayley transform $R_{\MC{X}}(\mathcal{U})$ is a retraction associated with $\mathcal{T}$. 
				It is easy to check that $\mathcal{T}_{\MC{O}_\MC{X}}{\mathcal{V}} = \mathcal{V}$ 
				and  
				% $\MC{T}_{\mathcal{U}}(a\mathcal{V}_1+b\mathcal{V}_2) 
				% = a\mathcal{T}_{\mathcal{U}}\mathcal{V}_1+b\mathcal{T}_{\mathcal{U}}\mathcal{V}_2
				% $ 
				for $\mathcal{V}_1,\mathcal{V}_2 \in T_{\MC{X}}\st{n,p,l}$.
				The smoothness follows immediately from \eqref{eq:cayley vec transport}. According to Definition \ref{def:vector_transport}, $\MC{T}$ is indeed a vector transport on $\st{n,p,l}$. It follows from the skew-symmetry of $\mathcal{W}_{\mathcal{U}}$ that $
				\left\langle\mathcal{T}_{\MC{U}}(\MC{V}), \mathcal{T}_{\MC{U}}(\MC{V})\right\rangle_{R_{\MC{X}}(\MC{U})}=\langle\MC{V}, \MC{V}\rangle_{\MC{X}}
				$
				for all $\MC{U}, \MC{V} \in T_{\MC{X}} \st{n,p,l}$.
			\end{proof}

\section{Examples of \eqref{prob:prototype}}\label{sec:appl}
We present some (potential) examples of \eqref{prob:prototype} and related problems in this section. 
%The prototype problem
%In this section, we present some applications case of \eqref{prob:prototype}.

%$\TH$
\paragraph{Best approximation.} 
Given $\mathcal A\in\T{n}{p}{l}$, its best $k$-term approximation was given in \cite[Thm. 4.3]{kilmer2011factorization}. Such a problem can also be formulated as ($k\leq \min\{n,p\}$):
%\begin{align*}
$	\min_{\mathcal U\in\st{n,k,l}, \mathcal S\in\mathbb R^{k\times k\times l},\mathcal V \in \st{p,k,l}}~ \bigfnorm{\mathcal A - \mathcal U\tprod\mathcal S\tprod\mathcal V^{\top}}^2.$
%\end{align*}
By using Proposition \ref{prop:for_realization} to eliminate the variable $\mathcal S$, such a problem is equivalent to 
\begin{align}\label{prob:t-svd_truncated}
	\max_{\mathcal U\in \st{n,k,l},\mathcal V \in \st{p,k,l} }~\bigfnorm{\mathcal U^{\top}\tprod \mathcal A\tprod\mathcal V  }^2.  \end{align}
By denoting $\mathcal W:= \left[\begin{smallmatrix}
	\mathcal{O}& \mathcal U \\
	\mathcal V & \mathcal{O} 
\end{smallmatrix}\right]\in\mathbb R^{(n+p)\times 2k\times l}$ with $\mathcal O$ the zero tensor of proper size and reformulating the objective function accordingly, such a problem is of the form \eqref{prob:prototype}. 

%\paragraph{Best approximation to symmetric tensors}
Given $\mathcal A\in\T{n}{n}{l}$ which is symmetric, its eigenvalue decomposition was given in \cite{zheng2021t}. Correspondingly, one can define its best $k$-term symmetric approximation as $(k\leq n)$:
%\begin{align*}
$	\min_{\mathcal U\in\st{n,k,l},\mathcal S\in\mathbb R^{n\times n\times l} }~\bigfnorm{\mathcal A-\mathcal U\tprod\mathcal S\tprod\mathcal U^{\top}}^2. $
%\end{align*}
Similarly, the variable $\mathcal S$ can be eliminated and the problem is equivalent to 
\begin{align}\label{prob:t-evd_truncated}
	& \max_{\mathcal U\in\st{n,k,l}} \bigfnorm{\mathcal U^{\top}\tprod\mathcal A\tprod\mathcal U}^2,
\end{align}
which is of the form \eqref{prob:prototype}. Similar to their matrix counterparts, \eqref{prob:t-svd_truncated} and \eqref{prob:t-evd_truncated} are also equivalent to
\begin{align}\label{prob:t-evd_or_svd_truncated_trace}
	\min_{\mathcal U\in \st{n,k,l},\mathcal V\in \st{p,k,l} }-\trace{\mathcal U^{\top}\tprod\mathcal A\tprod\mathcal V }~{\rm and}~\min_{\mathcal U\in \st{n,k,l}}-\trace{\mathcal U^{\top}\tprod\mathcal A\tprod\mathcal U}.
\end{align}

\paragraph{Best approximation  with missing entries.}
In real-world applications, we are sometimes faced the situation that part of the observation data is missing; this   troubles the approximation problem.  Similar to tensor completion, one can formulate the problem as  
\begin{align*}
	&\min_{\mathcal U\in\st{n,k,l}, \mathcal S\in\mathbb R^{k\times k\times l},\mathcal V \in\st{p,k,l} } \bigfnorm{\Omega \circledast  \bigxiaokuohao{\mathcal A - \mathcal U\tprod\mathcal S\tprod\mathcal V^{\top}}}^2,
\end{align*}
where $\circledast $ denotes the Hadamard operator and $\Omega\in\T{n}{p}{l}$ is a 0-1 tensor whose entries take 1 if the associated entries of $\mathcal A$ are available and 0 otherwise. The variable $\mathcal S$ cannot be eliminated and such a problem is a variant of \eqref{prob:prototype} that can be possibly solved in an alternating fashion. 
%\paragraph{Best approximation to symmetric tensors with missing entries}

In the symmetric tensors setting, similar troubles might occur. Such a problem is thus formulated as 
\begin{align}\label{prob:missing_entries}
	&\min_{\mathcal U\in\st{n,k,l},\mathcal S\in\mathbb R^{k\times k\times l} }~\bigfnorm{\Omega\circledast \bigxiaokuohao{\mathcal A-\mathcal U\tprod\mathcal S\tprod\mathcal U^{\top}}}^2. 
\end{align}

\paragraph{Joint f-diagonalization.}
The  connection of simultaneous f-diagonalization to commutative tensors was discovered theoretically in  \cite{miao2021t,liu2020study}. For more than two tensors, the joint f-diagonalization is difficult in theory. In the matrix case, however, this can be resolved numerically by formulating the problem as optimization models over orthogonal or non-orthogonal constraints; see, e.g, \cite{bunse1993numerical}. Similarly, for joint f-diagonalization of more than two tensors of size $n\times n\times l$, one can consider the following optimization models:
\begin{align}\label{prob:joint_diag}
	\min_{\mathcal U\in\st{n,k,l}} \sum^N_{i=1}\nolimits{\rm off}\bigxiaokuohao{\mathcal U^{\top}\tprod\mathcal A_i\tprod\mathcal U},%~{\rm s.t.}~\mathcal U^{\top}\tprod\mathcal U = \mathcal I,	
\end{align}
where $\mathcal A_i\in\T{n}{n}{l}$, $i=1,\ldots,N$, and ${\rm off}(\mathcal X) = \sum_{i_3=1}^l\sum_{1\leq i_1\neq i_2\leq k}^k \bigxiaokuohao{x_{i_1i_2i_3}}^2 $ or $=\sum^l_{i_3=1}\sum^k_{1\leq i_1\neq i_2\leq k} \bigjueduizhi{x_{i_1i_2i_3}}$,  which is analogous to its matrix counterpart.
%, or
%\begin{align*}
%	&\min_{\mathcal U\in\mathbb R^{n\times k\times l}, \mathcal S_i\in\mathbb R^{k\times k\times l} } \sum^N_{i=1}\bigfnorm{\mathcal A_i - \mathcal U\tprod\mathcal S_i\tprod\mathcal U^{\top}}^2\\
%	&~~~~~~~~~~{\rm s.t.}~~~~~~~~~~~~~\mathcal U^{\top}\tprod\mathcal U =\mathcal I, \mathcal S_i~{\rm f-diagonal}. 
%\end{align*}

\paragraph{Joint t-SVD.}
We first consider the matrix cases. Assume that $N$ matrices $A_1,\ldots,A_N\in\mathbb R^{n\times p}$ are given, which are regarded as $N$ samples. The joint SVD is to find common orthogonal matrices $U\in\mathbb R^{n\times k},V\in\mathbb R^{p\times k}$, $k\leq \min\{n,p\}$, such that $U^{\top}A_iV$ are as diagonal as possible. The joint SVD is useful in image representation and dimension reduction; see, e.g. \cite{pesquet2001joint,shashua2001linear}.  Now assume that the samples are third-order tensors $\mathcal A_1,\ldots,\mathcal A_N\in\T{n}{p}{l}$. To perform dimension reduction on these samples, it is quite natural to extend such an idea to obtain the following joint t-SVD models:
\begin{align*}
	\min_{\mathcal U\in\st{n,k,l},\mathcal V\in\st{p,k,l}} \sum^N_{i=1}\nolimits{\rm off}\bigxiaokuohao{\mathcal U^{\top}\tprod\mathcal A_i\tprod\mathcal V}.
\end{align*}
%or
%\begin{align*}
%	&\min_{\mathcal U\in\mathbb R^{n\times k\times l}, \mathcal S_i\in\mathbb R^{k\times k\times l} ,\mathcal V\in\mathbb R^{n\times k\times l}} \sum^N_{i=1}\bigfnorm{\mathcal A - \mathcal U\tprod\mathcal S_i\tprod\mathcal V^{\top}}^2\\
%	&~~~~~~~~~~{\rm s.t.}~~~~~~~~~~~~~~~~~~~~~~~~\mathcal U^{\top}\tprod\mathcal U =\mathcal I,\mathcal V^{\top}\tprod\mathcal V=\mathcal I, \mathcal S_i~{\rm f-diagonal}. 
%\end{align*}

\paragraph{Sparse tensor PCA.} 
Sparse tensor PCA was introduced in \cite{allen2012sparse} and then its solution methods were  further studied in \cite{wang2020sparse,mao2022several}. The purpose is to find sparse principal  components for higher-order data. The model of \cite{allen2012sparse} is based on the canonical polyadic format and is approximately solved by a deflation approach, where the orthogonality cannot be assured. However, by using t-product, it is more natural to directly extend the sparse matrix PCA to the third-order tensor setting, leading to the following model:
\begin{align}\label{prob:stpca}
	\min_{\mathcal U\in\st{n,k,l}} -\trace{ {\mathcal U^{\top}\tprod\mathcal A\tprod\mathcal A^{\top}\tprod\mathcal U}} + \rho\bignorm{\mathcal U}_1,
\end{align}
where $\mathcal A\in\mathbb R^{n\times p\times l}$ is the data tensor consisting of samples,   $\bignorm{\mathcal U}_1=\sum_{i_1,i_2,i_3=1}^{n,k,l}\bigjueduizhi{u_{i_1i_2i_3}}$, and $\rho>0$. Based on the formulas derived in this work, it is expected that the recently developed Riemannian algorithms, such as \cite{chen2020proximal,huang2021riemannian}, can be applied to the above problem without many modifications.

Besides the above basic examples, some applications have been or can be formulated as optimization over the tensor Stiefel manifold in the literature; see, e.g., \cite{lin2020tensor,xu2021t,xu2022iterative,hoover2018multilinear,ozdemir2022high}. For instance, \cite{lin2020tensor} proposed a tensor subspace representation method for hyperspectral image denoising, where the model is exactly an optimization over the tensor Stiefel manifold. 
To save space, we do not introduce them in detail here; interested readers can be referred to them.

We make the following remark to end this section. 
\begin{remark} Considering the relation \eqref{eq:IDFT}, one may wonder whether \eqref{prob:prototype} can be solved in the Fourier domain. That is to say, since $\MC{X} = L^{-1}L(\MC{X}) = L^{-1}(\hat{\MC{X}})$, \eqref{prob:prototype} is equivalent to minimizing $f(L^{-1}(\hat{\MC{X}}))$ subject to ${ (\hat{X}^{(i)})^H\hat{X}^{(i)} = I_p, i = 1,\cdots,\lceil  \frac{l+1}{2}  \rceil }, 
	\hat{X}^{(i)} = \conj{\hat{X}^{(l+2-i)}},i=1+ \lceil  \frac{l+1}{2}  \rceil,\ldots,l$, which is minimizing a real-valued function over the product of $k$ real matrix Stiefel manifolds and $\frac{l-k}{2}$ complex matrix Stiefel manifolds (when $l$ is even, $k = 2$; when $l$ is odd, $k = 1$). Comparing with \eqref{prob:prototype}, solving this problem may have some drawbacks. Firstly, it is in the complex field which is more complicated to analyze  than \eqref{prob:prototype}; secondly, the introduction of $L^{-1}$ in the problem might destroy certain structure of the problem; thirdly, if $f$ is nonsmooth, such as $f$ involves the $\ell_1$ norm    as that in the sparse tensor PCA model \eqref{prob:stpca}, then as far as we know, no Riemannian algorithms can handle problems involving the term $\|{L^{-1}(\hat{\MC{X}})}\|_1$. 
\end{remark}
%\paragraph{$\ell_1$-PCA}

\color{black}

\section{Preliminary Numerical Examples}\label{sec:experiments} We conducted preliminary numerical experiments to verify the derived formulas. To this end,
we applied a Riemannian CG algorithm summarized in Algorithm \ref{alg:Riemannian CG} with various retractions  to the four problems  introduced in Section \ref{sec:appl}, namely, \eqref{prob:t-evd_or_svd_truncated_trace}, \eqref{prob:missing_entries}, \eqref{prob:joint_diag}, and \eqref{prob:stpca}.
%Four test problems are best approximation of $\MC{A}\in \TSP{n}{n}{l}$, best approximation of $\MC{A}\in \T{n}{n}{l}$ with missing entries, joint f-diagonalization of $\MC{A}_i \in \T{n}{n}{l}, i = 1,\cdots, I$, and sparse tensor PCA of $\MC{A}\in \T{n}{p}{l}$, respectively.
The considered retractions were the t-QR based one \eqref{eq:t-qr based retraction}, the t-PD based one \eqref{eq:t-pd based retraction}, and the t-Cayley transform based one \eqref{eq:cayley based retraction}, while the vector transports were constituted by the orthogonal projector \eqref{eq:Orthogonal projector based vector transport} associated with the various retractions mentioned above. 
%These methods are different from each other only in ways of constructing retractions and vector transports. 
All the experiments were
conducted on an Intel i7 CPU desktop computer with 16 GB of RAM. The supporting software is Matlab R2022a. Tensorlab \cite{tensorlab2013} and Tensor-Tensor Product Toolbox \cite{lu2018tproduct} were employed for tensor operations.

Algorithm \ref{alg:Riemannian CG} was modified from \cite{zhu2017riemannian} to the t-product setting. In the algorithm, we set $\beta_{k+1} = \text{min}\{\beta^{\mathrm{FR}}_{k+1},\beta^{\mathrm{D}}_{k+1}\}$ where
\begin{equation}\label{eq:step_Dai}
	\beta^{\mathrm{D}}_{k+1} = \frac{||\operatorname{grad}f(\MC{X}_{k+1})||^2_{\MC{X}_{k+1}}}{\text{max}\{\innerprod{\operatorname{grad}f(\MC{X}_{k+1})}{\MC{T}_{\alpha_k\MC{Z}_k}(\MC{Z}_k)}_{\MC{X}_{k+1}}-\innerprod{\operatorname{grad}f(\MC{X}_k)}{\MC{Z}_k}_{\MC{X}_k},-\innerprod{\operatorname{grad}f(\MC{X}_k)}{\MC{Z}_k}_{\MC{X}_k}\}}
\end{equation}
is a generalization of Dai’s nonmonotone parameter \cite{zhu2017riemannian} and 
\begin{equation}\label{eq:step_FR}
	\beta^{\mathrm{FR}}_{k+1} = {||\operatorname{grad}f(\MC{X}_{k+1})||^2_{\MC{X}_{k+1}}}/{||\operatorname{grad}f(\MC{X}_{k})||^2_{\MC{X}_k}}
\end{equation}
is the Fletcher–Reeves parameter. The steplength $\alpha_{k+1} = \text{max}\{\text{min}\{\alpha^{\mathrm{BB}}_{k+1},\alpha_{\text{max}}\},\alpha_{\text{min}}\}$ where 
\begin{equation}\label{eq:step_BB}
	\alpha^{\mathrm{BB}}_{k+1} = {\innerprod{\MC{S}_{k}}{\MC{S}_{k}}_{\MC{X}_k}}/{\big|\innerprod{\MC{S}_{k}}{\MC{V}_{k}}_{\MC{X}_k}\big|},
\end{equation}
with $\MC{S}_{k} = -\alpha_k\MC{T}_{\alpha_k\MC{Z}_k}(\operatorname{grad}f(\MC{X}_k)), ~\MC{V}_{k} = \operatorname{grad}f(\MC{X}_{k+1})+\alpha_k^{-1}\MC{S}_{k}$, which is a Riemannian generalization of the Barzilai-Borwein steplength \cite{iannazzo2018riemannian}. The iterative algorithms were stopped if $||\MC{X}_{k+1}-\MC{X}_k||_F/\sqrt{n}<10^{-6}$ or $|f(\MC{X}_{k+1})-f(\MC{X}_k)|/(1+|f(\MC{X}_k)|)<10^{-12}$ or $k > 1000$. In the algorithm, the initial steplength $\alpha_0=10^{-3}$, $(\alpha_{\min},\alpha_{\max})=(10^{-20},1)$; $\lambda = 0.2$, and $\delta = 10^{-4}$. We remark that in the problem \eqref{prob:stpca}, for simplicity, we just used the subgradient of the objective function in our computation; in \eqref{prob:missing_entries}, $\mathcal U$ and $\mathcal S$ are computed in an alternating fashion, where $\mathcal U$ is computed by one step of Algorithm \ref{alg:Riemannian CG}.

\iffalse
\begin{table}[htbp]
	\setlength{\abovecaptionskip}{0.5 cm}
	\setlength{\belowcaptionskip}{0.2cm}
	\renewcommand{\arraystretch}{1.5}
	\setlength{\tabcolsep}{20pt}
	\centering
	\caption{Parameters for Algorithm \ref{alg:Riemannian CG}}
	\begin{mytabular} {ccc}     
		\toprule
		Parameter &
		Value & Description \\
		\toprule 
		$\epsilon_x$ & $10^{-6}$ & Tolerance for the difference of variables\\
		$\epsilon_f$ & $10^{-12}$ & Tolerance for the difference of function values\\
		$\lambda$ & $0.2$ & Steplength shrinkage factor\\
		$\delta$ & $10^{-4}$ & Nonmonotone line search Armijo-Wolfe constant\\
		$\alpha$ & $10^{-3}$ & Initial steplength\\
		$\alpha_{\text{min}}$ & $10^{-20}$ & Lower threshold for steplength\\
		$\alpha_{\text{max}}$ & $1$ & Upper threshold for steplength\\
		$k_{\text{max}}$ & 1000 & Maximal number of iterations\\
		\bottomrule
	\end{mytabular}
	\label{tab:1}
\end{table}
\fi

In the experiments, we set $(n,p,l)= (50, 10, 8)$. The data tensors in \eqref{prob:t-evd_or_svd_truncated_trace}  is set to be $\MC{A} = \MC{V}^{\top}\tprod \MC{V}$, where $\MC{V} = \texttt{randn}(n,n,l)$. In  \eqref{prob:missing_entries},  $\MC{A} = \MC{X}\tprod \MC{W}\tprod \MC{X}^{\top}$ where $\MC{X}\in \st{n,p,l}$ and the  f-diagonal tensor $\MC{W}\in \T{p}{p}{l}$ were both randomly generated; the entries were randomly missing with missing ratio being  $30\%$. In \eqref{prob:joint_diag},  the tensors were constructed as   $\MC{A}_i =  \MC{X}\tprod\MC{C}_i\tprod\MC{X}^{\top}+r\frac{\MC{E}_i}{||\MC{E}_i||_F}, i\in[N]$, where the number of samples $N = 3$, $\MC{X}\in \st{n,p,l}$, the f-diagonal tensors $\MC{C}_i\in \T{p}{p}{l}$, and the noise term $\MC{E}_i\in \T{n}{n}{l}$ were
all randomly generated with the noise level $r = 0.1$. In \eqref{prob:stpca}, $\MC{A} = \texttt{randn}(n,p,l)$  and the  parameter $\rho = 0.1$. 
% Values of the input parameters for Algorithm \ref{alg:Riemannian CG} are summarized in Table \ref{tab:1}. 
All the initial points for the algorithm $\MC{X}_0$ were randomly generated feasible points. For each case, we randomly generated 50 instances, and the averaged results are presented.  Figure \ref{fig:obj} shows the curves of the objective values of the four test problems versus iterations, whose colors are respectively green (t-QR based retraction), blue (t-PD based retraction), cyan (t-Cayley transform based retraction).
\begin{figure}[h] %data from 8-Jun-2022-18-46-20.mat
	\centering
	\includegraphics[width=.9\textwidth]{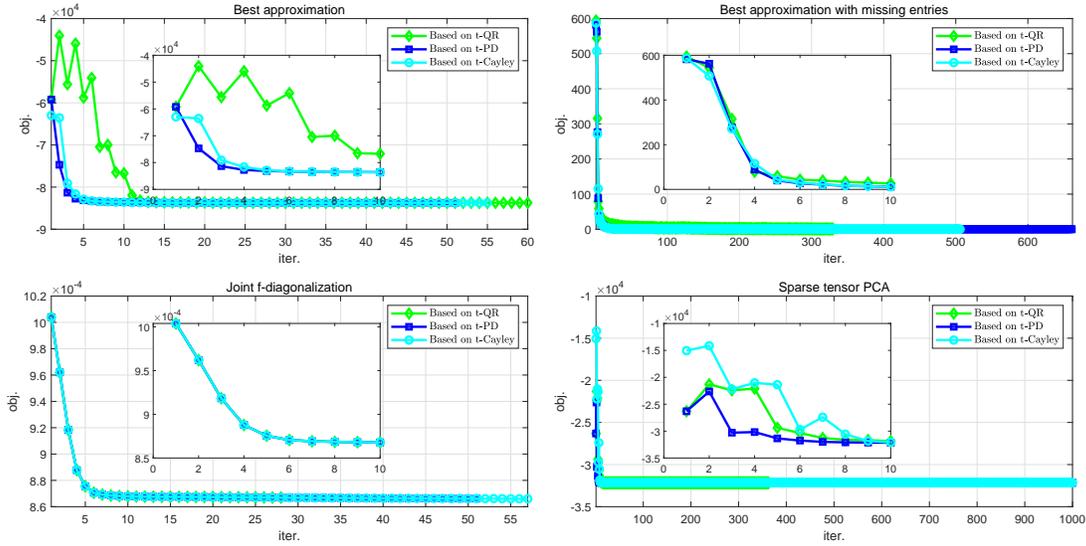}
	\caption{Objective values on four test problem versus iterations with different retractions
	} 
	\label{fig:obj} 
\end{figure}

The performance and  the feasibility of  Riemannian CG   on the   problems with various retractions are reported in Tables \ref{tab:2} and   \ref{tab:3}, respectively. ``obj.'' stands for the objective value,  ``feasi.'' specifies the feasibility $||\MC{X}^{\top}\tprod \MC{X} - \MC{I}||_F$, ``iter.'' means the iterations, ``time.'' represents the CPU time where the unit is second, and ``re.'' means the relative error: specifically, ``re.'' $=||\MC{X}\tprod \MC{W}\tprod \MC{X}^{\top}-\MC{U}_{\text{out}}\tprod \MC{S}_{\text{out}}\tprod \MC{U}_{\text{out}}^{\top}||_F/||\MC{W}||_F$ in \eqref{prob:missing_entries} and  ``re.''$=\bigxiaokuohao{\sum_{i=1}^N||\MC{X}\tprod\MC{C}_i\tprod\MC{X}^{\top}-\MC{U}_{\text{out}}\tprod\MC{U}_{\text{out}}^{\top}\tprod\MC{A}_i\tprod \MC{U}_{\text{out}}\tprod\MC{U}_{\text{out}}^{\top}||_F/||\MC{C}_i||_F}/N$ in  \eqref{prob:joint_diag}, where
$\MC{U}_{\text{out}}$ is generated by the Algorithm \ref{alg:Riemannian CG} and $\MC{S}_{\text{out}}$ is generated by the nonmonotone gradient method with Barzilai-borwein step size on $\T{p}{p}{l}$ in an alternating fashion. Empirically, we can observe that the algorithm converges in all the examples, indicating that the derived formulas (using orthogonal projector based vector transport) are correct. 
\begin{table}[htbp]%data from 8-Jun-2022-18-46-20.mat
	\setlength{\abovecaptionskip}{0.5 cm}
	\setlength{\belowcaptionskip}{0.2cm}
	\renewcommand{\arraystretch}{2.5}
	\setlength{\tabcolsep}{5pt}
	\centering
	\caption{The performance of Riemannian CG method on four test problem with various retractions}
	\begin{mytabular} {ccccccccccccc}      
		\toprule
		\multicolumn{1}{c}{}&\multicolumn{3}{c}{Best approximation} & \multicolumn{3}{c}{ with missing entries} & \multicolumn{3}{c}{Joint f-diagonalization} & \multicolumn{3}{c}{Sparse tensor PCA} \\
		\cmidrule(r){2-4} \cmidrule(r){5-7} \cmidrule(r){8-10} \cmidrule(r){11-13}
		Retration &
		obj. & iter.   & time. & re. & iter.   & time. & re. & iter.   & time. & obj. & iter.   & time.\\
		\toprule    
		t-QR  & -8.37E+04 & 6.00E+01 & 5.39E-01 & 1.59E-01 & 3.30E+02 & 1.72E+01 & 2.24E-03 & 2.80E+01 & 2.31E+00 & -3.22E+04 & 3.63E+02 & 2.70E+00 \\
		t-PD  & -8.37E+04 & 5.10E+01 & 4.98E-01 & 1.62E-02 & 6.62E+02 & 3.38E+01 & 2.31E-03 & 5.10E+01 & 4.18E+00 & -3.22E+04 & 1.00E+03 & 8.24E+00 \\
		t-Cayley & -8.37E+04 & 5.50E+01 & 1.04E+00 & 6.83E-03 & 5.06E+02 & 3.08E+01 & 2.32E-03 & 5.70E+01 & 5.19E+00 & -3.22E+04 & 1.00E+03 & 1.78E+01 \\
		\bottomrule
	\end{mytabular}
	\label{tab:2}
\end{table}
\begin{table}[htbp]%data from 8-Jun-2022-18-46-20.mat
	\setlength{\abovecaptionskip}{0.5 cm}
	\setlength{\belowcaptionskip}{0.2cm}
	\renewcommand{\arraystretch}{2.5}
	\setlength{\tabcolsep}{5pt}
	\centering
	\caption{The feasibility of Riemannian CG method on four test problem with various retractions}
	\begin{mytabular} {cccccccccccc}      
		\toprule
		\multicolumn{3}{c}{Best approximation} & \multicolumn{3}{c}{ with missing entries} & \multicolumn{3}{c}{Joint f-diagonalization} & \multicolumn{3}{c}{Sparse tensor PCA} \\
		\cmidrule(r){1-3} \cmidrule(r){4-6} \cmidrule(r){7-9} \cmidrule(r){10-12}
		t-QR & t-PD   & t-Cayley & t-QR & t-PD   & t-Cayley & t-QR & t-PD   & t-Cayley & t-QR & t-PD   & t-Cayley\\
		\toprule 
		1.47E-15 & 7.47E-15 & 2.26E-14 & 1.29E-15 & 5.06E-15 & 1.67E-13 & 1.06E-15 & 5.47E-15 & 2.09E-14 & 1.40E-15 & 3.95E-15 & 3.95E-13 \\
		\bottomrule
	\end{mytabular}
	\label{tab:3}
\end{table}

% \begin{small}
	\begin{algorithm}[!h]
		\caption{A Riemannian nonmonotone conjugate gradient  method on $\st{n,p,l}$}\label{alg:Riemannian CG}
		% 	\hspace*{0.02in} 
		% 	{\bf Require:} Tensor Stiefel $\st{n,p,l} $; vector transport $\mathcal{T}$ on $\st{n,p,l} $ with associated retraction $R$; real-valued \hspace*{0.75in} function $f$ on $\st{n,p,l} $.\\
		%  	\hspace*{0.02in} {\bf Goal: } 
		%  		$
		%  		\min_{\MC{X}\in\st{n,p,l}} f(\MC{X})= -{\rm tr}\bigxiaokuohao{\MC{X}^{\top}\tprod\mathcal A\tprod\MC{X}} = \sum_{i = 1}^l\operatorname{sum}( \operatorname{svds}(\hat{A}^{(i)},p)),
		%  		$
		%  where $\MC{A}\in \TSP{n}{n}{l}$\\
		\hspace*{0.02in} {\bf Input:} $\MC{X}_0\in \st{n,p,l} $, 
		% 	$\MC{Z}_0 = - \operatorname{grad}f(\MC{X}_0)$,
		$\alpha_0=10^{-3}$,  $(\alpha_{\text{min}},\alpha_{\text{max}}) = (10^{-20},1)$.\\
		\hspace*{0.02in} {\bf Output:}  $\{\MC{X}_k\}, \{f(\MC{X}_k)\}$ and $\{\operatorname{grad}f(\MC{X}_k)\}$.
		\label{Alg:R-CG}
		\begin{algorithmic}[1]%一行一个标行号
			\STATE  Set $k = 0,~\MC{Z}_0 = - \operatorname{grad}f(\MC{X}_0)$.
			\WHILE{$||\MC{X}_{k+1}-\MC{X}_k||_F/\sqrt{n}>10^{-6}$ or $|f(\MC{X}_{k+1})-f(\MC{X}_k)|/(1+|f(\MC{X}_k)|)>10^{-12}$ or $k < 1000$}
			\IF{$f(R_{\MC{X}_k}(\alpha_k\MC{Z}_k))\leq\operatorname{max}\{f(\MC{X}_k),f(\MC{X}_{k-1})\}+10^{-4}\alpha_k\innerprod{\operatorname{grad}f(\MC{X}_k)}{\MC{Z}_k}_{\MC{X}_k}$ }
			\STATE Set
			$\MC{X}_{k+1} = 			R_{\MC{X}_k}(\alpha_k\MC{Z}_k).$
			\ELSE
			\STATE Set $\alpha_k \leftarrow  0.2\alpha_k$ and go to line 3.
			\ENDIF
			\STATE Compute $\MC{Z}_{k+1} = -\operatorname{grad}f(\MC{X}_{k+1})+\beta_{k+1}\MC{T}_{\alpha_k\MC{Z}_k}(\MC{Z}_k)$, where $\beta_{k+1} = \text{min}\{\beta^{\mathrm{FR}}_{k+1},\beta^{\mathrm{D}}_{k+1}\}$.
			\STATE Compute $\alpha_{k+1} = \text{max}\{\text{min}\{\alpha^{\mathrm{BB}}_{k+1},\alpha_{\text{max}}\},\alpha_{\text{min}}\}$.
			\STATE Set $k \leftarrow  k+1$.
			\ENDWHILE
		\end{algorithmic}
	\end{algorithm}

\section{Concluding Remarks} \label{sec:conclusion}
Optimization over the matrix Stiefel manifold draws much attention in recent years. With the properties and decompositions built upon the t-product of third-order tensors, we study computation over the tensor Stiefel manifold $\st{n,p,l}=\{\MC{X}\in\T{n}{p}{l} \mid \MC{X}^{\top}\tprod \MC{X} = \mathcal I  \}$ in this work. 
Firstly, it is shown that $\st{n,p,l}$  endowed with the Frobenius norm is known to admit a Riemannian manifold structure; then,   explicit expressions  over $\st{n,p,l}$, such as the tangent space,   Riemannian gradient,  Riemannian Hessian, several retractions, and  vector transports,  are derived, which may serve as building blocks for designing and analyzing Riemannian algorithms over the tensor Stiefel manifold. As byproducts, we define the skew tensors, t-polar decomposition, and obtain the analytical solution to the tensor Sylvester equation in the t-product sense.

We also remark that although this work is focused on the   tensor Stiefel manifold in the sense of t-product, it is straightforward to derive similar results in the sense of the more general tensor-tensor product \cite{kilmer2021tensor}.

In the future, it would be necessary to further study properties and computation concerning the tensor Stiefel manifold, and it would be interesting to find more instances of the form \eqref{prob:prototype}. In particular, we prefer to systematically study the sparse tensor PCA model \eqref{prob:stpca} and algorithms in our future work.

{\footnotesize \section*{Acknowledgment} This work was supported by   National Natural Science Foundation of China Grant 12171105,     Fok Ying Tong Education Foundation Grant 171094, and the special foundation for Guangxi Ba Gui Scholars. All authors  equally contributed to this research.
	}
\color{black}
\bibliography{orth_tensor,t_prod,tensor,tensorCompletion,manifold,sparse_pca,ref,texp_ref}

\begin{thebibliography}{10}

\bibitem{absil2009optimization}
P.-A. Absil, R.~Mahony, and R.~Sepulchre.
\newblock {\em Optimization algorithms on matrix manifolds}.
\newblock Princeton University Press, 2009.

\bibitem{allen2012sparse}
G.~I. Allen.
\newblock {Sparse higher-order principal components analysis}.
\newblock In {\em International Conference on Machine Learning}, pages 27--36,
  April 2012.

\bibitem{boumal2022intromanifolds}
N.~Boumal.
\newblock An introduction to optimization on smooth manifolds.
\newblock To appear with Cambridge University Press, Jan 2022.

\bibitem{braman2010third}
K.~Braman.
\newblock Third-order tensors as linear operators on a space of matrices.
\newblock {\em Linear Algebra Appl.}, 433(7):1241--1253, 2010.

\bibitem{breiding2018riemannian}
P.~Breiding and N.~Vannieuwenhoven.
\newblock A {Riemannian} trust region method for the canonical tensor rank
  approximation problem.
\newblock {\em SIAM J. Optim.}, 28(3):2435--2465, 2018.

\bibitem{bunse1993numerical}
A.~Bunse-Gerstner, R.~Byers, and V.~Mehrmann.
\newblock Numerical methods for simultaneous diagonalization.
\newblock {\em SIAM J. Matrix Anal. Appl.}, 14(4):927--949, 1993.

\bibitem{chen2020proximal}
S.~Chen, S.~Ma, A.~M.-C. So, and T.~Zhang.
\newblock Proximal gradient method for nonsmooth optimization over the
  {Stiefel} manifold.
\newblock {\em SIAM J. Optim.}, 30(1):210--239, 2020.

\bibitem{cichocki2015tensor}
A.~Cichocki, D.~Mandic, L.~De~Lathauwer, G.~Zhou, Q.~Zhao, C.~Caiafa, and H.~A.
  Phan.
\newblock Tensor decompositions for signal processing applications: From
  two-way to multiway component analysis.
\newblock {\em IEEE Signal Process. Mag.}, 32(2):145--163, 2015.

\bibitem{comon2014tensors}
P.~Comon.
\newblock Tensors: a brief introduction.
\newblock {\em IEEE Signal Process. Mag.}, 31(3):44--53, 2014.

\bibitem{gao2018new}
B.~Gao, X.~Liu, X.~Chen, and Y.-X. Yuan.
\newblock A new first-order algorithmic framework for optimization problems
  with orthogonality constraints.
\newblock {\em SIAM J. Optim.}, 28(1):302--332, 2018.

\bibitem{gilman2022grassmannian}
K.~Gilman, D.~A. Tarzanagh, and L.~Balzano.
\newblock Grassmannian optimization for online tensor completion and tracking
  with the t-{SVD}.
\newblock {\em IEEE Trans. Signal Process.}, 70:2152--2167, 2022.

\bibitem{Hall2015}
B.~C. Hall.
\newblock {\em Lie groups, {L}ie Algebras, and representations}.
\newblock Springer, 2015.

\bibitem{heidel2018riemannian}
G.~Heidel and V.~Schulz.
\newblock A {Riemannian} trust-region method for low-rank tensor completion.
\newblock {\em Numer. Linear Algebra Appl.}, 25(6):e2175, 2018.

\bibitem{holtz2012manifolds}
S.~Holtz, T.~Rohwedder, and R.~Schneider.
\newblock On manifolds of tensors of fixed {TT}-rank.
\newblock {\em Numer. Math.}, 120(4):701--731, 2012.

\bibitem{hoover2018multilinear}
R.~C. Hoover, K.~Caudle, and K.~Braman.
\newblock Multilinear discriminant analysis through tensor-tensor
  eigendecomposition.
\newblock In {\em 2018 17th IEEE International Conference on Machine Learning
  and Applications (ICMLA)}, pages 578--584. IEEE, 2018.

\bibitem{hu2019structured}
J.~Hu, B.~Jiang, L.~Lin, Z.~Wen, and Y.-X. Yuan.
\newblock Structured quasi-{Newton} methods for optimization with orthogonality
  constraints.
\newblock {\em SIAM J. Sci. Comput.}, 41(4):A2239--A2269, 2019.

\bibitem{hu2020brief}
J.~Hu, X.~Liu, Z.-W. Wen, and Y.-X. Yuan.
\newblock A brief introduction to manifold optimization.
\newblock {\em J. Oper. Res. Soc. China}, 8(2):199--248, 2020.

\bibitem{huang2013optimization}
W.~Huang.
\newblock {\em Optimization algorithms on Riemannian manifolds with
  applications}.
\newblock PhD thesis, The Florida State University, 2013.

\bibitem{huang2018riemannian}
W.~Huang, P.-A. Absil, and K.~A. Gallivan.
\newblock A {Riemannian BFGS} method without differentiated retraction for
  nonconvex optimization problems.
\newblock {\em SIAM J. Optim.}, 28(1):470--495, 2018.

\bibitem{huang2021riemannian}
W.~Huang and K.~Wei.
\newblock Riemannian proximal gradient methods.
\newblock {\em Math. Program.}, https://doi.org/10.1007/s10107-021-01632-3,
  2021.

\bibitem{iannazzo2018riemannian}
B.~Iannazzo and M.~Porcelli.
\newblock The {R}iemannian {B}arzilai--{B}orwein method with nonmonotone line
  search and the matrix geometric mean computation.
\newblock {\em IMA J. Numer. Anal.}, 38(1):495--517, 2018.

\bibitem{kernfeld2015tensor}
E.~Kernfeld, M.~Kilmer, and S.~Aeron.
\newblock Tensor--tensor products with invertible linear transforms.
\newblock {\em Linear Alg. Appl.}, 485:545--570, 2015.

\bibitem{kilmer2013third}
M.~E. Kilmer, K.~Braman, N.~Hao, and R.~C. Hoover.
\newblock Third-order tensors as operators on matrices: A theoretical and
  computational framework with applications in imaging.
\newblock {\em SIAM J. Matrix Anal. Appl.}, 34(1):148--172, 2013.

\bibitem{kilmer2021tensor}
M.~E. Kilmer, L.~Horesh, H.~Avron, and E.~Newman.
\newblock Tensor-tensor algebra for optimal representation and compression of
  multiway data.
\newblock {\em Proc. Natl. Acad. Sci. U. S. A.}, 118(28):e2015851118, 2021.

\bibitem{kilmer2011factorization}
M.~E. Kilmer and C.~D. Martin.
\newblock Factorization strategies for third-order tensors.
\newblock {\em Linear Alg. Appl.}, 435(3):641--658, 2011.

\bibitem{kolda2006multilinear}
T.~G. Kolda.
\newblock Multilinear operators for higher-order decompositions.
\newblock Technical report, Citeseer, 2006.

\bibitem{kolda2010tensor}
T.~G. Kolda and B.~W. Bader.
\newblock Tensor decompositions and applications.
\newblock {\em SIAM {R}ev.}, 51:455--500, 2009.

\bibitem{kressner2014low}
D.~Kressner, M.~Steinlechner, and B.~Vandereycken.
\newblock Low-rank tensor completion by {Riemannian} optimization.
\newblock {\em BIT Numer. Math.}, 54(2):447--468, 2014.

\bibitem{lin2020tensor}
J.~Lin, T.-Z. Huang, X.-L. Zhao, T.-X. Jiang, and L.~Zhuang.
\newblock A tensor subspace representation-based method for hyperspectral image
  denoising.
\newblock {\em IEEE Tran. Geosci. Remote Sens.}, 59(9):7739--7757, 2020.

\bibitem{liu2020study}
W.-H. Liu and X.-Q. Jin.
\newblock A study on {T}-eigenvalues of third-order tensors.
\newblock {\em Linear Alg. Appl.}, 2020.

\bibitem{lu2018tproduct}
C.~Lu.
\newblock {\em Tensor-Tensor Product Toolbox}.
\newblock Carnegie Mellon University, June 2018.
\newblock \url{https://github.com/canyilu/tproduct}.

\bibitem{lu2019tensor}
C.~Lu, J.~Feng, Y.~Chen, W.~Liu, Z.~Lin, and S.~Yan.
\newblock Tensor robust principal component analysis with a new tensor nuclear
  norm.
\newblock {\em IEEE Trans. Pattern Anal. Mach. Intell.}, 42(4):925--938, 2019.

\bibitem{lund2020tensor}
K.~Lund.
\newblock The tensor t-function: A definition for functions of third-order
  tensors.
\newblock {\em Numer. Linear Algebr. Appl.}, 27(3):e2288, 2020.

\bibitem{mao2022several}
X.~Mao and Y.~Yang.
\newblock Several approximation algorithms for sparse best rank-1 approximation
  to higher-order tensors.
\newblock {\em J. Glob. Optim.}, https://doi.org/10.1007/s10898-022-01140-4,
  2022.

\bibitem{miao2020generalized}
Y.~Miao, L.~Qi, and Y.~Wei.
\newblock Generalized tensor function via the tensor singular value
  decomposition based on the {T}-product.
\newblock {\em Linear Alg. Appl.}, 590:258--303, 2020.

\bibitem{miao2021t}
Y.~Miao, L.~Qi, and Y.~Wei.
\newblock T-{Jordan} canonical form and t-{Drazin} inverse based on the
  t-product.
\newblock {\em Commun. Appl. Math. Comput. Sci.}, 3(2):201--220, 2021.

\bibitem{ozdemir2022high}
C.~Ozdemir, R.~C. Hoover, K.~Caudle, and K.~Braman.
\newblock High-order multilinear discriminant analysis via order-$n$ tensor
  eigendecomposition.
\newblock {\em arXiv preprint arXiv:2205.09191}, 2022.

\bibitem{pesquet2001joint}
J.-C. Pesquet-Popescu, B.and~Pesquet and A.~P. Petropulu.
\newblock Joint singular value decomposition-a new tool for separable
  representation of images.
\newblock In {\em ICIP}, volume~2, pages 569--572. IEEE, 2001.

\bibitem{qi2021tubal}
L.~Qi and Z.~Luo.
\newblock Tubal matrix.
\newblock {\em arXiv preprint arXiv:2105.00793}, 2021.

\bibitem{shashua2001linear}
A.~Shashua and A.~Levin.
\newblock Linear image coding for regression and classification using the
  tensor-rank principle.
\newblock In {\em CVPR}, volume~1, pages I--I. IEEE, 2001.

\bibitem{Sidiropoulos2016ten}
N.~Sidiropoulos, L.~De~Lathauwer, X.~Fu, K.~Huang, E.~Papalexakis, and
  C.~Faloutsos.
\newblock Tensor decomposition for signal processing and machine learning.
\newblock {\em IEEE Trans. Signal Process.}, 65(13):3551--3582.

\bibitem{song2020riemannian}
G.-J. Song, X.-Z. Wang, and M.~K. Ng.
\newblock Riemannian conjugate gradient descent method for third-order tensor
  completion.
\newblock {\em arXiv preprint arXiv:2011.11417}, 2020.

\bibitem{steinlechner2016riemannian}
M.~Steinlechner.
\newblock Riemannian optimization for high-dimensional tensor completion.
\newblock {\em SIAM J. Sci. Comput.}, 38(5):S461--S484, 2016.

\bibitem{tuIntroductionManifolds2011}
L.~W. Tu.
\newblock {\em An Introduction to Manifolds}.
\newblock Universitext. {Springer-Verlag New York}, second edition, 2011.

\bibitem{uschmajew2013geometry}
A.~Uschmajew and B.~Vandereycken.
\newblock The geometry of algorithms using hierarchical tensors.
\newblock {\em Linear Alg. Appl.}, 439(1):133--166, 2013.

\bibitem{van1978computing}
C.~Van~Loan.
\newblock Computing integrals involving the matrix exponential.
\newblock {\em IEEE Trans. Autom. Control}, 23(3):395--404, 1978.

\bibitem{van2000ubiquitous}
C.~F. Van~Loan.
\newblock The ubiquitous kronecker product.
\newblock {\em J. Comput. Appl. Math.}, 123(1-2):85--100, 2000.

\bibitem{tensorlab2013}
N.~Vervliet, O.~Debals, L.~Sorber, M.~Van~Barel, and L.~De~Lathauwer.
\newblock Tensorlab 3.0, Mar. 2016.
\newblock Available online.

\bibitem{wang2020sparse}
Y.~Wang, M.~Dong, and Y.~Xu.
\newblock A sparse rank-1 approximation algorithm for high-order tensors.
\newblock {\em Appl. Math. Lett.}, 102:106140, 2020.

\bibitem{xu2021t}
S.-S. Xu, T.-Z. Huang, J.~Lin, and Y.~Chen.
\newblock T-hy-demosaicing: Hyperspectral reconstruction via tensor subspace
  representation under orthogonal transformation.
\newblock {\em IEEE J. Sel. Top. Appl. Earth Obs. Remote Sens.}, 14:4842--4853,
  2021.

\bibitem{xu2022iterative}
T.~Xu, T.-Z. Huang, L.-J. Deng, and N.~Yokoya.
\newblock An iterative regularization method based on tensor subspace
  representation for hyperspectral image super-resolution.
\newblock {\em IEEE Trans. Geosci. Remote Sens.}, 2022.

\bibitem{zhang2002inequalities}
X.~Zhang, Z.~P. Yang, and C.~G. Cao.
\newblock Inequalities involving {Khatri-Rao} products of positive semidefinite
  matrices.
\newblock {\em Appl. Math. E-Notes}, 2:117--124, 2002.

\bibitem{zheng2021t}
M.-M. Zheng, Z.-H. Huang, and Y.~Wang.
\newblock T-positive semidefiniteness of third-order symmetric tensors and
  {T}-semidefinite programming.
\newblock {\em Comput. Optim. Appl.}, 78(1):239--272, 2021.

\bibitem{zhu2017riemannian}
X.~Zhu.
\newblock A {Riemannian} conjugate gradient method for optimization on the
  {Stiefel} manifold.
\newblock {\em Comput. Optim. Appl.}, 67(1):73--110, 2017.

\end{thebibliography}
\bibliographystyle{plain}

\appendix
\section{Appendix}
\subsection{Proof of Proposition \ref{prop:DFT tran}}\label{apx:1}
\begin{proof}
		Taking the conjugate transpose of both sides of the equation in item 1 of Proposition \ref{prop:bcirc_properties}, then multiplying both sides by $(F_l\otimes I_n)$, we get
		\begin{equation*}
			(F_l\otimes I_p)\operatorname{bcirc}(\mathcal{A}^{\top})  
			%	=\begin{bmatrix}
				%		(\hat{A}^{(1)})^H &  & \\ 
				%		& \ddots  & \\ 
				%		&  & (\hat{A}^{(l)})^H
				%	\end{bmatrix}
			=\Diag{(\hat{A}^{(i)})^H}
			(F_l\otimes I_n).
		\end{equation*}
		Taking the first column of block matrix on both sides of the above equation yields
		\begin{equation*}
			(F_l\otimes I_p)\operatorname{unfold}(\mathcal{A}^{\top}) = \frac{1}{\sqrt{l}}\Diag{(\hat{A}^{(i)})^H}
			\Vector{I_n},
		\end{equation*}
		which combing with Definition \ref{def:DFT} and  \eqref{remark:equality} gives $L(\MC{A}^{\top}) = \Fold{(\hat{A}^{(i)})^{H}}.$
				\end{proof}		
\subsection{Proof of Lemma \ref{lem:positive definite} }\label{apx:2}
\begin{proof}
	According to \cite[Thm. 5]{zheng2021t}, $\mathcal{I} + \mathcal{V}^{\top}\ast\mathcal{V}\in \TSPP{p}{p}{l}$ if only if $I_p+(\hat{V}^{(i)})^H\hat{V}^{(i)}, i \in [l]$ are Hermitian  positive definite.
\end{proof}
\subsection{The proof of Lemma \ref{lemma:orthogonal complement}}\label{apx:3}
\begin{proof}
	For any tensor $\MC{A}\in (\TSkew{p}{p}{l})^{\perp}$, there holds
	\begin{eqnarray*}
		\innerprod{\MC{A}-\MC{A}^{\top}}{\MC{A}-\MC{A}^{\top}} &=& \innerprod{\MC{A}-\MC{A}^{\top}}{\MC{A}} - \innerprod{\MC{A}-\MC{A}^{\top}}{\MC{A}^{\top}}\\
		&=& \innerprod{\MC{A}-\MC{A}^{\top}}{\MC{A}} - \innerprod{\MC{A}^{\top}-\MC{A}}{\MC{A}}=0,
	\end{eqnarray*}
	where the last equation follows from $\MC{A}-\MC{A}^{\top}\in\TSkew{p}{p}{l}$.
	Thus $\MC{A}\in\TS{p}{p}{l}$.
\end{proof}
\subsection{Proof of Proposition \ref{prop:trace_inner_t_prod}}\label{apx:4}
% \begin{proposition}\cite{zheng2021t}
% 	\label{prop: inner_tensor_bcirc}
% 	For any $\mathcal A,\mathcal B\in \T{n}{p}{l}$, 
% 	$    \innerprod{\bcirc{\mathcal A}}{\bcirc{\mathcal B}} = l \innerprod{\mathcal A}{\mathcal B}$. 
% \end{proposition}
\begin{proof}
	We denote $\mathcal D:= \mathcal A^{\top} \in \T{p}{n}{l}$ and $\mathcal C := \mathcal D\tprod\mathcal B \in \T{p}{p}{l}$. Then according to Definition \ref{def:trace_tprod} and by the definition of $\hat C$ in item 1 of Proposition \ref{prop:bcirc_properties}, we get
	$$\trace{\mathcal C} = \sum^l_{i=1}\nolimits\trace{\hat C_i} = \trace{\hat C}.$$
	Using item 1 of Proposition \ref{prop:bcirc_properties} again, we have
	\begin{align*}
		\trace{\hat C} &= \trace{ \hat D \hat B    } \\
		&= \trace{  \fourieri{p} \bcirc{\mathcal D} \fourieriH{n} \fourieri{n}  \bcirc{\mathcal B} \fourieriH{p}   }\\
		&= \trace{\bcirc{\mathcal A^{\top} }\bcirc{\mathcal B}} = \trace{\bcirc{\mathcal A}^{\top}\bcirc{\mathcal B}} \\
		&= \innerprod{\bcirc{\mathcal A}}{\bcirc{\mathcal B}} \\
		&= l\innerprod{\mathcal A}{\mathcal B},
	\end{align*} 
	where the third line uses \cite[Lem. 3]{miao2021t} and the last equation comes from \cite[Rmk. 9]{zheng2021t}. 
\end{proof}

\subsection{Proof of Theorem \ref{th:t_qr}}\label{apx:5}
% \color{red}
\begin{proof}
	t-QR was proposed in \cite[Sect. 5]{kilmer2013third}. Similar to t-SVD,
	to compute t-QR, we would only need to compute individual matrix QR's for
	about half the frontal slices of $\hat{\MC{A}}$ and the remaining part is obtained by the conjugate symmetry of the Fourier transform. Specifically, for $i= 1, \cdots, \lceil \frac{l+1}{2}\rceil,$ let $	\hat{A}^{(i)} = \hat{Q}^{(i)}\cdot \hat{R}^{(i)}$ be the QR decomposition of $\hat{A}^{(i)}\in \C{n}{p}$ \footnote{For QR factorization of complex matrices, we can choose that $R$ factor is upper triangular with real nonzero diagonal elements.} where $\hat{Q}^{(i)}\in \C{n}{p}, (\hat{Q}^{(i)})^H\cdot\hat{U}^{(i)} = I_p$, $\hat{R}^{(i)}\in \C{p}{p}_{upp}$ and $\operatorname{diag}(\hat{R}^{(i)})\in \M{p}{p}$, namely, the diagonal entries of $\hat R^{(i)}$ are real.  For $i=1+ \lceil  \frac{l+1}{2}  \rceil,\ldots,l$, 
	$
	\hat A^{(i)} = \conj{\hat A^{(l+2-i)}}, \hat Q^{(i)} = \conj{\hat U^{(l+2-i)}}, \hat R^{(i)} = \conj{\hat R^{(l+2-i)}}.
	$ It follows from Remark \ref{remark:Linear relationship} and Remark \ref{remark:DFT equality} that $\mathcal Q\in\T{n}{p}{l} \in \st{n,p,l}$ and $\mathcal R\in\TU{p}{p}{l}$. Here $\mathcal R$ to be real is because of Remark \ref{remark:Linear relationship} and direct computation.  
	Using Remark \ref{remark:Linear relationship} again, further we  have $\hat{A}^{(1)}\in \M{n}{p}, \hat{Q}^{(1)}\in \M{n}{p}, \hat{R}^{(1)}\in \M{p}{p}$.
	
We then show the uniqueness of the decomposition.	As we know, for QR decomposition of a matrix $\hat{A}^{(i)}\in \C{n}{p}$ with $n\geq p$, if $\hat{A}^{(i)}$ is of full rank $p$, then the QR decomposition $\hat{A}^{(i)} = \hat{Q}^{(i)}\hat{R}^{(i)}$ is unique if we require that the diagonal entries of $\hat{R}^{(i)}$ are all positive, i.e.,
$\hat{\MC{R}}\in\TCUP{p}{p}{l}$. 
	Since the Fourier transform is bijective, the uniqueness of the matrix QR decomposition leads to the uniqueness of the  t-QR decomposition.
\end{proof} \color{black}
\subsection{Proof of Lemma \ref{lemma:upp_dim}}\label{apx:6}
\begin{proof}
	 Theorem \ref{th:t_qr} shows that $L^{-1}\bigxiaokuohao{\TCUP{p}{p}{l}}$	is isomorphic to 
	$$\TCUP{p}{p}{l} = \bigdakuohao{\hat{\MC{R}}\bigg|\hat{R}^{(1)}\in \M{p}{p}_{upp+}, \hat{R}^{(i)}\in \C{p}{p}_{upp+},\operatorname{diag}(\hat{R}^{(i)})\in \M{p}{p}, 	i = [l]\setminus \{1\},  \hat{R}^{(i)} = \operatorname{conj}(\hat{R}^{(l+2-i)}), 
		i= 2, \cdots, \lceil \frac{l+1}{2}\rceil
	}.$$

	If $l$ is even, then it holds that
	$$\TCUP{p}{p}{l}= \bigdakuohao{\hat{\MC{R}}\bigg|\hat{R}^{(1)},\hat{R}^{(\frac{l}{2}+1)}\in \M{p}{p}_{upp+},  \hat{R}^{(i)}\in \C{p}{p}_{upp+}, \operatorname{diag}(\hat{R}^{(i)})\in \M{p}{p}, 	i = [l]\setminus \{1\}, \hat{R}^{(i)} = \operatorname{conj}(\hat{R}^{(l+2-i)}),
		i = 2,\cdots,\frac{l}{2}
	}.
	$$ 
	There are two real upper triangular $p\times p$ matrices,   whose dimensions are  $\frac{(1+p)p}{2}$;  there are $\frac{l-2}{2}$ pairs of complex upper triangular $p\times p$ matrices with positive diagonal elements, whose dimensions are  $\frac{(p-1)p}{2}\times 2 + p$.
	Hence the dimension of $\TCUP{p}{p}{l}$ is 
	$2\times \frac{(1+p)p}{2}+\frac{l-2}{2}\times\big(\frac{(p-1)p}{2}\times 2 + p\big)=\frac{p^2l}{2}+p$.
	
	If $l$ is odd, then it holds that
	$$\TCUP{p}{p}{l} = \bigdakuohao{\hat{\MC{R}}\bigg|\hat{R}^{(1)}\in \M{p}{p}_{upp+}, \hat{R}^{(i)}\in \C{p}{p}_{upp+}, \operatorname{diag}(\hat{R}^{(i)})\in \M{p}{p}, 	i = [l]\setminus \{1\}, \hat{R}^{(i)} = \operatorname{conj}(\hat{R}^{(l+2-i)}), 
		i = 2,\cdots,\frac{l+1}{2}
	}.$$ 
	There is one real upper triangular $p\times p$ matrix, whose dimensions is $\frac{(1+p)p}{2}$; there are $\frac{l-1}{2}$ pairs of complex upper triangular $p\times p$ matrices with positive diagonal elements, whose dimensions are  $\frac{(p-1)p}{2}\times 2 + p$.
	Hence the dimension of $\TCUP{p}{p}{l}$ is 
	$ \frac{(1+p)p}{2}+\frac{l-1}{2}\times\big(\frac{(p-1)p}{2}\times 2 + p\big)=\frac{p^2l+p}{2}$.
\end{proof}
\subsection{Proof of Theorem \ref{th:t_pd}}\label{apx:7}
\begin{proof}
	Let the compact t-SVD of $\mathcal A= \mathcal U\tprod\mathcal S\tprod\mathcal V^{\top}$. Let $\mathcal P :=\mathcal U\tprod\mathcal V^{\top}$ and $\mathcal H := \mathcal V\tprod\mathcal S\tprod\mathcal V^{\top}$. %Then $\mathcal P^{\top}\tprod\mathcal P = \mathcal I$.
	Then it is clear that \eqref{eq:t_pd} is satisfied. To see that $\mathcal H\in \TSP{p}{p}{l}$, first we show that $\mathcal S\in\TSP{p}{p}{l}$. This is obvious, as each $\hat S^{(i)}$ is diagonal with nonnegative entries, and so $\mathcal S\in\TSP{p}{p}{l}$, according to Remark \ref{rmk:t_psd}. By \cite[Thm. 7]{zheng2021t}, there is a unique $\mathcal T$ such that $\mathcal T\tprod\mathcal T^{\top}= \mathcal S$. Then $\mathcal H $ can be written as $\mathcal H = \mathcal V\tprod\mathcal T \tprod\bigxiaokuohao{\mathcal V\tprod\mathcal T}^{\top}$, which together with \cite[Thm. 8]{zheng2021t} shows that $\mathcal H \in\TSP{p}{p}{l}$. 
	
	To show the uniqueness of $\mathcal H$, note that $  \mathcal A^{\top}\tprod\mathcal A=\mathcal H\tprod\mathcal H$, which by \cite[Thm. 8]{zheng2021t} is clearly symmetric positive semidefinite. Revoking again \cite[Thm. 7]{zheng2021t} gives the uniqueness of $\mathcal H$.  
	
	If $\mathcal A^{\top}\tprod\mathcal A\in \TSPP{p}{p}{l}$, \cite[Thm. 8]{zheng2021t} shows that $\mathcal H$ is nonsingular (invertable, Def. \ref{def:inverse}), and so $\mathcal P = \mathcal A\tprod\mathcal H^{-1}$, which is unique. 
\end{proof}
\begin{remark}
	The proof of Theorem \ref{th:t_pd} gives the way to obtain t-PD from the compact t-SVD. This is analogous to the matrix case.  
\end{remark}
\subsection{Proof of Proposition \ref{prop:t_pd_for_retraction} }\label{apx:8}
\begin{proof}
	This can be easily derived from the proof of Theorem \ref{th:t_pd}. Here the root of a symmetric positive definite tensor was defined in \cite[Thm. 7]{zheng2021t}.
\end{proof}
\subsection{Proof of Proposition \ref{prop:sym equal}}\label{apx:ex9}
\begin{proof}
If $\hat{\mathcal{A}}\in \mathbb{C}_*^{n\times p \times l}$, then $(\hat{A}^{(i)})^H\hat{A}^{(i)}, i \in [l]$ are Hermitian  positive definite. Note that \cite[Thm. 5]{zheng2021t} shows that $(\hat{A}^{(i)})^H\hat{A}^{(i)}, i \in [l]$  are Hermitian  positive definite if only if $ \mathcal{A}^{\top}\ast\mathcal{A}\in \TSPP{p}{p}{l}$.
\end{proof}
\subsection{Proof of Theorem \ref{th:max_sol_related_pd} }\label{apx:9}
\begin{proof}
	Let $\mathcal D:=\mathcal U^{\top}\in\T{p}{n}{l}$. Then for any $\mathcal P\in\st{n,p,l}$,
	\begin{align*}
		l	\innerprod{\mathcal A}{\mathcal P} &= \trace{\mathcal A^{\top}\tprod\mathcal P}\\
		&= \trace{\mathcal V\tprod\mathcal S\tprod\mathcal U^{\top}\tprod\mathcal P}\\
		&= \trace{\mathcal S\tprod\mathcal D\tprod\mathcal P\tprod\mathcal V}\\
		&= \trace{\hat S \hat D\hat P\hat V  } \\
		&= \sum^l_{i=1} \trace{ \hat  S^{(i)} \hat  D^{(i)} \hat  P^{(i)} \hat  V^{(i)}} = \sum^l_{i=1}\trace{\hat  S^{(i)} \hat  W^{(i)}} ,
	\end{align*}
	where we let $\hat  W^{(i)}:= \hat  D^{(i)}\hat  P^{(i)}\hat  V^{(i)} \in \mathbb C^{p\times p}$.
	Note that $\hat  D^{(i)} (\hat  D^{(i)})^H = I$, $(\hat  P^{(i)})^H \hat  P^{(i)} = I$, $(\hat  V^{(i)})^H\hat  V^{(i)} = I$. Thus $    | (\hat  W^{(i)})_{jj}  | \leq 1  $, $i \in [l]$, $j \in [p]$. Therefore,
	\begin{align*}
		\sum^l_{i=1}\trace{\hat  S^{(i)} \hat W^{(i)}}  = \sum^l_{i=1}\sum^p_{j=1} (\hat S^{(i)})_{jj} (\hat W^{(i)})_{jj}  \leq \sum^l_{i=1}\sum^p_{j=1}(\hat S^{(i)})_{jj} | \hat W^{(i)}_{jj}| \leq \sum^l_{i=1}\trace{\hat S^{(i)}} = \trace{\hat S},
	\end{align*}
	where  $\hat S^{(i)}\geq 0$. 
	On the other hand, take $\mathcal P:=\mathcal U\tprod\mathcal V^{\top}$. It is easy to see that
	\[
	l\innerprod{\mathcal A}{\mathcal P} = \trace{\hat S},
	\]
	namely, the upper bound is tight, which is achieved when $\mathcal P = \mathcal U\tprod\mathcal V^{\top}$. This gives the desired result. 
\end{proof}
\subsection{Proof of  the well-defined property of \eqref{eq:exp well defined}}\label{apx:10}
% \begin{remark}\label{remark:equality}
% To be convenient,   we will use the notation $    \Delta $  as the frontal-slice-wise product (cf. \cite[Definition 2.1]{kernfeld2015tensor}) between two tensors in the Fourier domain, i.e., if
% 	$\hat C^{(i)} = \hat A^{(i)}\hat B^{(i)}, i \in [l],$ then it holds that
% 	$L(\MC{A})\Delta L(\MC{B}) = \Fold{\hat{A}^{(i)}\hat{B}^{(i)}}$; in other words, 
% 	\begin{equation}\label{eq:tri product}
% 		L(\MC{C}) 
% %		= L(\mathcal A\tprod\mathcal B)
% 		= L(\MC{A})\Delta L(\MC{B})
% 		 \Leftrightarrow \hat C^{(i)} = \hat  A^{(i)}\hat B^{(i)}, i \in [l] 
% 		\Leftrightarrow
% % 		\hat C = \hat A\hat B
% % 		 \Leftrightarrow 
% 		 \MC{C} = \mathcal A\tprod\mathcal B.
% 	\end{equation} 
% \end{remark}
\begin{proof}
To be convenient,   we will use the notation $    \Delta $  as the frontal-slice-wise product (cf. \cite[Def. 2.1]{kernfeld2015tensor}) between two tensors in the Fourier domain, i.e., if
	$\hat C^{(i)} = \hat A^{(i)}\hat B^{(i)}, i \in [l],$ then it holds that
	$L(\MC{A})\Delta L(\MC{B}) = \Fold{\hat{A}^{(i)}\hat{B}^{(i)}}$; in other words, 
	\begin{equation}\label{remark:equality}
% 	\label{eq:tri product}
		L(\MC{C}) 
%		= L(\mathcal A\tprod\mathcal B)
		= L(\MC{A})\Delta L(\MC{B})
		 \Leftrightarrow \hat C^{(i)} = \hat  A^{(i)}\hat B^{(i)}, i \in [l] 
		\Leftrightarrow
% 		\hat C = \hat A\hat B
% 		 \Leftrightarrow 
		 \MC{C} = \mathcal A\tprod\mathcal B.
	\end{equation} 
Using this notation, we have
	%\[\mathcal{A}^k = L^{-1}(L(\mathcal{A})\Delta \cdots \Delta L(\mathcal{A})) = L^{-1}(\mathrm{fold}(
	%\begin{pmatrix}
	%	(\hat{{A}}^{(1)})^k \\
	%	\vdots \\
	%	(\hat{{A}}^{(l)})^k
	%\end{pmatrix})),\]
	\[\mathcal{A}^k = L^{-1}\bigxiaokuohao{
		L(\mathcal{A})\Delta \cdots \Delta L(\mathcal{A})}
	= L^{-1}\bigxiaokuohao{
		%\mathrm{fold}(
		\Fold{(\hat{{A}}^{(i)})^k}
		%\begin{pmatrix}
		%	(\hat{{A}}^{(1)})^k \\
		%	\vdots \\
		%	(\hat{{A}}^{(l)})^k
		%\end{pmatrix}
	}.\]
% \color{red}	where \(\hat{\mathcal{A}} = L(A)\) is the third-order tensor obtained by applying (non-normalized) DFT 
% 	to the mode-\(3\) vectors of \(\mathcal{A}\), and \(\hat{{A}}^{(1)}, \cdots, \hat{{A}}^{(l)}\) are its frontal slices. \color{black}
	Thus for any \(N\),
	$$\sum_{k=0}^{N} \frac{1}{k!} \mathcal{A}^k=\sum_{k=0}^N  \frac{1}{k!} L^{-1}\bigxiaokuohao{
		%\mathrm{fold}(
		\Fold{(\hat{{A}}^{(i)})^k}
		%\begin{pmatrix}
		%	(\hat{{A}}^{(1)})^k \\
		%	\vdots \\
		%	(\hat{{A}}^{(l)})^k
		%\end{pmatrix}
	}=L^{-1}\bigxiaokuohao{
		%\mathrm{fold}(
		\Fold{\sum_{k=0}^N \frac{1}{k!} (\hat{{A}}^{(i)})^k}
		%\begin{pmatrix}
		%\sum_{k=0}^N \frac{1}{k!} (\hat{{A}}^{(1)})^k \\
		%\vdots \\
		%\sum_{k=0}^N \frac{1}{k!} (\hat{{A}}^{(l)})^k
		%\end{pmatrix}
	}.$$
	%\[
	%\begin{aligned}
	%	&\sum_{k=0}^{N} \frac{1}{k!} \mathcal{A}^k \\
	%	=&\sum_{k=0}^N  \frac{1}{k!} L^{-1}(\mathrm{fold}(
	%	\begin{pmatrix}
		%		(\hat{{A}}^{(1)})^k \\
		%		\vdots \\
		%		(\hat{{A}}^{(l)})^k
		%	\end{pmatrix}))\\
	%	=&L^{-1}(
	%	\mathrm{fold}
	%	\begin{pmatrix}
		%		\sum_{k=0}^N \frac{1}{k!} (\hat{{A}}^{(1)})^k \\
		%		\vdots \\
		%		\sum_{k=0}^N \frac{1}{k!} (\hat{{A}}^{(l)})^k
		%	\end{pmatrix}
	%	).
	%\end{aligned}  
	%\]
	Let \(N \to \infty\), it holds that
	\begin{equation}
		\ex{\mathcal{A}} = L^{-1}\bigxiaokuohao{
			\Fold{\ex{\hat{{A}}^{(i)}}}} 
		= L^{-1}\bigxiaokuohao{
			\Fold{\ex{(L(\MC{A}))^{(i)}}}} 
		%\mathrm{fold}(
		%\begin{pmatrix}
		%	e^{\hat{{A}}^{(1)}}\\
		%	\vdots \\
		%	e^{\hat{{A}}^{(l)}}
		%\end{pmatrix})
		,\end{equation}
	since the series defining the matrix exponential is convergent \cite[Prop. 2.1]{Hall2015}.
\end{proof}
\subsection{Proof of equivalence of \eqref{eq:t-exp} and \eqref{eq:t-exp2}}\label{apx:11}
\begin{proof}
	Using \eqref{eq:t-exp2} and item 1 of Proposition \ref{prop:bcirc_properties}, we have 
\begin{eqnarray*}
		 \ex{\mathcal{A}}
		&=& \mathrm{fold}\bigxiaokuohao{
			\ex{ \mathrm{bcirc}(\mathcal{A})}\mathrm{unfold}(\mathcal{I})
		}\\
		&=& \mathrm{fold}\bigxiaokuohao{
			\ex{(F^H_l \otimes I_n)
				\Diag{\hat{{A}}^{(i)}}
				%	\begin{p$  $matrix}
					%		\hat{{A}}^{(1)} & & &\\
					%		&\hat{{A}}^{(2)} & &\\
					%		& & \ddots & \\
					%		& & & \hat{{A}}^{(l)}
					%	\end{pmatrix}
				(F_{l} \otimes I_n)}\mathrm{unfold}(\mathcal{I})
		}\\
		&=& \mathrm{fold}\bigxiaokuohao{
			(F^H_l \otimes I_n) \ex{
				\Diag{	\hat{{A}}^{(i)}}
				%	\begin{pmatrix}
					%		\hat{{A}}^{(1)} & & &\\
					%		&\hat{{A}}^{(2)} & &\\
					%		& & \ddots & \\
					%		& & & \hat{{A}}^{(l)}
					%	\end{pmatrix}
			}(F_{l} \otimes I_n)\mathrm{unfold}(\mathcal{I})
		}\\
		&=& \mathrm{fold}\bigxiaokuohao{
			(F^H_l \otimes I_n) \ex{
				\Diag{\hat{{A}}^{(i)}}}
			%	 [\begin{pmatrix}
				%		\hat{{A}}^{(1)} & & &\\
				%		&\hat{{A}}^{(2)} & &\\
				%		& & \ddots & \\
				%		& & & \hat{{A}}^{(l)}
				%	\end{pmatrix}]
			\frac{1}{\sqrt{l}}
			\Vector{I_n}
			%	\begin{pmatrix}
				%		I_n \\
				%		\vdots \\
				%		I_n
				%	\end{pmatrix})
		}
		\\
		&=& \mathrm{fold}\bigxiaokuohao{
			\frac{1}{\sqrt{l}}
			(F^H_l \otimes I_n)
			\Diag{\ex{\hat{{A}}^{(i)}}}
			%	\begin{pmatrix}
				%		\ex [\hat{{A}}^{(1)}] & & &\\
				%		&\ex [\hat{{A}}^{(2)}] & &\\
				%		& & \ddots & \\
				%		& & & \ex [\hat{{A}}^{(l)}]
				%	\end{pmatrix}
			\Vector{I_n}
			%	\begin{pmatrix}
				%		I_n \\
				%		\vdots \\
				%		I_n
				%	\end{pmatrix}
		}\\  
		&=& \mathrm{fold}\bigxiaokuohao{
			\frac{1}{\sqrt{l}}
			(F^H_l \otimes I_n)
			\Vector{\ex{\hat{{A}}^{(i)}}}
			%	\begin{pmatrix}
				%		\ex [\hat{{A}}^{(1)}] \\
				%		\vdots \\
				%		\ex [\hat{{A}}^{(l)}]
				%	\end{pmatrix})
		}
		\\
		&=& 
		%		\Fold{\ex{{{A}}^{(i)}}}\\
		L^{-1}\bigxiaokuohao{
			\Fold{\ex{\hat{{A}}^{(i)}}}
			%	\begin{pmatrix}
				%		\ex [\hat{{A}}^{(1)}] \\
				%		\vdots \\
				%		\ex [\hat{{A}}^{(l)}]
				%	\end{pmatrix}
		},
\end{eqnarray*}
	where 
	%the second equality follows from item 2 of Proposition \ref{prop:bcirc_properties} and 
	the third equality is due to the following property of the matrix exponential (\cite[Prop. 2.3,\,\,6]{Hall2015}): If $X^{\top}X=I$, then
	$\ex{XAX^{\top}} = X\ex{A}X^{\top},$
	and the fifth equality comes from the following formula which follows immediately from definition:
	$ \ex{
		\Diag{D_{i}}
		%\begin{pmatrix}
		%	D_{1} & & \\
		%	& \ddots & \\
		%	& & D_l
		%\end{pmatrix}
	} = 
	\Diag{\ex{D_i}}
	%\begin{pmatrix}
	%	\ex [D_1] & & \\
	%	& \ddots & \\
	%	& & \ex[D_l]
	%\end{pmatrix}.
	$,
	and the last equality comes from \eqref{eq:IDFT} and \eqref{eq:t-exp}.
	and the fact that $\ex{A^{(i)}} = (\ex{\MC{A}})^{(i)}$ gives the penultimate equation.
\end{proof}
\subsection{Proof of Proposition \ref{prop:exp smooth} }\label{apx:12}
\begin{proof}Since  the  t-exponential mapping
	\[
	\begin{aligned}
		& \ex{\mathcal{A}}
		=&L^{-1}\bigxiaokuohao{
			\Fold{\ex{L(\mathcal{A})^{(i)}}}
			%		\mathrm{fold}(
			%		\begin{pmatrix}
				%			\ex [L(\mathcal{A})^{(1)}]\\
				%			\vdots\\
				%			\ex [L(\mathcal{A})^{(l)}]
				%		\end{pmatrix})
		}
	\end{aligned}  
	\]
	is the composite of the matrix exponential mapping and linear mappings and the matrix exponential is smooth (\cite[Prop. 2.16]{Hall2015}), we conclude that
	the t-exponential mapping is smooth.
\end{proof}
\subsection{The proof of Proposition \ref{prop:exp der} }\label{apx:13}
\begin{proof}
	Using the corresponding property of the matrix exponential \cite[Prop. 2.4]{Hall2015}, we obtain
\begin{eqnarray*}
			\frac{\mathrm{d}}{\mathrm{d}t}\ex{t\mathcal{A}}
			&=&\frac{\mathrm{d}}{\mathrm{d}t} L^{-1}\bigxiaokuohao{
				\Fold{\ex{t\hat{{A}}^{(i)}}}
				%		\mathrm{fold}(\begin{pmatrix}
					%			e^{t\hat{{A}}^{(1)}}\\
					%			\vdots \\
					%			e^{t\hat{{A}}^{(l)}}
					%		\end{pmatrix})
			}\\
			&=&L^{-1}\bigxiaokuohao{
				\Fold{\frac{\mathrm{d}}{\mathrm{d}t}\ex{t\hat{{A}}^{(i)}}}
				%		\mathrm{fold}(\begin{pmatrix}
					%			\frac{d}{dt}e^{t\hat{{A}}^{(1)}}\\
					%			\vdots \\
					%			\frac{d}{dt}e^{t\hat{{A}}^{(l)}}
					%		\end{pmatrix})
			}\\
			&=& L^{-1}\bigxiaokuohao{
				\Fold{\ex{t\hat{{A}}^{(1)}}\hat{{A}}^{(i)}}
				%		\mathrm{fold}(\begin{pmatrix}
					%			e^{t\hat{{A}}^{(1)}}\hat{{A}}^{(1)} \\
					%			\vdots\\
					%			e^{t\hat{{A}}^{(1)}}\hat{{A}}^{(l)}
					%		\end{pmatrix})
			}\\
			&=&L^{-1}\bigxiaokuohao{L(\ex{t\mathcal{A}})\Delta L(\mathcal{A})}
			= \ex{t\mathcal{A}}*\mathcal{A}.
\end{eqnarray*}
		where the first equality comes from \eqref{eq:t-exp}, while  \eqref{remark:equality} gives the last two equality.
		Similarly we can show that \(\frac{\mathrm{d}}{\mathrm{d}t} \ex{t\mathcal{A}}=\mathcal{A}*\ex{t\mathcal{A}}\).
	\end{proof}
\subsection{Proof of Proposition \ref{prop:exp decomp} }\label{apx:14}
\begin{proof}
	Applying the corresponding property in the matrix case \cite[Prop. 2.3,\,\,6]{Hall2015} and \eqref{remark:equality}, it follows that
\begin{eqnarray*}
		\ex{\MC{X}*\mathcal{A}*\MC{X}^{\top}}
		&=&L^{-1}\bigxiaokuohao{
			\Fold{\ex{\bigxiaokuohao{L(\MC{X}*\mathcal{A}*\MC{X}^{\top})}^{(i)}}}
			%		\mathrm{fold}(
			%		\begin{pmatrix}
				%			\ex[L(\MC{X}*\mathcal{A}*\MC{X}^T)^{(1)}]\\
				%			\vdots\\
				%			\ex[L(\MC{X}*\mathcal{A}*\MC{X}^T)^{(l)}]
				%		\end{pmatrix})
		}\\
		&=&L^{-1}\bigxiaokuohao{
			\Fold{	\ex{(L(\MC{X})\Delta L(\mathcal{A}) \Delta L(\MC{X}^{\top}))^{(i)}}}
			%		\mathrm{fold}(
			%		\begin{pmatrix}
				%			\ex[(L(\MC{X})\Delta L(\mathcal{A}) \Delta L(\MC{X}^{\top}))^{(1)}]\\
				%			\vdots\\
				%			\ex[(L(\MC{X})\Delta L(\mathcal{A}) \Delta L(\MC{X}^{\top}))^{(l)}]
				%		\end{pmatrix})
		}\\
		&=&L^{-1}\bigxiaokuohao{
			\Fold{\ex{\hat{\MC{X}}^{(i)}\hat{{A}}^{(i)}\hat{\MC{X}}^{(i)H}}}
			%		\mathrm{fold}(\begin{pmatrix}
				%			\ex[\hat{\MC{X}}^{(1)}\hat{{A}}^{(1)}\hat{\MC{X}}^{(1)H}]\\
				%			\vdots\\
				%			\ex[\hat{\MC{X}}^{(l)}\hat{{A}}^{(l)}\hat{\MC{X}}^{(l)H}]
				%		\end{pmatrix})
		}\\
		&=&L^{-1}\bigxiaokuohao{
			\Fold{\hat{\MC{X}}^{(i)}\ex{\hat{{A}}^{(i)}}\hat{\MC{X}}^{(i)H}}
			%		\mathrm{fold}(\begin{pmatrix}
				%			\hat{\MC{X}}^{(1)}\ex[\hat{{A}}^{(1)}]\hat{\MC{X}}^{(1)H}\\
				%			\vdots\\
				%			\hat{\MC{X}}^{(l)}\ex[\hat{{A}}^{(l)}]\hat{\MC{X}}^{(l)H}
				%		\end{pmatrix})
		}\\
		&=&L^{-1}\bigxiaokuohao{
			L(\MC{X})\Delta L(\ex{\mathcal{A}})\Delta L(\MC{X}^H)}
		=\MC{X}*\ex{\mathcal{A}}*\MC{X}^{\top},
\end{eqnarray*}
	where the first equality comes from \eqref{eq:t-exp}.
	%	 The stament there is given for an invertible matrix. 
	%	The argument in the proof clearly also works for any partially orthogonal (or unitary) matrix.
\end{proof}
\subsection{Proof of Proposition \ref{prop:exp Diag tensor}}\label{apx:15}
\begin{proof}
	Denotes $\MC{A} = \operatorname{Diag}\left(\mathcal{D}_j: j \in[p]\right)$ and $\MC{B}=	\operatorname{Diag}\left(\ex{\mathcal{D}_j}: j \in[p]\right) $.  Applying \eqref{eq:t-exp}, we get
\begin{eqnarray*}
		\ex{\MC{A}}
		%		&\ex{
			%			\operatorname{Diag}\left(\mathcal{D}_j: j \in[p]\right) 
			%%		\begin{pmatrix}
				%%			\mathcal{D}_1 & & &\\
				%%			& \mathcal{D}_2 & &\\
				%%			& & \ddots & \\
				%%			& & & \mathcal{D}_m
				%%		\end{pmatrix}
			%	}\\
		&=&L^{-1}\bigxiaokuohao{
			\Fold{\ex{
					\hat{A}^{(i)}
				}
			}
		}\\
		&=&L^{-1}\bigxiaokuohao{
			\Fold{\ex{
					\operatorname{Diag}\left(\hat{{D}_j}^{(i)}: j \in[p]\right) 
					%			\begin{pmatrix}
						%				\hat{\mathcal{D}_1}^{(i)}&&\\
						%				&\ddots&\\
						%				&&\hat{\mathcal{D}_m}^{(i)}
						%			\end{pmatrix}
			}}
			%		\mathrm{fold}(
			%		\begin{pmatrix}
				%			\ex[\begin{pmatrix}
					%				\hat{\mathcal{D}_1}^{(1)}&&\\
					%				&\ddots&\\
					%				&&\hat{\mathcal{D}_m}^{(1)}
					%			\end{pmatrix}]\\
				%			\vdots\\
				%			\ex[\begin{pmatrix}
					%				\hat{\mathcal{D}_1}^{(l)}&&\\
					%				&\ddots&\\
					%				&&\hat{\mathcal{D}_m}^{(l)}
					%			\end{pmatrix}]
				%		\end{pmatrix})
		}\\
		&=&L^{-1}\bigxiaokuohao{
			\Fold{
				\operatorname{Diag}\left(\ex{\hat{{D}_j}^{(i)}}: j \in[p]\right) 
				%			\begin{pmatrix}
					%					\ex[\hat{\mathcal{D}_1}^{(i)}]&&\\
					%					&\ddots&\\
					%					&&\ex[\hat{\mathcal{D}_m}^{(i)}]
					%			\end{pmatrix}
			}
			%		\mathrm{fold}(
			%		\begin{pmatrix}
				%			\begin{pmatrix}
					%				\ex[\hat{\mathcal{D}_1}^{(1)}]&&\\
					%				&\ddots&\\
					%				&&\ex[\hat{\mathcal{D}_m}^{(1)}]
					%			\end{pmatrix}\\
				%			\vdots\\
				%			\begin{pmatrix}
					%				\ex[\hat{\mathcal{D}_1}^{(l)}]&&\\
					%				&\ddots&\\
					%				&&\ex[\hat{\mathcal{D}_m}^{(l)}]
					%			\end{pmatrix}
				%		\end{pmatrix})
		}\\
		&=&L^{-1}\bigxiaokuohao{\Fold{\hat{B}^{(i)}}
		}
		=\MC{B}
		%		\operatorname{Diag}\left(\ex{\mathcal{D}_j}: j \in[p]\right)
		%		\begin{pmatrix}
			%			\ex[\mathcal{D}_1] & & &\\
			%			& \ex[\mathcal{D}_2] & &\\
			%			& & \ddots & \\
			%			& & & \ex[\mathcal{D}_m]
			%		\end{pmatrix}
		,
\end{eqnarray*}
	where 
	%	the frist and last equality comes from \eqref{eq:t-exp}, and 
	the third equality is due to the  property of the matrix exponential \cite{van1978computing}:
	$ \ex{
		\Diag{C_{i}}
		%\begin{pmatrix}
		%	D_{1} & & \\
		%	& \ddots & \\
		%	& & D_l
		%\end{pmatrix}
	} = 
	\Diag{\ex{C_i}}
	%\begin{pmatrix}
	%	\ex [D_1] & & \\
	%	& \ddots & \\
	%	& & \ex[D_l]
	%\end{pmatrix}.
	$.
	%	In the derivation we used the corresponding property in the matrix case. We do not have a convenient reference for this, but the property
	%	follows immediately from the definition of matrix exponential.
\end{proof}
\subsection{Proof of Proposition \ref{prop:exp tranpose} }\label{apx:16}
\begin{proof}
\begin{eqnarray*}
		(\ex{\mathcal{A}})^{\top}
		%		=&L^{-1}\bigxiaokuohao{
			%			\Fold{	\bigxiaokuohao{L(\ex{\mathcal{A}})}^{(i)H}}
			%%		\mathrm{fold}(\begin{pmatrix}
				%%			L(\ex[\mathcal{A}])^{(1)H}\\
				%%			\vdots\\
				%%			L(\ex[\mathcal{A}])^{(l)H}
				%%		\end{pmatrix})
			%	}\\
		&=&L^{-1}\bigxiaokuohao{
			\Fold{(\ex{\hat{{A}}^{(i)}})^H}
			%		\mathrm{fold}(\begin{pmatrix}
				%			(\ex[\hat{{A}}^{(1)}])^H\\
				%			\vdots\\
				%			(\ex[\hat{{A}}^{(l)}])^H
				%		\end{pmatrix})
		}\\
		&=&L^{-1}\bigxiaokuohao{
			\Fold{\ex{\hat{{A}}^{(i)H}}}
			%		\mathrm{fold}(\begin{pmatrix}
				%			\ex[\hat{{A}}^{(1)H}]\\
				%			\vdots\\
				%			\ex[\hat{{A}}^{(l)H}]
				%		\end{pmatrix})
		}
		=\ex{\mathcal{A}^{\top}},
\end{eqnarray*}
	where  Proposition \ref{prop:DFT tran} gives the first equality, while the second equality comes from  the corresponding property in the matrix case \cite[Prop. 2.3,\,\,2]{Hall2015}.
\end{proof}
\subsection{Proof of Proposition \ref{prop:exp addition}}\label{apx:17}
\begin{proof}
	Using \eqref{remark:equality}, we have
\begin{eqnarray*}
		\ex{\mathcal{A}}*\ex{\mathcal{B}}
		&=&L^{-1}(L(\ex{\mathcal{A}})\Delta L(\ex{\mathcal{B}}))\\
		&=&L^{-1}\bigxiaokuohao{
			\Fold{\ex{\hat{{A}}^{(i)}}\ex{\hat{{B}}^{(i)}}}
			%		\mathrm{fold}(
			%		\begin{pmatrix}
				%			\ex[\hat{{A}}^{(1)}]\ex[\hat{{B}}^{(1)}]\\
				%			\vdots\\
				%			\ex[\hat{{A}}^{(l)}]\ex[\hat{{B}}^{(l)}]
				%		\end{pmatrix})
		}\\
		&=&L^{-1}\bigxiaokuohao{
			\Fold{\ex{\hat{{A}}^{(i)}+\hat{{B}}^{(i)}}}
			%		\mathrm{fold}(\begin{pmatrix}
				%			\ex[\hat{{A}}^{(1)}+\hat{{B}}^{(1)}]\\
				%			\vdots\\
				%			\ex[\hat{{A}}^{(l)}+\hat{{B}}^{(l)}]
				%		\end{pmatrix})
		}
		=\ex{\mathcal{A}+\mathcal{B}},
\end{eqnarray*}
	where the second equality comes from the property in the matrix exponential \cite[Prop. 2.3,\,\,5]{Hall2015}.
\end{proof}

\subsection{Proof of Theorem \ref{th:sylvester}}\label{apx:19}
\begin{lemma}\label{prop:1}
	Let $\mathcal{A}\in \mathbb{R}^{m\times n \times l},\mathcal{B}\in \mathbb{R}^{n\times k\times l}
	%	,\mathcal{C}\in \mathbb{R}^{m\times k\times l}
	$. Then
	\begin{equation*}
		(I_{kl}\otimes [A^{(1)},\cdots,A^{(l)}])\operatorname{vec}(\widetilde{\operatorname{bcirc}}(\mathcal{B}))\\
		= ([I_k]_{l\times l}\odot \operatorname{bcirc}(\mathcal{A}) )\operatorname{vec}(\mathcal{B}).
	\end{equation*}
%	where $
%	%	\widetilde{\operatorname{bcirc}}(\mathcal{B}) = \begin{bmatrix}
%		%		B^{(1)} &  B^{(2)} & \cdots & B^{(l)}\\ 
%		%		B^{(l)} & B^{(1)}  & \cdots & B^{(l-1)} \\ 
%		%		\vdots & \vdots & \ddots    &\vdots \\ 
%		%		B^{(2)}& \cdots & B^{(l)} & B^{(1)}
%		%	\end{bmatrix}, \operatorname{bcirc}(\mathcal{B}) 
%	% =\begin{bmatrix}
%		%	B^{(1)} & B^{(l)}  & \cdots  & B^{(2)}\\ 
%		%	B^{(2)} & B^{(1)}  & \cdots & B^{(3)}\\ 
%		%	\vdots  & \ddots  & \ddots & \vdots \\ 
%		%	B^{(l)}  & \cdots & B^{(2)} & B^{(1)}
%		%\end{bmatrix}
%		%,
%		[I_k]_{l\times l}  = \begin{bmatrix}
%			I_k & \cdots  & I_k\\ 
%			\vdots  & \ddots  & \vdots \\ 
%			I_k& \cdots  & I_k
%		\end{bmatrix}_{l\times l}
%		$ and $\odot$ is the the Khatri-Rao product for partitioned matrices.
		
	\end{lemma}
	\begin{proof}
		By definition, the left hand side part is 
		\begin{equation*}
			LHS = 
			\begin{bmatrix}
				\begin{smallmatrix}
				[A^{(1)},\cdots,A^{(l)}]
				&  \ddots &  &  &  &  & \\ 
				&  & [A^{(1)},\cdots,A^{(l)}] &  &  &  & \\ 
				&  &  & \ddots &   &  & \\ 
				&  &  &  &  [A^{(1)},\cdots,A^{(l)}] &  & \\ 
				&  &  &  &  & \ddots &  \\ 
				&  &  &  &  &  &  [A^{(1)},\cdots,A^{(l)}]
	\end{smallmatrix}
			\end{bmatrix}_{kl} \cdot  \begin{bmatrix}
			\begin{smallmatrix}
				\begin{bmatrix}
					\begin{smallmatrix}
					B^{(1)}\\ 
					B^{(l)}\\ 
					\vdots \\ 
					B^{(2)}
					\end{smallmatrix}
				\end{bmatrix}_{:1}\\ 
				\vdots 
				\\ \begin{bmatrix}
					\begin{smallmatrix}
					B^{(1)}\\ 
					B^{(l)}\\ 
					\vdots \\ 
					B^{(2)}
					\end{smallmatrix}
				\end{bmatrix}_{:k}
				\\ 
				\vdots \\ 
				\begin{bmatrix}
					\begin{smallmatrix}
					B^{(l)}\\ 
					B^{(l-1)}\\ 
					\vdots \\ 
					B^{(1)}
					\end{smallmatrix}
				\end{bmatrix}_{:1}\\ 
				\vdots 
				\\ \begin{bmatrix}
					\begin{smallmatrix}
					B^{(l)}\\ 
					B^{(l-1)}\\ 
					\vdots \\ 
					B^{(1)}
					\end{smallmatrix}
				\end{bmatrix}_{:k}
	\end{smallmatrix}
			\end{bmatrix},
		\end{equation*}
	where $B^{(i)}_{:j}$ is the $j$th column of  $B^{(i)}, i\in [l]$	and the right hand side part is 
		\begin{equation*}
			RHS = 	\begin{bmatrix}
				\begin{smallmatrix}
				\begin{bmatrix}
					\begin{smallmatrix}
					A^{(1)} &  & \\ 
					& \ddots  & \\ 
					&  & A^{(1)}
					\end{smallmatrix}
				\end{bmatrix}_k & \begin{bmatrix}
				\begin{smallmatrix}
					A^{(l)} &  & \\ 
					& \ddots  & \\ 
					&  & A^{(l)}
					\end{smallmatrix}
				\end{bmatrix}_k & \cdots  & \begin{bmatrix}
				\begin{smallmatrix}
					A^{(2)} &  & \\ 
					& \ddots  & \\ 
					&  & A^{(2)}
					\end{smallmatrix}
				\end{bmatrix}_k\\ 
				\begin{bmatrix}
					\begin{smallmatrix}
					A^{(2)} &  & \\ 
					& \ddots  & \\ 
					&  & A^{(2)}
					\end{smallmatrix}
				\end{bmatrix}_k & \begin{bmatrix}
				\begin{smallmatrix}
					A^{(1)} &  & \\ 
					& \ddots  & \\ 
					&  & A^{(1)}
					\end{smallmatrix}
				\end{bmatrix}_k & \cdots  & \begin{bmatrix}
				\begin{smallmatrix}
					A^{(3)} &  & \\ 
					& \ddots  & \\ 
					&  & A^{(3)}
					\end{smallmatrix}
				\end{bmatrix}_k\\ 
				\vdots & \vdots  & \ddots  & \vdots  \\ 
				\begin{bmatrix}
					\begin{smallmatrix}
					A^{(l)} &  & \\ 
					& \ddots  & \\ 
					&  & A^{(l)}
					\end{smallmatrix}
				\end{bmatrix}_k & \begin{bmatrix}
				\begin{smallmatrix}
					A^{(l-1)} &  & \\ 
					& \ddots  & \\ 
					&  & A^{(l-1)}
					\end{smallmatrix}
				\end{bmatrix}_k & \cdots  & \begin{bmatrix}
				\begin{smallmatrix}
					A^{(1)} &  & \\ 
					& \ddots  & \\ 
					&  & A^{(1)}
					\end{smallmatrix}
				\end{bmatrix}_k
			\end{smallmatrix}
			\end{bmatrix}_l \cdot \begin{bmatrix}
			\begin{smallmatrix}
				\begin{bmatrix}
					\begin{smallmatrix}
					(B^{(1)})_{:1}\\ 
					\vdots \\ 
					(B^{(1)})_{:k}
					\end{smallmatrix}
				\end{bmatrix}\\ 
				\begin{bmatrix}
					\begin{smallmatrix}
					(B^{(2)})_{:1}\\ 
					\vdots \\ 
					(B^{(2)})_{:k}
					\end{smallmatrix}
				\end{bmatrix}\\
				\vdots \\ 
				\begin{bmatrix}
					\begin{smallmatrix}
					(B^{(l)})_{:1}\\ 
					\vdots \\ 
					(B^{(l)})_{:k}
					\end{smallmatrix}
				\end{bmatrix}\\
			\end{smallmatrix}
			\end{bmatrix}, 
		\end{equation*}
    		We observe that the $(q,1)-$th  block of partitioned matrice on LHS is 
    		\begin{equation}\label{eq:LHS}
\sum\nolimits_{i=1}^{l}A^{(i)}B^{(h_i)}_{:j}, \qquad q
		= (p-1)k+j\in[kl]
		, ~ j\in[k],~p\in[l],
			\end{equation}
		where $$h_i = \begin{cases}
			l+p+1-i,& i> p \\ 
		p+1-i,&  i\leq p
		\end{cases}.$$
		While the $(q,1)-$th block of partitioned matrice on RHS is 
	$\sum\nolimits_{i=1}^{l}A^{(h_i)}B^{(i)}_{:j},$ 
		which is equal to \eqref{eq:LHS}.
	\end{proof}
	\begin{lemma}\cite{van2000ubiquitous}\label{lem:vec}
		Let $C\in \M{m}{n},X\in \M{n}{p},B\in \M{k}{p}$. Then 
		\begin{equation*}
			Y = CXB^{\top}\Leftrightarrow \operatorname{vec}({Y}) = (B\otimes C)\operatorname{vec}({X}).
		\end{equation*}
	\end{lemma}
	\begin{lemma}\label{prop:vec}
		Let $\mathcal{A}\in \mathbb{R}^{m\times n \times l},\mathcal{B}\in \mathbb{R}^{n\times k\times l},\mathcal{C}\in \mathbb{R}^{m\times k\times l}$. Then \begin{equation*}
			\mathcal{C} = \mathcal{A}\ast  \mathcal{B}\Leftrightarrow \operatorname{vec}(\mathcal{C}) = (\widetilde{\operatorname{bcirc}}(\mathcal{B})^{\top}\otimes I_m)\operatorname{vec}(\mathcal{A})=([I_k]_{l\times l}\odot\operatorname{bcirc}(\mathcal{A}) )\operatorname{vec}(\mathcal{B}).
		\end{equation*} 
	\end{lemma}
	\begin{proof}
		We observe that $\operatorname{vec}(\mathcal{C}) = \operatorname{vec}([C^{(1)},\cdots,C^{(l)}])$.
		Since $\operatorname{unfold}(\mathcal{C}) = \operatorname{bcirc}(\mathcal{A})\operatorname{unfold}(\mathcal{B}),$ i.e., $[C^{(1)},\cdots,C^{(l)}] = [A^{(1)},\cdots,A^{(l)}]\widetilde{\operatorname{bcirc}}(\mathcal{B})$, we have
		\begin{eqnarray*}
			\operatorname{vec}(\mathcal{C}) &=& \operatorname{vec}([C^{(1)},\cdots,C^{(l)}])\\
			&=& \operatorname{vec}([A^{(1)},\cdots,A^{(l)}]\widetilde{\operatorname{bcirc}}(\mathcal{B}))\\
			&=& (\widetilde{\operatorname{bcirc}}(\mathcal{B})^{\top}\otimes I_m)\operatorname{vec}([A^{(1)},\cdots,A^{(l)}])\\
			&=& (\widetilde{\operatorname{bcirc}}(\mathcal{B})^{\top}\otimes I_m)\operatorname{vec}(\mathcal{A}),
		\end{eqnarray*}
		where the  third equation comes from Lemma \ref{lem:vec}.
		Similarly, by lemma \ref{prop:1}, there holds 
		\begin{eqnarray*}
			\operatorname{vec}(\mathcal{C}) &=& \operatorname{vec}([C^{(1)},\cdots,C^{(l)}])\\
			&=& \operatorname{vec}([A^{(1)},\cdots,A^{(l)}]\widetilde{\operatorname{bcirc}}(\mathcal{B}))\\
			&=& (I_{kl}\otimes [A^{(1)},\cdots,A^{(l)}])\operatorname{vec}(\widetilde{\operatorname{bcirc}}(\mathcal{B}))\\
			%		&=& ([I_k]_{l\times l}\odot \operatorname{bcirc}(\mathcal{A}) )\operatorname{vec}([B^{(1)},\cdots,B^{(l)}])\\
			%		&=& (I_k\otimes\widetilde{\operatorname{bcirc}}(\mathcal{A}) )vec(unfold(\mathcal{B}))\\
			%		&=& (I_k\otimes\operatorname{bcirc}(\mathcal{A}) )vec([B_1,\cdots,B_l])\\
			&=& ([I_k]_{l\times l}\odot\operatorname{bcirc}(\mathcal{A}) )\operatorname{vec}(\mathcal{B}),
		\end{eqnarray*}
		where the  third equation follows from Lemma \ref{lem:vec}.
	\end{proof}
	\begin{proof}
		Applying lemma \ref{prop:vec}, the tensor Sylvester equation \eqref{eq:sylvester} can be rewritten  in the form 
		\begin{equation}
			\operatorname{vec}(\mathcal{C}) = \bigxiaokuohao{ \widetilde{\operatorname{bcirc}}(\mathcal{B})^{\top}\otimes I_k+[I_k]_{l\times l}\odot\operatorname{bcirc}(\mathcal{A}) }\operatorname{vec}(\mathcal{X}).
		\end{equation}
	\end{proof}		
	
	\subsection{The Euclidean gradient $\operatorname{grad}f(\MC{X})$ and the Euclidean directional derivative $Df(\MC{X})[\MC{H}]$      in subsection \ref{subsection:grad Hess}}\label{apx:18}
	Similar to \cite[Def. 4]{zheng2021t}, for third-order tensor $\mathcal{X}\in \T{n}{p}{l}$, we can also introduce the definition of the Euclidean gradient $\operatorname{grad}f(\MC{X})$ and the Euclidean Hessian $\operatorname{Hess}f(\MC{X})$   from the Fréchet differentiable. 
\begin{definition}
Let $f: \MC{U} \subseteq \mathbb{R}^{n \times p \times l} \rightarrow \mathbb{R}$ be a continuous map. Then, we say $f$ is t-differentiable at $\mathcal{X} \in \MC{U}$ if and only if there exists a third-order tensor $\operatorname{grad}f(\mathcal{X})\in\T{n}{p}{l}$ such that
$$
\lim _{\mathcal{H} \rightarrow \mathcal{O}} \frac{\left\|f(\mathcal{X}+\mathcal{H})-f(\mathcal{X})-\left\langle\operatorname{grad}f(\mathcal{X}), \mathcal{H}\right\rangle\right\|_F}{\|\mathcal{H}\|_F}=0,
$$
where $\operatorname{grad}f(\mathcal{X})$ is called the gradient of $f$ at $\mathcal{X}$ and $Df(\MC{X})[\MC{H}]= \left\langle\operatorname{grad}f(\mathcal{X}), \mathcal{H}\right\rangle$  called the directional derivative of $f$ at $\mathcal{X}$ along $\mathcal{H}$.
 And we say $f$ is twice t-differentiable at $\mathcal{X} \in U$ if and only if $f$ is continuously t-differentiable and there exists a bounded linear operator $\operatorname{Hess}f(\mathcal{X}):\T{n}{p}{l}\rightarrow\T{n}{p}{l}$ such that
$$
\lim _{\mathcal{H} \rightarrow \mathcal{O}} \frac{\left\|\operatorname{grad}f(\mathcal{X}+\mathcal{H})-\operatorname{grad}f(\mathcal{X})-\operatorname{Hess}f(\mathcal{X})[\MC{H}]\right\|_F}{\|\mathcal{H}\|_F}=0.
$$
% and we call $\operatorname{Hess}f(\mathcal{X})$ the Hess of $f$ at $\mathcal{X}$ and $\operatorname{Hess}f(\mathcal{X})[\MC{H}] = \operatorname{Hess}f(\mathcal{X}) * \mathcal{H}$.
Furthermore, we say $f$ is t-differentiable 
(twice t-differentiable)
on $\MC{U}$ if and only if $f$ is t-differentiable 
(twice t-differentiable) 
at every $\mathcal{X} \in \MC{U}$.
\end{definition}

\begin{theorem}
Let $f$ be a continuous map from $\MC{U} \subseteq \mathbb{R}^{n \times p \times l}$ to $\mathbb{R}$. Then
% \begin{itemize}
%     \item [(i)]
 $f$ is t-differentiable on $U$ if and only if $\frac{\partial f(\mathcal{X})}{\partial[\operatorname{vec}(\mathcal{X})]}$ exists for every $\mathcal{X} \in \MC{U}$, where $\frac{\partial f(\mathcal{X})}{\partial[\operatorname{vec}(\mathcal{X})]}$ is a vector in $\mathbb{R}^{npl}$ with $\left(\frac{\partial f(\mathcal{X})}{\partial[\operatorname{vec}(\mathcal{X})]}\right)_{i}=\frac{\partial f(\mathcal{X})}{\partial\left([\operatorname{vec}(\mathcal{X})]_{i}\right)}$ for any $i \in[npl]$. Especially, for any $\mathcal{X} \in \MC{U},$
 \begin{equation}\label{eq:Euclidean gradient computation}
 \operatorname{grad}f(\mathcal{X})=\operatorname{vec}^{-1}\bigxiaokuohao{\frac{\partial f(\mathcal{X})}{\partial[\operatorname{vec}(\mathcal{X})]}},
 \end{equation}
 where $\mathbf{v}=\mathrm{vec}(\MC{A})$ denotes the vectorized tensor of $\MC{A}$ and $\mathrm{vec}^{-1}(\mathbf{v})=\MC{A}$ represents  the operator that converts a vector $\mathbf{v}$ back to a tensor $\MC{A}$, which can all be implemented with functions \texttt{reshape}, \texttt{permute} and \texttt{ipermute} of Matlab (cf. \cite{kolda2006multilinear}). 
%  \item [(ii)]
%   $f$ is twice T-differentiable on $\MC{U}$ if and only if is continuously T-differentiable on  $\MC{U}$ and $\frac{\partial\left[\operatorname{vec}\left(\operatorname{grad}f(\mathcal{X})\right)\right]}{\partial[\operatorname{vec}(\mathcal{X})]}$ is a block Toeplitz matrix with each block being of size $n\times n$  matrix for every $\mathcal{X} \in \MC{U}$. In particular, 
%   \begin{equation}\label{eq:Euclidean Hess computation}
% \operatorname{Hess} f(\mathcal{X})=\operatorname{btoep}^{-1}\bigxiaokuohao{\frac{\partial\left[\operatorname{vec}\left(\operatorname{grad}f(\mathcal{X})\right)\right]}{\partial[\operatorname{vec}(\mathcal{X})]}},
%   \end{equation}
%   for any $\mathcal{X} \in \MC{U}$, where $\operatorname{btoep}(\MC{L}):=([I_p]_{l\times l}\odot\operatorname{bcirc}(\mathcal{L}))$ represents the operator that transfroms $\MC{L}$ into a block Toeplitz matrix with each block being of size $p\times p$ matrix.
%  \end{itemize}
\end{theorem}
\begin{proof} The proof is similar to that of \cite[Thm. 1]{zheng2021t} and is omitted.
\end{proof}			

\end{document}